\documentclass[preprint]{imsart}
\usepackage{amssymb,latexsym,amsmath,amsbsy,amsthm,amsxtra,amsgen}
\usepackage[ansinew]{inputenc}
\usepackage{eurosym}
\usepackage{graphicx}
\usepackage{subfig}
\usepackage{caption}
\usepackage{everysel}
\usepackage{keyval}
\usepackage{ragged2e}
\usepackage{float}
\usepackage{comment}
\usepackage{nicefrac,dsfont,bbm,amstext,amsfonts,amssymb}

\usepackage{imsart}
\usepackage[OT1]{fontenc}
\usepackage{amsthm,amsmath}
\usepackage[numbers]{natbib}
\usepackage[colorlinks,citecolor=blue,urlcolor=blue]{hyperref}

 \newtheorem{theorem}{Theorem}[section]
 \newtheorem{corollary}[theorem]{Corollary}
 \newtheorem{lemma}[theorem]{Lemma}
 \newtheorem{proposition}[theorem]{Proposition}
 \newtheorem{conj}[theorem]{Conjecture}

 \theoremstyle{definition}
 \newtheorem{defn}[theorem]{Definition}
 \newtheorem{rem}[theorem]{Remark}
 \newtheorem{ex}[theorem]{Example}
 \newtheorem{ass}[theorem]{Assumption}

 \newcommand{\etav}{\boldsymbol \eta }

 \newcommand{\lambdav}{\mbox{\boldmath$\lambda$}}
 \newcommand{\epsilonv}{\mbox{\boldmath$\epsilon$}}
 \newcommand{\varepsilonv}{\mbox{\boldmath$\varepsilon$}}
 
 \newcommand{\ind}{\mathbf{1}}

 \newcommand{\R}{\mathds{R}}
 \newcommand{\Z}{\mathds{Z}}

 \newcommand{\Ln}{\mathbb{L}}
 \newcommand{\Cov}{\textmd{{Cov}}}

 \newcommand{\E}{\mathbb{E}}

 \newcommand{\N}{\mathds{N}}

 \newcommand{\F}{\mathcal{F}}
 \newcommand{\OO}{\mathcal{O}}

 \newcommand{\oo}{\mbox{\scriptsize $\mathcal{O}$}}
 \newcommand{\Gaussian}{\mathcal{N}}

 \newcommand{\A}{\mathcal{A}}
 \newcommand{\J}{J}
 \newcommand{\JJ}{J}

 \newcommand{\ad}{\mathfrak{a}}
 \newcommand{\bd}{\mathfrak{b}}
 \newcommand{\cd}{\mathfrak{c}}

 \newcommand{\dd}{\mathfrak{d}}
 \newcommand{\ld}{\mathfrak{l}}
 
 \newcommand{\kd}{\mathfrak{k}}
 
 \newcommand{\hd}{\mathfrak{h}}
 \newcommand{\pd}{\mathfrak{p}}
 \newcommand{\td}{\mathfrak{t}}
 \newcommand{\rd}{\mathfrak{r}}
 \newcommand{\mm}{{m}}

 \newcommand{\B}{\mathcal{B}}

 \newcommand{\e}{e}

\newcommand{\DD}{D}
\newcommand{\GG}{G}

\newcommand{\TT}{\mathcal{T}}

\newcommand{\Aoneb}{{\bf (G1)$^{b}$}}
\newcommand{\Atwob}{{\bf (G2)$^{b}$}}
\newcommand{\Athreeb}{{\bf (G3)$^{b}$}}

\newcommand{\Aonem}{{\bf (G1)$^{\mathfrak{m}}$}}

\newcommand{\Aonestar}{{\bf (G1)}}
\newcommand{\Atwostar}{{\bf (G2)}}
\newcommand{\Athreestar}{{\bf (G3)}}

\newcommand{\AoneD}{{\bf (D1)}}
\newcommand{\AtwoD}{{\bf (D2)}}
\newcommand{\AthreeD}{{\bf (D3)}}

\newcommand{\AoneC}{{\bf (C1)}}
\newcommand{\AtwoC}{{\bf (C2)}}

\newcommand{\AoneCh}{{\bf($\mathrm{\bf C_h1}$)}}
\newcommand{\AtwoCh}{{\bf($\mathrm{\bf C_h2}$)}}
\newcommand{\AthreeCh}{{\bf($\mathrm{\bf C_h3}$)}}

\renewcommand{\star}{}

\newcommand{\itmi}{{\bf (i)}}
\newcommand{\itmii}{{\bf (ii)}}
\newcommand{\itmiii}{{\bf (iii)}}
\newcommand{\itmiv}{{\bf (iv)}}
\newcommand{\itmv}{{\bf (v)}}
\newcommand{\itmvi}{{\bf (vi)}}
\newcommand{\itmvii}{{\bf (vii)}}
\newcommand{\itmviii}{{\bf (viii)}}

\newcommand{\Bone}{{\bf (E1)}}
\newcommand{\Btwo}{{\bf (E2)}}
\newcommand{\Bthree}{{\bf (E3)}}

\newcommand{\Cone}{{\bf (F1)}}
\newcommand{\Ctwo}{{\bf (F2)}}
\newcommand{\Cthree}{{\bf (F3)}}

\newcommand{\Sone}{{\bf (Step 1)}}
\newcommand{\Stwo}{{\bf (Step 2)}}

\newcommand{\cov}{\mathbb{C}\text{o\hspace*{0.02cm}v}}

\makeatletter
\newcommand{\opnorm}{\@ifstar\@opnorms\@opnorm}
\newcommand{\@opnorms}[1]{%
  \left|\mkern-1.5mu\left|\mkern-1.5mu\left|
   #1
  \right|\mkern-1.5mu\right|\mkern-1.5mu\right|
}
\newcommand{\@opnorm}[2][]{%
  \mathopen{#1|\mkern-1.5mu#1|\mkern-1.5mu#1|}
  #2
  \mathclose{#1|\mkern-1.5mu#1|\mkern-1.5mu#1|}
}

\begin{document}

\begin{frontmatter}

\title{Optimal eigen expansions and uniform bounds \protect\thanksref{T1}}
\runtitle{Eigen expansions and uniform bounds}
\thankstext{T1}{}

\begin{aug}
  \author{\fnms{Moritz Jirak}\corref{}\thanksref{t2}\ead[label=e1]{jirak@math.hu-berlin.de}}%

  \thankstext{t2}{}

  \runauthor{M. Jirak}

  \affiliation{Humboldt Universit\"{a}t zu Berlin}
  %\printead{e1}

  \address{Institut f\"{u}r Mathematik,
Unter den Linden 6
D-10099 Berlin,\\
          \printead{e1}}

\end{aug}

\begin{abstract}
Let $\bigl\{X_k\bigr\}_{k \in \Z} \in \Ln^2(\TT)$ be a stationary process with associated lag operators ${\boldsymbol{\cal C}}_h$. Uniform asymptotic expansions of the corresponding empirical eigenvalues and eigenfunctions are established under almost optimal conditions on the lag operators in terms of the eigenvalues (spectral gap). In addition, the underlying dependence assumptions are optimal in a certain sense, including both short and long memory processes. This allows us to study the relative maximum deviation of the empirical eigenvalues under very general conditions. Among other things, convergence to an extreme value distribution is shown. We also discuss how the asymptotic expansions transfer to the long-run covariance operator ${\boldsymbol{\cal \GG}}$ in a general framework.
\end{abstract}

\end{frontmatter}

\section{Introduction}

Principal component analysis (PCA) has emerged as one of the most important tools in multivariate and highdimensional data analysis. In the latter, functional principal component analysis (FPCA) is becoming more and more important. A comprehensive overview and some leading examples can be found in ~\cite{horvath_kokoszka_book_2012}, ~\cite{jolliffe_book_2002}, ~\cite{ramsay_silverman_2005}. Given a functional time series ${\bf X} = \bigl\{X_k\bigr\}_{k \in \Z}$, it is typically assumed that ${\bf X}$ lies in the Hilbert space $\Ln^2({\cal T})$, where ${\cal T} \subset \R^d$ is compact. The fundamental tool in the area of PCA and FPCA - both in theory and practice - is the usage of (functional) principal components (FPC). To fix ideas, let us introduce some notation. If ${\bf X}$ is stationary with $\E\bigl[\|X_k\|_{\Ln^2}^2\bigr] < \infty$, then the mean $\mu = \E\bigl[X_k\bigr]$ and the covariance operator
\begin{align}\label{defn_cov_operator}
{\boldsymbol{\cal C}}\bigl(\cdot\bigr) = \E\bigl[\langle X_k - \mu, \cdot \rangle (X_{k} - \mu)\bigr],
\end{align}
exist. Here $\langle \cdot,\cdot \rangle$ denotes the inner product in $\Ln^2$, and $\|\cdot\|_{\Ln^2}$ the corresponding norm. The eigenfunctions of ${\boldsymbol{\cal C}}_h$ are called the functional principal components and denoted by ${\bf e} = \{\e_j\}_{j \in \N}$, i.e; we have ${\boldsymbol{\cal C}}(\e_j) = \lambda_j \e_j$, where $\lambdav = \{\lambda_j\}_{j \in \N}$ denotes the eigenvalues. The eigenfunctions ${\bf e}$ are usually estimated by the empirical eigenfunctions $\widehat{\bf e} = \{\widehat{\e}_j\}_{j \in \N}$, defined as the eigenfunctions of the empirical covariance operator
\begin{align}\label{defn_cov_est}
\widehat{\boldsymbol{\cal C}}\bigl(\cdot\bigr) = \frac{1}{n}\sum_{k = 1}^n \langle X_k - \bar{X}_n, \cdot \rangle \bigl(X_{k} - \bar{X}_n\bigr),
\end{align}
where $\bar{X}_n = \frac{1}{n} \sum_{k = 1}^n X_k$. Hence $\widehat{\boldsymbol{\cal C}}(\widehat{\e}_j) = \widehat{\lambda}_j \widehat{\e}_j$, where $\widehat{\lambdav} = \{\widehat{\lambda}_j\}_{j \in \N}$ denotes the empirical eigenvalues. Due to the fundamental importance of eigenfunctions and eigenvalues for FPCA and PCA, corresponding results on the asymptotic behavior of empirical eigenfunctions and values are of high interest. ~\cite{anderson1963} was among the first to give such results, (see also ~\cite{dauxois_1982}), and established a CLT for $\widehat{\lambda}_j$ (resp. $\widehat{\e}_j$) if $j$ is fixed. Fueled from highdimensional applications, uniform bounds where $j$ increases with the sample size $n$ have become very important, leading to a significant rise in complexity of the problem. Well-known pathwise bounds are provided in the Lemma given below (cf. ~\cite{bhatia_mcintosh_1983},~\cite{bosq_2000}).

\begin{lemma}\label{lem_eigen_gen_upper_bound}
If ${\bf X} \in \Ln^2({\cal T})$ and $\E\bigl[\|X_k\|_{\Ln^2}^2\bigr] < \infty$, then
\begin{align*}
\bigl|\widehat{\lambda}_j - \lambda_j \bigr| \leq \bigl\|\widehat{\boldsymbol{\cal C}} - {\boldsymbol{\cal C}} \bigr\|_{{\cal L}}, \quad \bigl\|\widehat{e}_j - e_j \bigr\|_{\Ln^2} \leq \frac{2\sqrt{2}}{\psi_j} \bigl\|\widehat{\boldsymbol{\cal C}} - {\boldsymbol{\cal C}} \bigr\|_{{\cal L}},
\end{align*}
where $\psi_j = \min\bigl\{\lambda_{j-1} - \lambda_{j}, \lambda_j - \lambda_{j+1}\bigr\}$ (with $\psi_1 = \lambda_1 - \lambda_2$) and $\|\cdot\|_{{\cal L}}$ denotes the operator norm.
\end{lemma}

\begin{rem}
Strictly speaking, we consider the difference $\widehat{\e}_j - c_j\e_j$, where $c_j = \operatorname{sign}(\langle \widehat{e}_j, \e_j \rangle)$. Since $c_j$ is unidentifiable, we assume without loss of generality that throughout the remaining sequel $c_j =1$, which is the common approach in the literature.
\end{rem}

The attractiveness of the above bounds lies in their simplicity, but unfortunately they are far from optimal from a probabilistic perspective. Indeed, the results of ~\cite{dauxois_1982} tell us that in case of $\widehat{\lambda}_j - \lambda_j$, the correct bound should include the additional factor $\lambda_j$, i.e; $\lambda_j\|\widehat{\boldsymbol{\cal C}} - {\boldsymbol{\cal C}}\|_{{\cal L}}$. A similar claim can be made for $\|\widehat{e}_j - e_j\|_{\Ln^2}$. In this spirit, based on Lemma \ref{lem_eigen_gen_upper_bound}, asymptotic expansions for $\widehat{\lambda}_j - \lambda_j$ and $\widehat{e}_j - e_j$ which allow for increasing $j$ have been established in ~\cite{hall2007}, ~\cite{hall_hosseini_2006}, ~\cite{hall_hosseini_2009} (see also ~\cite{bosq_2000}, ~\cite{cardot_mas_sarda_2007}, ~\cite{mas_complex_2014}). These results have proved to be an indispensable tool in the literature, see for instance ~\cite{blanchard_2007}, ~\cite{cardot_kneip_2007}, ~\cite{cardot_mas_sarda_2007}, ~\cite{hall2007}, ~\cite{horvath_kokoszka_book_2012}, ~\cite{jolliffe_book_2002}, ~\cite{meister_2011} to name a few. But the corresponding (asymptotic) analysis is often based on heavy structural assumptions regarding ${\bf X}$ and the spacings (spectral gap) ${\bf \Psi} = \{\psi_j\}_{j \in \N}$ of the eigenvalues, limiting its applicability. In particular, often only the covariance operator $\boldsymbol{\cal C}$ is considered, and a common key assumption is that ${\bf X}$ is an IID sequence, which is rather restrictive, see ~\cite{hoermann_2010},~\cite{horvath_kokoszka_book_2012},~\cite{panaretos2013} and also Sections \ref{sec_dep_optimal} and \ref{sec_long_mem}. In the presence of serial correlation, the lag operators $\boldsymbol{\cal C}_h$ and the long-run covariance operator ${\boldsymbol{\cal \GG}}$, formally defined as
\begin{align}\label{defn_lag_operator_and_longrun_operator}
{\boldsymbol{\cal C}}_h\bigl(\cdot\bigr) = \E\bigl[\langle X_k - \mu, \cdot \rangle (X_{k-h} - \mu)\bigr], \quad {\boldsymbol{\cal \GG}}\bigl(\cdot\bigr) = \sum_{h \in \Z}{\boldsymbol{\cal C}}_h\bigl(\cdot\bigr),
\end{align}
serve as a generalization of ${\boldsymbol{\cal C}} = {\boldsymbol{\cal C}}_0$. They play a fundamental role for dependent functional time series, see for instance ~\cite{hoermann_JRSS_2015},~\cite{panaretos_tavakoli_2013},~\cite{panaretos2013}. In this paper, we consider a general framework that contains both ${\boldsymbol{\cal C}}_h$ and ${\boldsymbol{\cal \GG}}$, avoiding the previously mentioned limitations. We derive exact asymptotic expansions of $\widehat{\lambda}_j$, $\widehat{\e}_j$ under optimal dependence assumptions, allowing for short memory (weak dependence), but also for long memory (strong dependence) in case of ${\boldsymbol{\cal C}}_h$, $h$ finite. In addition, we only require a 'natural condition' concerning the spectral gap ${\bf \Psi}$. It turns out that this condition is nearly optimal.\\
\\
As a particular application, we study the relative maximum deviation of the empirical eigenvalues of ${\boldsymbol{\cal C}}$, namely
\begin{align*}
T_{\JJ_n^+}^{} = \sqrt{n}\max_{1 \leq j < \JJ_n^+}\frac{\bigl|\widehat{\lambda}_j - \lambda_j\bigr|}{ \sigma_{j}\lambda_j},
\end{align*}
where $\JJ_n^+ \to \infty$, see Proposition \ref{prop_transf_abstract_to_cov} for a precise definition of $\JJ_n^+$. Under mild assumptions, we show that
\begin{align}\label{eq_max_eigen_value}
a_n\bigl(T_{\JJ_n^+}^{} - b_n\bigr) \xrightarrow{d} \mathcal{V},
\end{align}
where $\mathcal{V}$ is a distribution of Gumbel type. The latter is based on a high dimensional Gaussian approximation, which is of independent interest, see Theorem \ref{thm_gauss_approx}. Result \eqref{eq_max_eigen_value} is particularly important for the construction of simultaneous confidence sets and tests for the relevant number of FPCs to be used for statistical inference or modelling (cf. ~\cite{bathia2010}, ~\cite{jolliffe_book_2002},~\cite{ramsay_silverman_2005}). The range of further applications is surveyed in Section \ref{sec_applications_practical}. Here we also touch on the possibility of long-memory in functional time series.

An outline of the paper can be given as follows. In Section \ref{sec_main} the key expansions of $\widehat{\lambda}_j$ and $\widehat{\e}_j$ are established in a general framework, alongside some additional results. In particular, we discuss in detail the optimality of the underlying assumptions. Asymptotic expansions of $\widehat{\lambda}_j$ and $\widehat{\e}_j$ in the context of ${\boldsymbol{\cal C}}_h$ and ${\boldsymbol{\cal G}}$ are established in Sections \ref{sec_lag_operator} and \ref{sec_long_run_cov}, whereas Section \ref{sec_applications} is devoted to the study of \eqref{eq_max_eigen_value}. Additional fields of application are surveyed in Section \ref{sec_applications_practical}, with an emphasis on functional linear regression, ARH(1) processes and long-memory in a functional context. The proofs of the eigen expansions are given in Sections \ref{sec_proof_exp}, \ref{sec_proof_lag_operator} and \ref{sec_proof_longrun}. In Section \ref{sec_gaussian_approx}, a general high dimensional Gaussian approximation under dependence is established. Based on this result, we prove \eqref{eq_max_eigen_value} in Section \ref{sec_proof_max_and_trace}. Finally, Section \ref{sec_proof_applications_practical} presents the proofs of Section \ref{sec_applications_practical}.

\section{Preliminary notation and main asymptotic expansions}\label{sec_main}

For $p \geq 1$, denote with $\|\cdot\|_p$ the $L^p$-norm $\E[|\cdot|^p]^{1/p}$. We write $\lesssim$, $\gtrsim$, ($\thicksim$) to denote (two-sided) inequalities involving a multiplicative constant, $a \wedge b = \min\{a,b\}$ and $a \vee b = \max\{a,b\}$. Given a set $\mathcal{A}$, we denote with $\mathcal{A}^c$ its complement. Moreover, we write $\overline{X} = X - \E\bigl[X\bigr]$ for a random variable $X$.\\
\\
In the sequel, it is convenient to first consider a more abstract framework. Assume that the operator ${\boldsymbol{\cal \DD}}: \Ln^2(\mathcal{T}) \mapsto \Ln^2(\mathcal{T})$ has non-negative eigenvalues $\lambdav = \{\lambda_j\}_{j \in \N}$ and eigenfunctions ${\bf \e} = \{\e_j\}_{j \in \N}$, and satisfies the spectral representation
\begin{align}\label{defn_D_1}
{\boldsymbol{\cal \DD}}(\cdot) = \sum_{j = 1}^{\infty} \lambda_j \big \langle \e_j , \cdot \big \rangle \e_j, \quad \text{with $\sum_{j = 1}^{\infty} \lambda_j < \infty$.}
\end{align}
For a sequence of non-negative numbers $\{\widetilde{\lambda}_j\}_{j \in \N}$ with $\sum_{j = 1}^{\infty} \widetilde{\lambda}_j < \infty$ and real-valued random variables $\{\etav_{i,j}^{\boldsymbol{\cal \DD}}\}_{i,j \in \N}$, $\{\etav_{i,j}^{\boldsymbol{\cal R}}\}_{i,j \in \N}$ consider the empirical version
\begin{align}\nonumber \label{defn_D_2}
&\widehat{\boldsymbol{\cal \DD}}(\cdot) =  \sum_{i,j = 1}^{\infty} \sqrt{\widetilde{\lambda}_i \widetilde{\lambda}_j}\bigl(\etav_{i,j}^{\boldsymbol{\cal \DD}} - \etav_{i,j}^{\boldsymbol{\cal R}}\bigr)\big \langle \e_i, \cdot \big \rangle \e_j, \quad \text{with $\widehat{\boldsymbol{\cal \DD}}(\widehat{\e}_j) = \widehat{\lambda}_j \widehat{\e}_j$, $j \in \N$,}\\&\text{where we demand} \quad {\boldsymbol{\cal \DD}}(\cdot) = \sum_{i,j = 1}^{\infty} \sqrt{\widetilde{\lambda}_i \widetilde{\lambda}_j}\E\bigl[\etav_{i,j}^{\boldsymbol{\cal \DD}}\bigr] \big \langle \e_i, \cdot \big \rangle \e_j.
\end{align}
The random variables $\etav_{i,j}^{\boldsymbol{\cal \DD}}$ denote the contributing random components, whereas $\etav_{i,j}^{\boldsymbol{\cal R}}$ denote the negligible parts. In the sequel, both random variables depend on a sequence $\mm \to\infty$, i.e; $\etav_{i,j}^{\boldsymbol{\cal \DD}} = \etav_{i,j}^{\boldsymbol{\cal \DD}}(\mm)$ and $\etav_{i,j}^{\boldsymbol{\cal R}} = \etav_{i,j}^{\boldsymbol{\cal R}}(\mm)$. To simplify the notation, we often suppress this dependence if it is of no immanent relevance. This class of (empirical) operators is rich enough to include the lag operators $\boldsymbol{\cal C}_h$ (in fact only $\boldsymbol{\cal C}_h^* \boldsymbol{\cal C}_h$, see Section \ref{sec_lag_operator}), but also the more general long-run covariance operator $\boldsymbol{\cal \GG}$ (see Section \ref{sec_long_run_cov}). In order to provide an intuition for this setup, let us discuss how this translates in case of the covariance operator ${\boldsymbol{\cal C}}$, hence $\boldsymbol{\cal \DD} = {\boldsymbol{\cal C}}$ and $\widehat{\boldsymbol{\cal \DD}} = \widehat{\boldsymbol{\cal C}}$. Then obviously $\widetilde{\lambda}_j = \lambda_j$ and for $\mm = n$ we have
\begin{align}\label{relations_eta_C}
\etav_{i,j}^{\boldsymbol{\cal C}}(n) = \sum_{k = 1}^n\frac{\eta_{k,i} \eta_{k,j}}{n}, \quad \etav_{i,j}^{\boldsymbol{\cal R}}(n) = \sum_{k,l = 1}^n \frac{\eta_{k,i} \eta_{l,j}}{n^2}, \quad \eta_{k,j} = \frac{\langle \overline{X}_k, \e_j \rangle}{\lambda_j^{1/2}}.
\end{align}
Clearly, if ${\bf X}$ is stationary, then so is $\{\eta_{k,j}\}_{k \in \Z, j \in \N}$ and hence ${\boldsymbol{\cal C}}$ does not depend on $n$ in this case. We also note that $\E\bigl[\etav_{j,j}^{\boldsymbol{\cal C}}\bigr] = 1$ and $\E\bigl[\etav_{i,j}^{\boldsymbol{\cal C}}\bigr] = 0$ for $i \neq j$ since $\E[\eta_{k,i} \eta_{k,j}] = 0$ by the classical Kahunen-Lo\`{e}ve expansion (cf. ~\cite{horvath_kokoszka_book_2012}). This is actually true in a more general fashion. Since ${\bf \e}$ are the eigenfunctions of $\boldsymbol{\cal \DD}$, the two representations given in \eqref{defn_D_1} and \eqref{defn_D_2} yield that $(\widetilde{\lambda}_i \widetilde{\lambda}_j)^{1/2}\E\bigl[\etav_{i,j}^{\boldsymbol{\cal \DD}}\bigr] = 0$ for $i \neq j$. For the sake of reference, we formulate this simple observation as a lemma.
\begin{lemma}\label{lem_E_eta_ij_is_zero}
Assume $\boldsymbol{\cal \DD}$ satisfies \eqref{defn_D_1} and \eqref{defn_D_2} with eigenvalues $\lambdav$ and eigenfunctions ${\bf \e}$. Then $(\widetilde{\lambda}_i \widetilde{\lambda}_j)^{1/2}\E\bigl[\etav_{i,j}^{\boldsymbol{\cal \DD}}\bigr] = 0$ for $i \neq j$ and $\lambda_j = \widetilde{\lambda}_j \E\bigl[\etav_{j,j}^{\boldsymbol{\cal \DD}}\bigr]$.
\end{lemma}

Most of our results in the sequel depend on the centered version of ${\etav}_{i,j}^{\boldsymbol{\cal \DD}}$, i.e;
\begin{align*}
\overline{\etav}_{i,j}^{\boldsymbol{\cal \DD}} = \etav_{i,j}^{\boldsymbol{\cal \DD}} - \E\bigl[{\etav}_{i,j}^{\boldsymbol{\cal \DD}}\bigr], \quad i,j \in \N.
\end{align*}

We now demand the following conditions. % on the random variables $\{\overline{\etav}_{i,j}^{\boldsymbol{\cal \DD}}\}_{i,j \in \N}$, $\{\etav_{i,j}^{\boldsymbol{\cal R}}\}_{i,j \in \N}$ and the sequence $\{\lambda_j\}_{j \in \N}$.
\begin{ass}\label{ass_abstract}
The operators $\boldsymbol{\cal \DD}$, $\widehat{\boldsymbol{\cal \DD}}$ satisfy \eqref{defn_D_1} and \eqref{defn_D_2}. Moreover, for a universal constant $C^{\boldsymbol{\cal D}}$ and a universal sequence $s_{\mm}^{\boldsymbol{\cal \DD}}= \oo(1)$ and $\ad > 0$, $\hd, p \geq 1$, $\JJ_{\mm}^+ \in \N$ and $\mm \to \infty$ it holds that
\begin{enumerate}%[leftmargin=1cm]
\item[\AoneD]\label{A1D} $\mm^{\frac{1}{2}}\max_{i,j \in \N}\bigl\|\overline{\etav}_{i,j}^{\boldsymbol{\cal \DD}}(\mm)\bigr\|_q \leq C^{\boldsymbol{\cal D}}$ and $\mm^{\frac{1}{2}}\max_{i,j \in \N}\bigl\|\etav_{i,j}^{\boldsymbol{\cal R}}(\mm)\bigr\|_q \leq s_{\mm}^{\boldsymbol{\cal \DD}}$ \\for $q = p 2^{\pd + 4}$, $\pd = \lceil \hd/\ad \rceil$,
\item[\AtwoD]\label{A2D} $\max_{1 \leq j \leq \JJ_{\mm}^+}\biggl\{{\mm}^{-\frac{1}{2} + \ad}\sum_{\substack{i = 1\\i \neq j}}^{\infty} \frac{\lambda_i^{}}{|\lambda_j^{} - \lambda_i^{}|}, {\mm}^{-1 + 2\ad}\sum_{\substack{i = 1\\i \neq j}}^{\infty} \frac{\lambda_i^{} \lambda_j^{} }{(\lambda_j^{} - \lambda_i^{})^2}\biggr\} \leq C^{\boldsymbol{\cal D}}$ \\and $\lambda_{\JJ_{\mm}^+} \geq  {\mm}^{-\hd}/C^{\boldsymbol{\cal \DD}}$,
\item[\AthreeD]\label{A3D} $1/C^{\boldsymbol{\cal \DD}} \leq \E\bigl[{\etav}_{j,j}^{\boldsymbol{\cal \DD}}(\mm)\bigr] \leq {C}^{\boldsymbol{\cal \DD}}$ for $j \in \N$ and $\sum_{j = 1}^{\infty} {\lambda}_j  \leq C^{\boldsymbol{\cal \DD}}$.
\end{enumerate}
\end{ass}

\begin{rem}
Note that in the above assumptions, $\lambdav$ may depend on $\mm$. We can deal with this case in the sequel due to the universal bounds provided by $C^{\boldsymbol{\cal D}}$.
\end{rem}

Let us discuss these assumptions and compare them to the literature. As a general preliminary remark, we note that all of our results have analogues in a general Hilbert space setting ${\mathbb H}$. Working in $\Ln^2(\TT)$ is notationally less burdensome though, and the proofs are simpler. In particular, the Fubini-Tonelli Theorem allows to interchange the order of inner products and expectations. Since most related relevant results in the literature focus on the covariance operator $\boldsymbol{\cal C}$, we also consider this setup for our discussion, i.e; $\boldsymbol{\cal \DD} = \boldsymbol{\cal C}$ (and $\widehat{\boldsymbol{\cal \DD}} = \widehat{\boldsymbol{\cal C}}$). To this end, it is convenient to translate Assumption \ref{ass_abstract} to this special case to make the comparison transparent. Recall the notation introduced in \eqref{relations_eta_C}. We then have the following result.
\begin{proposition}\label{prop_transf_abstract_to_cov}
Let ${\bf X}$ be stationary with $\E\bigl[\|X_k\|_{\Ln^2}^2\bigr] \leq C^{\boldsymbol{\cal C}}$ for a universal constant $C^{\boldsymbol{\cal C}}$. Then $\boldsymbol{\cal C}$ satisfies \eqref{defn_D_1} and \eqref{defn_D_2} with summable eigenvalues $\lambdav$ and eigenfunctions ${\bf \e}$. Assume in addition that for some $\ad > 0$, $\hd, p \geq 1$ and universal sequence $s_{n}^{\boldsymbol{\cal C}}= \oo(1)$ we have that
\begin{enumerate}%[leftmargin=1cm]
\item[\AoneC]\label{A1C} $n^{\frac{1}{2}}\max_{i,j \in \N}\bigl\|\overline{\etav}_{i,j}^{{\boldsymbol{\cal C}}}(n)\bigr\|_q < C^{\boldsymbol{\cal C}}$, $n^{\frac{1}{4}}\max_{j \in \N}\bigl\|\sum_{k = 1}^n \eta_{k,j}\bigr\|_{2q} \leq s_{n}^{\boldsymbol{\cal C}}$,\\for $q = p 2^{\pd + 4}$, $\pd = \lceil \hd/\ad \rceil$,
\item[\AtwoC]\label{A2C} \hyperref[A2D]{\AtwoD} holds with $C^{\boldsymbol{\cal D}} = C^{\boldsymbol{\cal C}}$, $\mm = n$, $\JJ_n^+ \in \N$ and $\ad$ as above.
%\item[\AthreeC]\label{A3C} $1/C^{\boldsymbol{\cal C}} \leq \bigl|\E\bigl[\eta_{k,j}^2\bigr]\bigr| \leq C^{\boldsymbol{\cal C}}$ for $j \in \N$.
\end{enumerate}
Then Assumption \ref{ass_abstract} holds for $\boldsymbol{\cal \DD} = \boldsymbol{\cal C}$ with $\ad > 0$, $\hd, p \geq 1$, $\mm = n$, $\JJ_n^+ \in \N$, $s_{\mm}^{\boldsymbol{\cal \DD}} = s_{n}^{\boldsymbol{\cal C}}$ and $C^{\boldsymbol{\cal D}} = C^{\boldsymbol{\cal C}}$ as above.
\end{proposition}

Let us now compare the literature with Proposition \ref{prop_transf_abstract_to_cov}.\\
\\
\textit{Dependence assumptions}: Assumption \hyperref[A1C]{\AoneC} implicitly imposes a dependence assumption on the scores $\eta_{k,j}$. In contrast to the literature (cf. ~\cite{crambes2013}~\cite{hall_hosseini_2006} ~\cite{hall_hosseini_2009}, ~\cite{mas_complex_2014}), we \textit{do not} require the typical independence assumption. In fact, \hyperref[A1C]{\AoneC} is much more general. In Section \ref{sec_dep_optimal} we also discuss why looking at ${\boldsymbol{\cal C}}$ under dependence can be relevant in practice. It can be shown that \hyperref[A1C]{\AoneC} holds under general, sharp weak dependence conditions. This means that if these conditions fail, we no longer have weak dependence. However, much more is valid. Suppose that $\eta_{k,j} = \sum_{i = 0}^{\infty} \alpha_{i,j} \epsilon_{k-i,j}$ where $\bigl\{\epsilon_{k,j}\bigr\}_{k \in \Z,j \in \N}$ is standard Gaussian and IID and $\alpha_{i,j} \thicksim i^{-\alpha}$, $\alpha > 1/2$. Then we show in Section \ref{sec_dep_optimal} that
\begin{align}\label{equivalence}
\bigl\|\|\widehat{\boldsymbol{\cal C}} - {\boldsymbol{\cal C}} \|_{\Ln^2}\bigr\|_2 \lesssim n^{-1/2} \quad \text{is equivalent with '\hyperref[A1C]{\AoneC} holds for any fixed $p \geq 1$',}
\end{align}
where $\|\widehat{\boldsymbol{\cal C}} - {\boldsymbol{\cal C}} \|_{\Ln^2}$ denotes the Hilbert-Schmidt-norm. Hence the rate $n^{-1/2}$ carries over and \hyperref[A1C]{\AoneC} poses no restriction, as long as we consider the CLT-domain (normalization with $n^{-1/2}$). In this sense, condition \hyperref[A1C]{\AoneC} is optimal (in the CLT-Domain). Interestingly, this also allows for long memory sequences, and we even obtain a CLT for $\widehat{\lambda}_j$ and $\widehat{\e}_j$ under long memory conditions, i.e; where $\sum_{i = 1}^{\infty} \alpha_{i,j} = \infty$, see Theorem \ref{thm_gaussian_part}. Note that it is shown in ~\cite{merlevede_1997_opt} that $\sum_{i = 1}^{\infty} |\alpha_i| <\infty$ is necessary for the validity of a CLT for $\sum_{k = 1}^n X_k$ in an infinite dimensional Hilbert space (a different normalization doesn't help here, which is different from the univariate case, see ~\cite{merlevede_1997_opt} for details). Note that condition $\max_{j \in \N}\|n^{-3/4}\sum_{k = 1}^n \eta_{k,j}\|_{2q} = \oo(1)$ is usually for 'free' due to the additional factor $n^{-1/4}$, and is only necessary to control the empirical mean correction $\bar{X}_n$. Finally, we remark that our method of proof can also be used to derive corresponding results in the non-central domain, i.e; where $\bigl\|\|\widehat{\boldsymbol{\cal C}} - {\boldsymbol{\cal C}} \|_{\Ln^2}\bigr\|_2 \thicksim b_n $ with $\sqrt{n} = \oo\bigl(b_n\bigr)$. To keep this exposition at reasonable length, this is not pursued here.\\
\\
\textit{Structural conditions for eigenvalues}: \hyperref[A2C]{\AtwoC} is the key condition regarding the structure of the eigenvalues $\lambda_j$. Note that the special form of the terms appearing in \hyperref[A2C]{\AtwoC} is no coincidence, and is connected to the variance of the asymptotic distribution of the empirical eigenfunctions $\widehat{e}_j$ (cf. ~\cite{dauxois_1982}). The literature (cf. ~\cite{cardot_mas_sarda_2007},~\cite{crambes2013},~\cite{hall2007},~\cite{hall_hosseini_2006},~\cite{hall_hosseini_2009}) usually requires polynomial, exponential or convex structures regarding the decay-rate of the eigenvalues and particularly the spacing $\psi_j$. For instance, a common minimum assumption is that $\psi_j \gtrsim \lambda_j j^{-1}$, which reflects a polynomial behavior of the eigenvalues $\lambda_j$. As will be discussed below Theorem \ref{theorem_exp_eigen_vector}, \hyperref[A2C]{\AtwoC} turns out to be much weaker, in fact, we shall see that it is nearly optimal. To get a feeling of the implications of \hyperref[A2C]{\AtwoC}, let us consider the case where $\lambda_j$ satisfies a convexity condition, i.e;
\begin{align}\label{condi_convex}
\text{the function $\lambdav(x): \, x \mapsto \lambda_x$  is convex.}
\end{align}
If \eqref{condi_convex} holds, then one may verify (cf. Lemma \ref{lem_verify_ass_poly}) that
\begin{align}\label{eq_convex_condi}
\sum_{\substack{i = 1\\i \neq j}}^{\infty} \frac{\lambda_i }{|\lambda_j - \lambda_i|} \lesssim j \log j \quad \text{and} \quad \sum_{\substack{i = 1\\i \neq j}}^{\infty} \frac{\lambda_i \lambda_j }{(\lambda_j - \lambda_i)^2} \lesssim j^2,
\end{align}
hence \hyperref[A2C]{\AtwoC} is valid if $\JJ_n^+ \lesssim n^{1/2 - \ad} (\log n)^{-1}$. Note that these bounds are not directly influenced by the decay of $\lambdav$ or ${\bf \Psi}$. The convexity condition \eqref{condi_convex} itself is mild and includes many cases encountered in the literature (cf. ~\cite{crambes2013}), in particular polynomial or exponential cases
\begin{align}\label{ex_eigen_poly} \tag{${\bf E P}$}
\lambda_j \sim j^{\rd} \rho^{-j}, \, 0 < \rho < 1, |\rd| < \infty \quad \text{or} \quad \lambda_j \sim j^{-\rd}, \, \rd > 1.
\end{align}
Also note that \hyperref[A2C]{\AtwoC} implies that the first $\JJ_n^+$ eigenvalues are distinct. See ~\cite{dauxois_1982} for a flavour of results which allow for eigenspaces with rank greater than one.\\
\\
\textit{Moment assumptions}: The existence of all moments (often with additional Gaussian like growth conditions) is usually required in the literature (cf. ~\cite{crambes2013}~\cite{hall_hosseini_2006} ~\cite{hall_hosseini_2009}, ~\cite{mas_complex_2014}) in the context of expansions for $\widehat{\lambda}_j, \widehat{\e}_j$. In contrast, we only require a finite number of moments, which, however, may be large. On the other hand, all of our results will be expressed in terms of the $\|\cdot\|_p$-norm, and moving over to the weaker $\OO_P\bigl(\cdot\bigr)$ formulation, the moment assumptions can be lowered.\\
\\
For stating our results, we introduce the quantity
\begin{align}\label{defn_I_ij_C}
I_{i,j} = \bigl \langle \bigl(\widehat{\boldsymbol{\cal D}} - {\boldsymbol{\cal D}}\bigr)(\e_i), \e_j \bigr \rangle, \quad i,j \in \N,
\end{align}
which is one of the main contributing parts in the expansions given below. We first give the main results, followed by a discussion and comparison to the literature.
\begin{comment}
Observe that due to $\hyperref[A1D]{\AoneD}$ we actually have
\begin{align*}
I_{i,j} = n^{-1}\sqrt{\lambda_i \lambda_j} \bigl(\etav_{i,j} + \oo_P(1) \bigr), \quad i,j \in \N.
\end{align*}
\end{comment}
For the empirical eigenvalues $\widehat{\lambda}_j$, we have the following.

\begin{theorem}\label{theorem_exp_eigen_value}
Assume that Assumption \ref{ass_abstract} holds. Then for $1 \leq \J < \JJ_{\mm}^+$
\begin{align*}
\biggl\|\max_{1 \leq j \leq \J} \biggr|\frac{1}{\lambda_j}\biggl(\widehat{\lambda}_j - \lambda_j - I_{j,j}\biggr)\biggr|\biggl\|_p \lesssim \frac{\J^{1/p}\mm^{-\ad}}{\sqrt{\mm}}.
\end{align*}
\end{theorem}

The above result provides an exact uniform first-order expansion for $\widehat{\lambda}_j$. For a nonuniform version, the factor $\J^{1/p}$ in the bound on the RHS can be dropped. Next, we state the companion result for the empirical eigenfunctions $\widehat{\e}_j$.
\begin{theorem}\label{theorem_exp_eigen_vector}
Assume that Assumption \ref{ass_abstract} holds. Then for $1 \leq \J < \JJ_{\mm}^+$
\begin{align*}
\biggl\|\max_{1 \leq j \leq \J} \biggr\|\frac{1}{\sqrt{\Lambda_j}
}\biggl(\widehat{\e}_j - \e_j + \frac{\e_j}{2}\bigl\|\widehat{\e}_j - \e_j\bigr\|_{\Ln^2}^2 - \sum_{\substack{k = 1\\k \neq j}}^{\infty}\e_k \frac{I_{k,j}}{\lambda_j - \lambda_k}\biggr)\biggr\|_{\Ln^2}\biggl\|_p \lesssim \frac{\J^{1/p}\mm^{-\ad}}{\sqrt{\mm}},
\end{align*}
where $\Lambda_j = \sum_{\substack{k = 1\\k \neq j}}^{\infty}\frac{\lambda_j \lambda_k}{(\lambda_j - \lambda_k)^2}$, and we also have
\begin{align*}
\biggl\|\max_{1 \leq j \leq \J} \biggr|\frac{1}{\Lambda_j}\biggl(\bigl\|\widehat{\e}_j - \e_j\bigr\|_{\Ln^2}^2 - \sum_{\substack{k = 1\\k \neq j}}^{\infty}\frac{I_{k,j}^2}{(\lambda_j - \lambda_k)^2}\biggr)\biggr|\biggl\|_p \lesssim \frac{\J^{1/p}\mm^{-\ad}}{\mm}.
\end{align*}
\end{theorem}

Theorem \ref{theorem_exp_eigen_vector} provides both uniform expansions for $\widehat{\e}_j$ and the corresponding norm. As before, the factor $\J^{1/p}$ in the bound on the RHS can be dropped for a nonuniform version. We also have a slight modification of Theorems \ref{theorem_exp_eigen_value} and \ref{theorem_exp_eigen_vector}.

\begin{proposition}\label{prop_replace_I_with_eta}
Assume that Assumption \ref{ass_abstract} holds. Then for $1 \leq \J < \JJ_{\mm}^+$, one may replace
$\{I_{k,j}\}_{k \in \N}$  with $\{(\widetilde{\lambda}_k \widetilde{\lambda}_j)^{1/2} \overline{\etav}_{k,j}^{\boldsymbol{\cal \DD}}\}_{k \in \N}$ in Theorems \ref{theorem_exp_eigen_value} and \ref{theorem_exp_eigen_vector}. Recall also that $\widetilde{\lambda}_j = \lambda_j/\E[{\etav}_{j,j}^{\boldsymbol{\cal \DD}}]$ by Lemma \ref{lem_E_eta_ij_is_zero}.
\end{proposition}

As an immediate corollary, we obtain a probabilistic version of Lemma \ref{lem_eigen_gen_upper_bound} of correct order.
\begin{corollary}\label{corollary_norms}
Assume that Assumption \ref{ass_abstract} holds. Then for $1 \leq j < \JJ_{\mm}^+$
\begin{align*}
\bigr\|\widehat{\lambda}_j - \lambda_j\bigl\|_p \lesssim \frac{\lambda_j}{\sqrt{\mm}} \quad \text{and} \quad \bigr\|\|\widehat{\e}_j - e_j\|_{\Ln^2}^2\bigl\|_p \lesssim \frac{\Lambda_j}{\mm}.
\end{align*}
\end{corollary}

\subsection{Previous results and comparison}
Let us now compare Theorems \ref{theorem_exp_eigen_value} and \ref{theorem_exp_eigen_vector} to the literature in case of ${\boldsymbol{\cal \DD}} = {\boldsymbol{\cal C}}$. It seems that the currently best known expansions in this context can be found in ~\cite{hall_hosseini_2009}. Among other things, it is required that $\bigl\{X_k\bigr\}_{k \in \Z}$ is IID, all moments exist, and the error term $ER_{\JJ_n^+}$ in the expansions of $\widehat{\lambda}_j - \lambda_j$ (not weighted with $\lambda_j^{-1}$) is of magnitude
\begin{align}\label{eq_hall_bound}
ER_{\JJ_n^+} = \max_{1 \leq j \leq \JJ_n^+}n^{-3/2} (1 - \xi_j)^{-1/2}\psi_{j}^{-3} \lambda_j^{-1/2} s_j,\quad s_j = \sup_{t \in \TT}|\e_j(t)|,
\end{align}
and $\xi_j \in (0,1)$ is defined as $\xi_j = \inf_{k <j}\bigl(1 - \frac{\lambda_k}{\lambda_j}\bigr)$. We emphasize that this is the overall error term, hence one requires for instance at least $\sqrt{n}ER_{\JJ_n^+} = \oo(1)$ for the validity of a CLT, and $(n/\lambda_{\JJ_n^+}^2)^{1/2}ER_{\JJ_n^+} = \oo(1)$ for a weighted version. If we assume the convexity condition \eqref{condi_convex}, we see that \hyperref[A2C]{\AtwoC} is much weaker. In fact, takeing for instance $\lambda_j \thicksim j^{-\cd}$ we find that $ER_{\JJ_n^+} \gtrsim n^{-3/2} (\JJ_n^+)^{3 + 7\cd/2}$. On the other hand, we see from \eqref{eq_convex_condi} that if $\J_n^+ \thicksim n^{1/2 - \ad}$, $\ad > 0$, we still obtain valid asymptotic expansions, i.e; the expressions containing $I_{k,j}$ are still the principal terms in our expansions, reflecting the exact asymptotic behavior. In stark contrast, $ER_{\JJ_n^+}$ already explodes for $\ad$ small (resp. $\cd$ large) enough, rendering a vacuous result. Similarly, \hyperref[A2C]{\AtwoC} is valid if we only require
\begin{align}\label{eq_uniform_spacings_condition}
\max_{1 \leq j \leq \JJ_n^+}n^{-1/2}/\psi_j \lesssim n^{-\ad} \quad \text{for some arbitrary $\ad > 0$,}
\end{align}
and again obtain valid asymptotic expansions. On the other hand, the actual approximation error $ER_{\JJ_n^+}$ in ~\cite{hall_hosseini_2009} may even be unbounded, since $1/\lambda_j \to \infty$ as $j$ increases. In this sense, Assumption \ref{ass_abstract} is substantially weaker.

\subsection{Dependence assumptions: optimality}\label{sec_dep_optimal}
Throughout this section, we assume that ${\boldsymbol{\cal \DD}} = {\boldsymbol{\cal C}}$. We first present the following result.
\begin{theorem}\label{thm_gaussian_part}
Assume that ${\bf X}$ has zero mean such that for $\alpha > 3/4$
\begin{align}\nonumber
\eta_{k,j} = \sum_{i = 0}^{\infty} \alpha_{i,j} \epsilon_{k-i,j}, \quad \text{$0 \leq \alpha_{i,j} \thicksim i^{-\alpha}$ and $\epsilon_{k,j}$ are standard Gaussian IID.}
\end{align}
Then \hyperref[A1C]{\AoneC} holds. Moreover, if we have in addition \hyperref[A2C]{\AtwoC} (for $\JJ_n^+$ possibly finite), then for any fixed $1 \leq j < \JJ_n^+$
\begin{align*}
&\sqrt{n}\bigl(\widehat{\lambda}_j - \lambda_j\bigr) \xrightarrow{w} \Gaussian\bigl(0,\lambda_j^2 \sigma_{\lambda_j}^2\bigr) \quad \text{and}
&\sqrt{n}\bigl(\widehat{\e}_j - \e_j\bigr) \xrightarrow{w} \Gaussian\bigl(0,\Sigma_{\e_j}\bigr),
\end{align*}
where $\xrightarrow{w}$ denotes weak convergence in the corresponding (Hilbert) space, and $\sigma_{\lambda_j}^2$ ($\Sigma_{\e_j}$) denotes the corresponding variance (operator).
\end{theorem}
\begin{comment}
Following the discussion in ~\cite{arcones1994},~\cite{breuermajor}, condition ??? is necessary for
\begin{align*}
n^{-1/2}\bigl\|\etav_{i,j}\bigr\|_2 < \infty, \quad \text{for any $i,j \in \N$,}
\end{align*}
in the sense that if $\sum_{|k|\leq n}\Cov\bigl[\eta_{0,j},\eta_{k,j}\bigr]^2 \gtrsim L(n)$ for some slowly varying function $L(n)$ and fixed $j \in \N$, then $n^{-1/2}\bigl\|\etav_{j,j}\bigr\|_2 \gtrsim L(n)$. If we restrict ourselves to a more special case we get a strict relation in terms of $\bigl\|\|\widehat{\boldsymbol{\cal C}} - {\boldsymbol{\cal C}} \|_{\Ln^2}\bigr\|_2$. Let $\eta_{k,j} = \sum_{i = 0}^{\infty} \alpha_{i,j} \epsilon_{k-i,j}$ with $0 \leq \alpha_{i,j} \thicksim i^{-\alpha}$ and $\epsilon_{k,j}$ standard Gaussian IID.
\end{comment}
The above result indicates that $\alpha = 3/4$ is the boundary value for a CLT with normalization $\sqrt{n}$, see also the discussion in ~\cite{arcones1994},~\cite{breuermajor}. In fact, given the linear structure of $\eta_{k,j}$ one readily computes that
\begin{align*}
\bigl\|\|\widehat{\boldsymbol{\cal C}} - {\boldsymbol{\cal C}} \|_{\Ln^2}\bigr\|_2 \lesssim n^{-1/2} \quad \text{iff $\alpha > 3/4$.}
\end{align*}
On the other hand, Lemma \ref{lem_bound_C_op} below yields that \hyperref[A1C]{\AoneC} implies $\bigl\|\|\widehat{\boldsymbol{\cal C}} - {\boldsymbol{\cal C}} \|_{\Ln^2}\bigr\|_2 \lesssim n^{-1/2}$. Hence we obtain the equivalence in \eqref{equivalence}. Finally, note that the regime $1/2 < \alpha \leq 1$ is generally considered as \textit{long memory}. Hence by Theorem \ref{thm_gaussian_part} above, we obtain a CLT for $\widehat{\lambda}_j$ and $\widehat{\e}_j$ even in the presence of long memory, where $3/4 < \alpha \leq 1$. If $1/2 < \alpha \leq 3/4$, Non-central limit theorems arise. If $\alpha \leq 1/2$, then $\E\bigl[\|X_0\|_{\Ln^2}^2\bigr] = \infty$, which requires a completely different treatment.

\subsection{Spectral gap: almost optimality}\label{sec_spec_gap_aoptimal}

Next, we discuss the issue of 'almost optimality' of condition \hyperref[A2C]{\AtwoC}. To this end, we draw heavily from the noteworthy results of ~\cite{mas_complex_2014}. Suppose that $\bigl\{\eta_{i,j}\bigr\}_{i,j \in \N}$ are IID and satisfy $\E\bigl[|\eta_{i,j}|^{2p}\bigr] \leq p! C^{p-1}$ for some constant $C > 0$. If a structure condition like \eqref{ex_eigen_poly} holds, then it is shown in ~\cite{mas_complex_2014} that
\begin{align}\label{eq_ryumgaart}
\E\bigl[\|\widehat{\e}_j - \e_j\|_{\Ln^2}^2\bigr] \lesssim \frac{j^2 (\log n)^2}{n}.
\end{align}
As can be seen from Corollary \ref{corollary_norms}, this bound deviates from the optimal one by the additional factor $(\log n)^2$. On the other hand, note that in the polynomial case in \eqref{ex_eigen_poly}, this bound is also valid for $j > \JJ_n^+$ (we require $\ad > 0$), which is a slightly larger region. In ~\cite{mas_complex_2014}, a lower bound is also provided, which is $\frac{j^2}{n} \wedge 1$. Strictly speaking, it is proven for the projection $\widehat{\pi}_j = \widehat{\e}_j \otimes \widehat{\e}_j$, where $\otimes$ denotes the one-rank operation
\begin{align*}
u \otimes v (w) = \langle u, w \rangle v, \quad u,v,w \in \Ln^2(\mathcal{T}).
\end{align*}
According to ~\cite{mas_complex_2014}, it then holds that (recall that $\mathcal{L}$ denotes the operator norm)
\begin{align}\label{eq_ryumgaart_2}
\frac{j^2}{n} \wedge 1 \lesssim \E\bigl[\|\widehat{\pi}_j - \pi_j\|_{\mathcal{L}}^2\bigr] \lesssim \frac{j^2 (\log n)^2}{n} \wedge 1.
\end{align}
Heuristically, this may also be inferred from ~\cite{dauxois_1982}. On the other hand, Corollary \ref{corollary_norms} and elementary computations yield
\begin{align}\label{eq_emp_eigen_discuss}
\E\bigl[\|\widehat{\pi}_j - \pi_j\|_{\mathcal{L}}^2\bigr] \lesssim \frac{1}{n}\sum_{\substack{k = 1\\k \neq j}}^{\infty}\frac{\lambda_j \lambda_k}{(\lambda_j - \lambda_k)^2} \lesssim \frac{j^2}{n}, \quad \text{if $j \leq n^{1/2 - \ad} (\log n)^{-1}$,}
\end{align}
(in the polynomial case) and thus the order of the upper and lower bounds match for $j \leq n^{1/2 - \ad} (\log n)^{-1}$. If $j \geq n^{1/2}$, Cauchy-Schwarz yields the trivial optimal upper bound. Since $\ad > 0$ may be chosen arbitrarily small given sufficiently many (all) moments, we find that our conditions on the eigenvalues $\lambdav$ are essentially optimal. In other words, we obtain exact expansions and the optimal error bound for almost the complete region of indices $j$ where \eqref{eq_emp_eigen_discuss} still converges to zero.

\section{Lag operator}\label{sec_lag_operator}

While the covariance operator $\boldsymbol{\cal C}$ is a key object for serially uncorrelated data ${\bf X}$, the lag operator $\boldsymbol{\cal C}_h$ and the long-run covariance operator $\boldsymbol{\cal \GG}$ become more relevant in the presence of serial correlation. We focus on $\boldsymbol{\cal C}_h$ in this section, and then carry out a similar program in Section \ref{sec_long_run_cov} for $\boldsymbol{\cal \GG}$. To facilitate the discussion, let us first introduce a popular notion of weak dependence. In the remainder of this section, we assume that for each $j \in \N$, the score sequence $\bigl\{\eta_{k,j}\bigr\}_{k \in \Z}$ is a \textit{causal} weak Bernoulli sequence, which can be written as
\begin{align}
\eta_{k,j} = g_j\bigl(\ldots,\epsilon_{k-1,j},\epsilon_{k,j}\bigr)
\end{align}
for some measurable functions $g_j$ and IID sequences $\{\epsilonv_k\}_{k \in \Z}$ with $\epsilonv_k = \bigl\{\epsilon_{k,j}\bigr\}_{j \in \N}$. We do not specify any crosswise dependence between $\epsilon_{k,i}$, $\epsilon_{k,j}$ for $i \neq j$, allowing for a large flexibility. Let $\mathcal{E}_{k,j} = \bigl(\epsilon_{i,j}, \, i \leq k \bigr)$. To quantify the dependence of $\bigl\{\eta_{k,j}\bigr\}_{k \in \Z}$, we adopt the coupling idea. Let $\bigl\{\epsilon_{k,j}'\bigr\}_{k \in \Z,j \in \N}$ be an IID copy of $\bigl\{\epsilon_{k,j}\bigr\}_{k \in \Z,j \in \N}$ and $\mathcal{E}_{k,j}' = \bigl(\mathcal{E}_{-1,j},\epsilon_{0,j}', \epsilon_{1,j}, \ldots, \epsilon_{k,j} \bigr)$ the coupled version of $\mathcal{E}_{k,j}$. Then we define
\begin{align}\label{defn_Omega}
\Omega_p(k) = \max_{j \in \N}\bigl\|\eta_{k,j} -  \eta_{k,j}'\bigr\|_p \quad \text{for $p \geq 1$, where $\eta_{k,j}' = g_j\bigl(\mathcal{E}_{k,j}'\bigr)$.}
\end{align}
Roughly speaking, $\Omega_p(k)$ measures the overall degree of dependence of $\eta_{k,j} = g_j(\mathcal{E}_{k,j})$ on
$\epsilon_{0,j}'$ and it is directly related to the data-generating mechanism of the underlying process (~\cite{wu_2005} refers to $\Omega_p(k)$ as physical dependence measure). This dependence concept is well established in the literature, and popular processes like ARMA, GARCH, iterated random functions etc. fit into this framework (cf. ~\cite{wu_2005}, ~\cite{sipwu}). Consider for example the linear process $\eta_{k,j} = \sum_{l = 0}^{\infty} \alpha_l \epsilon_{k-l,j}$ where $\bigl\{\epsilon_{k,j}\bigr\}_{k,\in \Z,j \in \N}$ is IID with $\bigl\|\epsilon_{k,j}\bigr\|_p < \infty$. Then
\begin{align}\label{eq_condi_nec_for_CLT}
\sum_{k = 1}^{\infty} \Omega_p(k) < \infty \quad \text{holds iff} \quad \sum_{k = 1}^{\infty} |\alpha_k| <\infty.
\end{align}
In this sense, \eqref{eq_condi_nec_for_CLT} is \textit{necessary} for a CLT. In fact, if it is violated, one can construct examples such that
\begin{align*}
\lim_{n \to \infty} \frac{1}{n}\biggl\|\sum_{k = 1}^n \eta_{k,j} \biggr\|_2^2 = \infty, \quad j \in \N,
\end{align*}
and a different normalization than $n^{-1/2}$ is required (cf. ~\cite{peligradutev2}). In the sequel, all dependence conditions will be expressed in terms of summability conditions of $\Omega_p(k)$.

\begin{comment}
We start with the following result, which can be deduced from Corollary 1 and Theorem 2 in ~\cite{wu_asymptotic_bernoulli} resp. Theorem 1 in ~\cite{sipwu}.

\begin{proposition}\label{prop_max_norm_of_eta}
Assume that $\sum_{k = 1}^{\infty} \Omega_{2p}(k) < \infty$ for $p \geq 2$ and $h = \oo(n)$. Then
\begin{align*}
\max_{i,j \in \N}\bigl\|n^{-1/2}\overline{\etav}_{i,j}^{\boldsymbol{\cal C}_h}\bigr\|_p < \infty \quad \text{and} \quad \max_{j \in \N}\bigl\|n^{-3/4}\sum_{k = 1}^n \eta_{k,j}\bigr\|_{2p} \lesssim n^{-1/4}.
\end{align*}
\end{proposition}

Proposition \ref{prop_max_norm_of_eta} shows that condition \hyperref[A1C]{\AoneC} holds under sharp weak dependence conditions, including a large variety of linear and non-linear processes, see the earlier references above. Related results can be established under different weak dependence conditions, see for instance ~\cite{Peligrad_amaximal} or ~\cite{dedecker_prieur_2005}.
\end{comment}

A major difference when dealing with $\boldsymbol{\cal C}_h$ compared to $\boldsymbol{\cal C}$ (and $\boldsymbol{\cal G}$) is that it only satisfies a singular-value decomposition (SVD) in general, i.e; there exist orthonormal Bases ${\bf e} = \{\e_j\}_{j \in \N}$, ${\bf f} = \{f_j\}_{j \in \N}$ and a sequence of real numbers $\lambdav = (\lambda_j)_{j \in \N}$ tending to zero such that for fixed $h \in \Z$
\begin{align}\label{eq_lag_operator}
\boldsymbol{\cal C}_h (\cdot) = \E\bigl[\langle \overline{X}_k, \cdot \rangle \overline{X}_{k-h}\bigr] = \sum_{j = 1}^{\infty} \sqrt{\lambda_j} \langle \e_j, \cdot \rangle f_j, \quad \text{if $\E\bigl[\|X_k\|_{\Ln^2}^2\bigr] < \infty$.}
\end{align}
Hence a priori, $\boldsymbol{\cal C}_h$ does not fit into our framework. However, by considering the symmetrized version $\boldsymbol{\cal \DD}(\cdot) = \boldsymbol{\cal C}_h^* \boldsymbol{\cal C}_h (\cdot)$, we end up with an operator that meets our requirements. Here, $\boldsymbol{\cal C}_h^*$ denotes the adjoint operator of $\boldsymbol{\cal C}_h$, given by
\begin{align}\label{eq_lag_operator_adjoint}
\boldsymbol{\cal C}_h^* (\cdot) = \E\bigl[\langle \overline{X}_{k-h}, \cdot \rangle \overline{X}_{k}\bigr] = \sum_{j = 1}^{\infty} \sqrt{\lambda_j} \langle f_j, \cdot \rangle \e_j.
\end{align}
Routine computations (with $\overline{X}_k = \sum_{j = 1}^{\infty}\widetilde{\lambda}_j^{1/2} \eta_{k,j} \e_j$) then indeed reveal that
\begin{align}\label{eq_lag_operator_symmetrized}
\boldsymbol{\cal \DD}(\cdot) & = \sum_{j = 1}^{\infty} \lambda_j \langle \e_j, \cdot \rangle \e_j = \sum_{j = 1}^{\infty} \biggl(\widetilde{\lambda}_j \sum_{k = 1}^{\infty} \widetilde{\lambda}_k \E\bigl[\eta_{h,k}\eta_{0,j}\bigr]^2\biggr) \langle \e_j, \cdot \rangle \e_j.
\end{align}
\begin{comment}
\begin{align}
\boldsymbol{\cal \DD}(\cdot) &= \boldsymbol{\cal C}_h^* \boldsymbol{\cal C}_h (\cdot) = \sum_{j = 1}^{\infty} \lambda_j \langle \e_j, \cdot \rangle \e_j \\&= \sum_{i,j = 1}^{\infty} \sqrt{\widetilde{\lambda}_i \widetilde{\lambda}_j} \langle \e_i, \cdot \rangle \e_j \sum_{k = 1}^{\infty} \widetilde{\lambda}_k \E\bigl[\eta_{h,k}\eta_{0,i}\bigr] \E\bigl[\eta_{h,k}\eta_{0,j}\bigr].
\end{align}
\end{comment}
Hence $\boldsymbol{\cal \DD}$ has a spectral decomposition with eigenvalues $\lambdav$ and eigenfunctions ${\bf e}$ and satisfies \eqref{defn_D_1}. Representations \eqref{eq_lag_operator}, \eqref{eq_lag_operator_adjoint} motivate a natural plug-in estimator for $\boldsymbol{\cal \DD}$ (cf. ~\cite{bosq_2000}), given as (for $h \in \N$)
\begin{align}\label{defn_empirical_lag_operator_symmetrized}
\widehat{\boldsymbol{\cal \DD}}\bigl(\cdot\bigr) = \frac{1}{(n-h)^2}\sum_{1 \leq k,l \leq n-h} \langle X_{l+h} - \bar{X}_n, X_{k+h} - \bar{X}_n \rangle \langle X_k- \bar{X}_n, \cdot \rangle \bigl(X_l - \bar{X}_n\bigr).
\end{align}
The empirical SVD components $\widehat{\lambdav}=\{\widehat{\lambda}_j\}_{j \in \N}$, $\widehat{\bf e} = \{\widehat{\e}_j\}_{j \in \N}$ and $\widehat{\bf f} = \{\widehat{f}_j\}_{j \in \N}$ are then defined via
\begin{align}
\widehat{\boldsymbol{\cal \DD}}\bigl(\widehat{\e}_j\bigr) = \widehat{\lambda}_j \widehat{\e}_j, \quad \widehat{\boldsymbol{\cal C}}_h\bigl(\widehat{\e}_j\bigr) = \widehat{\lambda}_j^{1/2} \widehat{f}_j,
\end{align}
where the empirical lag operator $\widehat{\boldsymbol{\cal C}}_h$ is given by
\begin{align}\label{defn_lag_op_est}
\widehat{\boldsymbol{\cal C}}_h\bigl(\cdot\bigr) = \frac{1}{n-h}\sum_{k = h + 1}^n \langle X_k - \bar{X}_n, \cdot \rangle \bigl(X_{k-h} - \bar{X}_n\bigr), \quad 0 \leq h \leq n-1,
\end{align}
and analogously for $-n+1 \leq h < 0$. In order to apply Theorems \ref{theorem_exp_eigen_value} and \ref{theorem_exp_eigen_vector} to $\widehat{\lambdav}$ and $\widehat{\bf e}$, the key objective is to validate \hyperref[A1D]{\AoneD} for appropriate $\overline{\etav}_{i,j}^{\boldsymbol{\cal \DD}}$ and $\etav_{i,j}^{\boldsymbol{\cal R}}$. To this end, introduce
\begin{align*}
A_{l,h,r,i,j} = \bigl(\eta_{l+h,r}\eta_{l,i} - \E[\eta_{l+h,r}\eta_{l,i}]\bigr)\E\bigl[\eta_{l+h,r}\eta_{l,j}\bigr], \quad l,h,r,i,j \in \N.
\end{align*}
Recalling $\overline{X}_k = \sum_{j = 1}^{\infty}\widetilde{\lambda}_j^{1/2} \eta_{k,j} \e_j$, we then define $\etav_{i,j}^{\boldsymbol{\cal \DD}}$ for fixed $h \in \N$ as
\begin{align}
\etav_{i,j}^{\boldsymbol{\cal \DD}}(n) = \frac{1}{n-h} \sum_{l = 1}^{n-h} \sum_{r = 1}^{\infty} \widetilde{\lambda}_r \bigl(A_{l,h,r,i,j} + A_{l,h,r,j,i}\bigr) + \sum_{r = 1}^{\infty} \widetilde{\lambda}_r \E\bigl[\eta_{h,r}\eta_{0,i}\bigr]\E\bigl[\eta_{h,r}\eta_{0,j}\bigr].
\end{align}
Note that this automatically defines $\etav_{i,j}^{\boldsymbol{\cal R}}$ via \eqref{defn_D_2}, see also \eqref{eq_prop_lag_operator_0} in the proof. We then have the following result.

\begin{proposition}\label{prop_lag_operator}
Let $q \geq 2$ and assume $\E[\|X_k\|_{\Ln^2}^2] < \infty$ and $\Omega_{4q}(k) \lesssim k^{-\bd}$, $\bd>3/2$. Then $\boldsymbol{\cal \DD} = \boldsymbol{\cal C}_h^* \boldsymbol{\cal C}_h$ and $\widehat{\boldsymbol{\cal \DD}}$ as in \eqref{defn_empirical_lag_operator_symmetrized} satisfy \eqref{defn_D_1} and \eqref{defn_D_2} such that
\begin{align*}
n^{1/2}\max_{i,j \in \N}\bigl\|\overline{\etav}_{i,j}^{\boldsymbol{\cal \DD}}\bigr\|_q < \infty, \quad n^{1/2}\max_{i,j \in \N}\bigl\|\etav_{i,j}^{\boldsymbol{\cal R}}\bigr\|_q \lesssim n^{-1/2}.
\end{align*}
\end{proposition}

Related results can be established under different weak dependence conditions, see for instance ~\cite{dedecker_prieur_2005} or ~\cite{Peligrad_amaximal}. Using Proposition \ref{prop_lag_operator}, it is now easy to transfer the results, which we summarize in the following theorem.

\begin{theorem}\label{thm_lag_operator}
Suppose that $\E\bigl[\|X_k\|_{\Ln^2}^2\bigr] \leq C^{\boldsymbol{\cal C}_h}$ for a universal constant $C^{\boldsymbol{\cal C}_h}$. Assume in addition that for some $\ad > 0$, $\hd, p \geq 1$ we have that
\begin{enumerate}%[leftmargin=1cm]
\item[\AoneCh]\label{A1Ch} $\Omega_{4q}(k) \lesssim k^{-\bd}$, $\bd>3/2$ for $q = p 2^{\pd + 4}$, $\pd = \lceil \hd/\ad \rceil$,
\item[\AtwoCh]\label{A2Ch} \hyperref[A2D]{\AtwoD} holds with $C^{\boldsymbol{\cal D}} = C^{\boldsymbol{\cal C}_h}$, $\mm = n$, $\JJ_n^+ \in \N$ and $\ad$ as above,
\item[\AthreeCh]\label{A3Ch} $0 < \inf_{j \in \N} \sum_{r = 1}^{\infty} \widetilde{\lambda}_r \E\bigl[\eta_{h,r}\eta_{0,j}\bigr]^2$.
\end{enumerate}
Then Assumption \ref{ass_abstract} holds for $\boldsymbol{\cal \DD} = \boldsymbol{\cal C}_h^* \boldsymbol{\cal C}_h$ and $\widehat{\boldsymbol{\cal \DD}}$ as in \eqref{defn_empirical_lag_operator_symmetrized} with $\ad > 0$, $\hd, p \geq 1$, $\mm = n$, $\JJ_n^+ \in \N$, $s_{\mm}^{\boldsymbol{\cal \DD}} = s_{n}^{\boldsymbol{\cal \DD}} = n^{-1/2}$ and $C^{\boldsymbol{\cal D}} = C^{\boldsymbol{\cal C}_h}$ as above. In particular, Theorems \ref{theorem_exp_eigen_value} and \ref{theorem_exp_eigen_vector} apply to $\widehat{\lambdav}$ and $\widehat{\bf e}$.
\end{theorem}

It remains to deal with $\widehat{\bf f}$, which is the subject of Theorem \ref{thm_lag_operator_svd_f} below.

\begin{theorem}\label{thm_lag_operator_svd_f}
Grant the assumptions of Theorem \ref{thm_lag_operator}, and let $1 \leq p' \leq p$. Then
\begin{align*}
&\biggl\|\biggl\|\widehat{f}_j - f_j - \frac{(\widehat{\lambda}_j - \lambda_j)f_j}{2\lambda_j} -  \frac{{\boldsymbol{\cal C}}_h\bigl(\widehat{\e}_j - \e_j\bigr) + \bigl(\widehat{\boldsymbol{\cal C}}_h - {\boldsymbol{\cal C}}_h\bigr)\bigl(\e_j\bigr)}{\sqrt{\lambda_j}}\biggr\|_{\Ln^2} \biggr\|_{p'}\\&\lesssim  \frac{1}{\sqrt{\lambda_j n}}\biggl(\bigl\|\|\widehat{\e}_j - \e_j\|_{\Ln^2}\bigr\|_{4p'} + \frac{1}{\sqrt{n}}\biggr).
\end{align*}
\end{theorem}

As the proof shows, Theorem \ref{thm_lag_operator_svd_f} is essentially a concatenation of the previous results. Note in particular that the above expansion can be developed further in a straightforward manner by employing Theorems \ref{theorem_exp_eigen_value} and \ref{theorem_exp_eigen_vector}.

\section{Long-run covariance operator}\label{sec_long_run_cov}

The long-run covariance operator is a natural generalization of the covariance operator in the presence of serial correlation. From a statistical perspective, this is particularly relevant in the context of the CLT, where under appropriate conditions on ${\bf X}$, we have that
\begin{align}\label{CLT_longrun}
\frac{1}{\sqrt{n}}S_n = \frac{1}{\sqrt{n}}\sum_{k = 1}^n \overline{X}_k \xrightarrow{w} \Gaussian\bigl(0, \boldsymbol{\cal \GG}\bigr) \quad \text{and} \quad \sup_{n} n^{-1/2}\bigl\|\|S_n\|_{\Ln^2}\bigr\|_2 < \infty,
\end{align}
where ${\boldsymbol{\cal \GG}}(\cdot)$ is the long-run covariance operator, (formally) defined as
\begin{align*}
{\boldsymbol{\cal \GG}}(\cdot) = \sum_{h \in \Z}\boldsymbol{\cal C}_h(\cdot), \quad \boldsymbol{\cal C}_h(\cdot) = \E\bigl[\langle \overline{X}_k, \cdot \rangle \overline{X}_{k-h}\bigr].
\end{align*}
Note that $\boldsymbol{\cal \GG}$ in general only exists if $\sum_{h \in \Z} \|\boldsymbol{\cal C}_h\|_{{\cal L}} < \infty$, which is usually referred to as a weak dependence condition. In view of \eqref{CLT_longrun}, we see that $\boldsymbol{\cal \GG}$ takes over the role of $\boldsymbol{\cal C}$ if ${\bf X}$ has serial correlation: in the 'limit case' where $n^{-1/2}S_n$ is distributed as $\Gaussian\bigl(0, \boldsymbol{\cal \GG}\bigr)$, the best (in $\Ln^2$-sense) finite dimensional approximations are provided by the classical Kahunen-Lo\`{e}ve decomposition with respect to $\boldsymbol{\cal \GG}$. Hence we can expect that for large enough $n$, finite dimensional approximations of $n^{-1/2}S_n$ based on appropriate estimates $\widehat{\boldsymbol{\cal \GG}}$ are close to optimality too. We refer to ~\cite{hoermann_JRSS_2015}, ~\cite{horvath_kokoszka_reeder_JRSS_2013},~\cite{panaretos_tavakoli_2013}, ~\cite{panaretos2013}, and more recently ~\cite{cerovecki_hoermann_2015_arxiv} for further discussions. A unifying, even more general object than $\boldsymbol{\cal \GG}$ is the spectral density operator $\boldsymbol{\cal F}(\theta)$, first studied in ~\cite{panaretos2013}, which recently has attracted a lot of attention (cf. ~\cite{hoermann_JRSS_2015}, ~\cite{panaretos_tavakoli_2013}). A (detailed) study is beyond the scope of the present note, and is left open for future research. It appears though that at least some of the results can be transferred.\\
\\
Estimation of $\boldsymbol{\cal \GG}$ is a delicate issue, and already in the univariate/multivariate case a substantial body of literature has evolved around this problem, see for instance ~\cite{andrews},~\cite{han_wu_2014_max}, ~\cite{wu_asymptotic_bernoulli} and the many references therein. In the context of functional data, we refer for instance to ~\cite{hoermann_JRSS_2015}, ~\cite{horvath_kokoszka_reeder_JRSS_2013},  ~\cite{panaretos_tavakoli_2013},~\cite{panaretos2013}. The basic principle is plug-in estimation, which leads to the estimates
\begin{align}\label{defn_G_longrun_estimator}
&\widehat{{\boldsymbol{\cal \GG}}}^b\bigl(\cdot\bigr) = \widehat{{\boldsymbol{\cal C}}}_0\bigl(\cdot\bigr) + \sum_{h = 1}^b \omega_h \bigl(\widehat{{\boldsymbol{\cal C}}}_h(\cdot) + \widehat{{\boldsymbol{\cal C}}}_{-h}(\cdot)\bigr), \quad \text{where $\widehat{\boldsymbol{\cal C}}_h\bigl(\cdot\bigr)$ is as in \eqref{defn_lag_op_est},}
\end{align}
and $|\omega_h| \leq 1$ is a sequence of weight functions.
%is symmetric!
%Setting $\widehat{\boldsymbol{\cal \GG}}(\cdot) = \langle \widehat{{\bf \GG}}, \cdot \rangle$ renders the usual estimator.
%Depending on the choice of $\omega_h$, possibly complex-valued, $\widehat{{\boldsymbol{\cal \GG}}}^b$ can be used more generally as estimate for the spectral density (cf. ~\cite{panaretos2013}, ~\cite{hoermann_JRSS}), which is in analogy to the univariate case. %Observe that if $\omega_j = 1$ and $\omega_h = 0$ for $h \neq j$ then $\widehat{{\boldsymbol{\cal \GG}}}^b$ degenerates to the standard lag estimator $\widehat{{\boldsymbol{\cal C}}}_j$.
In the sequel, the choice of $\omega_h$ has little impact on the results, and we therefore set $\omega_h = 1$ for the remainder of this section. For consistent estimates, it is necessary that $b = b_n \to \infty$ as $n$ increases. Even so, in contrast to $\widehat{\boldsymbol{\cal C}}_h$, the estimate $\widehat{\boldsymbol{\cal \GG}}^b$ is biased. Depending on the decay rate of $\|\boldsymbol{\cal C}_h\|_{{\cal L}}$, the optimal choice of $b_n$ is $b_n \thicksim \log n$ (geometric decay), or $b_n \thicksim n^{1/(2s + 1)}$ (polynomial decay with $s$), see ~\cite{andrews}. Thus, the actual operator we are estimating is
\begin{align}\label{defn_long_run_b}
\boldsymbol{\cal \GG}^b\bigl(\cdot\bigr) = \sum_{|h|\leq b} \boldsymbol{\cal C}_h\bigl(\cdot\bigr).
\end{align}
Note that in general $\E[\widehat{\boldsymbol{\cal \GG}}^b] \neq \boldsymbol{\cal \GG}^b$ and hence $\widehat{\boldsymbol{\cal \GG}}^b$ is still biased, but this bias is negligible. We point out that subject to some regularity conditions (cf. ~\cite{panaretos2013})
\begin{align}\label{eq_rate_of_conv_long_run}
\bigl\|\|\widehat{\boldsymbol{\cal \GG}}^b -\boldsymbol{\cal \GG}^b\|_{\Ln^2}\bigr\|_2 \thicksim \sqrt{n/b},
\end{align}
which is the same rate as in the univariate case (cf. ~\cite{andrews}). Moreover, under quite general assumptions (cf. ~\cite{hoermann_JRSS_2015},~\cite{panaretos2013}), it follows that $\boldsymbol{\cal \GG}^b$ satisfies the spectral decomposition
\begin{align}\label{Gb_representation_1}
\boldsymbol{\cal \GG}^b(\cdot) = \sum_{j = 1}^{\infty} \lambda_j^{b} \big \langle \e_j^b, \cdot \rangle \e_j^b, \quad \sum_{j = 1}^{\infty} \lambda_j^b < \infty,
\end{align}
with eigenvalues $\lambdav^b = \{\lambda_j^b\}_{j \in \N}$ and eigenfunctions ${\bf \e}^b = \{\e_j^b\}_{j \in \N}$.
Since the actual underlying operator of interest is $\boldsymbol{\cal \GG}^b$, it is natural to (first) express our conditions in terms of $\lambdav^b$ and ${\bf \e}^b$. We can decompose $\overline{X}_k$ as
\begin{align}\label{eq_rep_X_k}
\overline{X}_k = \sum_{j = 1}^{\infty} \sqrt{\widetilde{\lambda}_j^b} \eta_{k,j}^b \e_j^b, \quad \widetilde{\lambda}_j^b = \E\bigl[\langle \overline{X}_k, \e_j^b \rangle^2 \bigr],\,\, \eta_{k,j}^b = \langle \overline{X}_k, \e_j \rangle (\widetilde{\lambda}_j^b)^{-1/2}.
\end{align}
Observe that in general $\E[\eta_{k,j}^b \eta_{k,i}^b] \neq 0$ for $i \neq j$, which is different from the Kahunen-Lo\`{e}ve expansion. In analogy to \eqref{relations_eta_C}, we also introduce the quantity
\begin{align}\label{definition_eta_bn}
\etav_{i,j}^b = \etav_{i,j}^{b}(n) = \sum_{k = 1}^n\frac{ {\eta_{k,i}^{b} \eta_{k,j}^{b}}}{n} + \sum_{h = 1}^{b}\sum_{k = h+1}^n \frac{\eta_{k,i}^{b} \eta_{k-h,j}^{b} + \eta_{k-h,i}^{b} \eta_{k,j}^{b}}{n-h}.
\end{align}
It is then easy to see that
\begin{align}\label{Gb_representation_2}
\widehat{\boldsymbol{\cal \GG}}^b(\cdot) = \sum_{i,j = 1}^{\infty} \sqrt{\widetilde{\lambda}_i^b \widetilde{\lambda}_j^b} \bigl(\etav_{i,j}^{b} + \etav_{i,j}^{\boldsymbol{\cal R}}\bigr)\langle \e_i^b, \cdot \rangle \e_j^b, \quad \boldsymbol{\cal \GG}^b(\cdot) = \sum_{i,j = 1}^{\infty} \sqrt{\widetilde{\lambda}_i^b \widetilde{\lambda}_j^b} \E\bigl[\etav_{i,j}^{b}\bigr] \langle \e_i^b, \cdot \rangle \e_j^b,
\end{align}
for appropriate (degenerate) random variables $\{\etav_{i,j}^{\boldsymbol{\cal R}}\}_{i,j \in \N}$ (see \eqref{eq_thm_longrun_b_remains_valid_1}). Takeing \eqref{Gb_representation_1} into account, we see that both \eqref{Gb_representation_1}, \eqref{Gb_representation_2} match the setup in \eqref{defn_D_1} and \eqref{defn_D_2}.
We can thus appeal to the results of Section \ref{sec_main}. To this end, it is convenient to denote with
\begin{align*}
\varphi_{i,j}^{b} = \E\bigl[\etav_{i,j}^{b}\bigr] = \sum_{|h| \leq b}\E\bigl[\eta_{h,i}^{b}\eta_{0,j}^{b}\bigr], \quad i,j \in \N.
\end{align*}
Note that by Lemma \ref{lem_E_eta_ij_is_zero} we have for $b \in \N$ (including $b = \infty$)
\begin{align}\label{lem_varphi_relations}
\varphi_{i,j}^{b} = 0 \quad \text{if $i \neq j$ and $\lambda_j^b = \varphi_{j,j}^{b} \widetilde{\lambda}_j^b$}.
\end{align}

Let us now translate Assumption \ref{ass_abstract} to our present setup.

\begin{ass}\label{ass_eigenvaectors_longrun}
The sequence ${\bf X}$ is stationary such that $\sum_{h \in \Z} \|\boldsymbol{\cal C}_h\|_{\mathcal{L}} < \infty$. Moreover, for $b = \oo(n)$, a universal constant $C^{\boldsymbol{\cal \GG}}<\infty$ and universal sequence $s_n^{\boldsymbol{\cal \GG}} = \oo(1)$ and $\ad > 0$, $\hd, p \geq 1$ and $\JJ_n^+ \in \N$ it holds that
\begin{enumerate}%[leftmargin=1cm]
\item[\Aoneb]\label{A1b} $(n/b)^{\frac{1}{2}}\max_{i,j \in \N}\bigl\|\overline{\etav}_{i,j}^{b}(n)\bigr\|_q \leq C^{\boldsymbol{\cal G}}$, $n^{-\frac{3}{4}} b^{\frac{1}{4}} \max_{j \in \N}\bigl\|\sum_{k = 1}^n \eta_{k,j}^b\bigr\|_{2q} \leq s_n^{\boldsymbol{\cal \GG}}$ for $q = p 2^{\pd + 4}$, $\pd = \lceil \hd/\ad \rceil$,
\item[\Atwob]\label{A2b} $\max_{1 \leq j \leq \JJ_n^+}\biggl\{{(n/b)}^{-\frac{1}{2} + \ad}\sum_{\substack{i = 1\\i \neq j}}^{\infty} \frac{\lambda_i^{b}}{|\lambda_j^{b} - \lambda_i^{b}|}, {(n/b)}^{-1 + 2\ad}\sum_{\substack{i = 1\\i \neq j}}^{\infty} \frac{\lambda_i^{b} \lambda_j^{b} }{(\lambda_j^{b} - \lambda_i^{b})^2}\biggr\} \leq C^{\boldsymbol{\cal \GG}}$ and $\lambda_{\JJ_n^+}^{b} \gtrsim {(n/b)}^{-\hd}$,
\item[\Athreeb]\label{A3b} $1/C^{\boldsymbol{\cal \GG}} \leq \varphi_{j,j}^{b} \leq C^{\boldsymbol{\cal \GG}}$ for $j \in \N$, $\sum_{j = 1}^{\infty} \lambda_j^b \leq C^{\boldsymbol{\cal \GG}}$.
\end{enumerate}
\end{ass}

Let us discuss these conditions. In view of \eqref{eq_rate_of_conv_long_run}, the choice $m = n/b$ is quite natural. Condition \hyperref[A1b]{\Aoneb} is a little more explicit than \hyperref[A1D]{\AoneD}, but of the same nature. \hyperref[A2b]{\Atwob}, \hyperref[A3b]{\Athreeb} are essentially translations of \hyperref[A2D]{\AtwoD}, \hyperref[A3D]{\AthreeD}. Note that in the present formulation, \hyperref[A3b]{\Athreeb} reflects the common non-degeneracy assumption encountered in the time series literature.\\
\\
The setup in Assumption \ref{ass_eigenvaectors_longrun} is quite general. Before looking at the possible range of applications, let us formulate the transferred results. To this end, in analogy to $I_{i,j}$ in \eqref{defn_I_ij_C}, we introduce $I_{i,j}^b$ as
\begin{align}\label{defn_I_ij_b}
I_{i,j}^b = \bigl \langle \bigl(\widehat{\boldsymbol{\cal \GG}}^b - {\boldsymbol{\cal \GG}}^b\bigr)(\e_i^b), \e_j^b \bigr \rangle, \quad i,j \in \N.
\end{align}

We then have the following general transfer result.

\begin{theorem}\label{thm_longrun_b_remains_valid}
Assume that Assumption \ref{ass_eigenvaectors_longrun} holds. Then for $1 \leq \J < \JJ_n^+$, Theorem \ref{theorem_exp_eigen_value} and Theorem \ref{theorem_exp_eigen_vector} remain valid if we substitute $n/b$, $\lambda_j^b$, $\e_j^b$, $\widehat{\lambda}_j^b$, $\widehat{\e}_j^b$ and $I_{i,j}^b$ at the corresponding places. Moreover, corresponding versions of Proposition \ref{prop_replace_I_with_eta} and Corollary \ref{corollary_norms} hold.
\end{theorem}

Due to the uniform bounds provided by $C^{\boldsymbol{\cal \GG}}$ in Assumption \ref{ass_eigenvaectors_longrun}, Theorem \ref{thm_longrun_b_remains_valid} can either be used pointwise (for arbitrary but fixed $b,n \in \N$), or uniformly in $b,n$, depending on whether Assumption \ref{ass_eigenvaectors_longrun} holds pointwise or uniformly. The strength and weakness of Theorem \ref{thm_longrun_b_remains_valid} is that everything is essentially expressed in terms of the operator $\boldsymbol{\cal \GG}^b$. The positive aspect is that this makes the assumptions rather general (in fact, almost optimal in a certain sense, see below). On the other hand, the drawback is that these conditions can be difficult to verify, since they explicitly depend on $b$. If $b = b_n$ is a function in $n$ this is not so useful, and one would be more interested in uniform bounds in terms of $n$. Let us mention here that the trouble mainly originates from \hyperref[A2b]{\Atwob} and not \hyperref[A1b]{\Aoneb}. It is therefore desirable to find simple conditions that depend in a more transparent way on $b$, and preferably mainly on ${\boldsymbol{\cal G}}$. More precisely, the aim is to find simple, sufficient conditions that imply a uniform validity of Assumption \ref{ass_eigenvaectors_longrun}. Before turning to this issue, let us first discuss an interesting case where the problem mentioned above does not occur.\\
\\
{\bf $\mathfrak{m}$-correlated processes}: We call ${\bf X}$ an $\mathfrak{m}$-correlated process if $\boldsymbol{\cal C}_h = 0$ for $|h| > \mathfrak{m}$, where $\mathfrak{m}$ is finite. Locally dependent processes are quite common in the literature, and often modeled as $\mathfrak{m}$-dependent processes. Clearly, $\mathfrak{m}$-dependency implies $\mathfrak{m}$-correlation. Moreover, we get that
\begin{align*}
\boldsymbol{\cal \GG}^b = \sum_{|h| \leq b}\boldsymbol{\cal C}_h = \sum_{|h| \leq \mathfrak{m}}\boldsymbol{\cal C}_h = \boldsymbol{\cal \GG}^{\mathfrak{m}} = \boldsymbol{\cal \GG}^{}, \quad \text{if $\mathfrak{m}\leq b$.}
\end{align*}
Note that $\mathfrak{m}$-correlation also implies that representations \eqref{Gb_representation_1} and \eqref{Gb_representation_2} are valid. Hence we conclude the following.
\begin{corollary}\label{cor_m_correlated}
If ${\bf X}$ is $\mathfrak{m}$-correlated and $\mathfrak{m}\leq b$, then we can replace $\e_j^b$, $\eta_{k,j}^b$ with $\e_j^{\mathfrak{m}}$, $\eta_{k,j}^{\mathfrak{m}}$ everywhere in \eqref{eq_rep_X_k} and \eqref{definition_eta_bn} (which alters \hyperref[A1b]{\Aoneb}), and $b$ with $\mathfrak{m}$ everywhere in \hyperref[A2b]{\Atwob} and \hyperref[A3b]{\Athreeb}.
\end{corollary}

Corollary \ref{cor_m_correlated} shows that Theorem \ref{thm_longrun_b_remains_valid} applies to a large class of processes under general and accessible conditions. Note in particular, that the optimality criterium used in Section \ref{sec_spec_gap_aoptimal} also applies since $\mathfrak{m}$ is finite. In the presence of $\mathfrak{m}$-dependence, the conditions can be further simplified. More precisely, routine calculations reveal that \hyperref[A1b]{\Aoneb} can be replaced with
\begin{enumerate}%[leftmargin=1.3cm]
\item[\Aonem]\label{A1m} $\max_{j \in \N}\bigl\|\eta_{k,j}^{\mathfrak{m}}\bigr\|_{2q}<\infty$ for $q = p 2^{\pd + 4}$, $\pd = \lceil \hd/\ad \rceil$.
\end{enumerate}

Let us now return to the problem of uniform bounds where $b = b_n \to \infty$ as $n$ increases. As mentioned earlier, it is desirable to find analogue conditions that depend in a more transparent way on $b$, and are expressed mainly in terms of ${\boldsymbol{\cal \GG}}$. To this end, it is convenient to denote with $\lambda_j = \lambda_j^{\infty}$, $\e_j = \e_j^{\infty}$, $\varphi_{i,j}^{\star} = \varphi_{i,j}^{\infty}$ and $\eta_{k,j}^{\star} = \eta_{k,j}^{\infty}$. For the sake of reference, we then restate the decomposition of $X_k$ in this context, which amounts to
\begin{align}\label{eq_rep_X_k_infty}
\overline{X}_k = \sum_{j = 1}^{\infty} \sqrt{\widetilde{\lambda}_j} \eta_{k,j} \e_j, \quad \widetilde{\lambda}_j = \E\bigl[\langle \overline{X}_k, \e_j \rangle^2 \bigr], \quad \eta_{k,j} = \langle \overline{X}_k, \e_j \rangle
\widetilde{\lambda}_j^{-1/2}.
\end{align}
Recall the notion of $\Omega_p(k)$, defined in \eqref{defn_Omega}. We then make the following set of assumptions.

\begin{ass}\label{ass_eigenvaectors_longrun_exp_decay}
Let $\ad > 0$, $1 < \cd^+ \leq \cd^- < \infty$ and $\JJ_n^+ \lesssim n^{1/2 - \ad} (\log n)^{-\frac{3}{2}}$. Put $p^* = p 2^{\pd + 4}$, $\pd = \lceil \cd^-/(2\ad) \rceil$ and $b \geq C_0 \log n$ for $C_0 > 0$ sufficiently large. It then holds that
\begin{enumerate}%[leftmargin=1cm]
\item[\Aonestar]\label{A1star} $\Omega_k(2p^*) \lesssim \rho^{k}, \quad 0 < \rho < 1$,
\item[\Atwostar]\label{A2star} the function $\lambdav(x): \, x \mapsto \lambda_x$ is convex and $j^{-\cd^-} \lesssim \lambda_j \lesssim j^{-\cd^+}$ uniformly for $j \in \N$,
\item[\Athreestar]\label{A3star} $1/C^{\boldsymbol{\cal G}} \leq \min_{j \in \N}\varphi_{j,j}^{\star}$ for $C^{\boldsymbol{\cal G}} > 0$.
\end{enumerate}
\end{ass}

\begin{rem}\label{rem_replace_poly_with_exp}
Condition $j^{-\cd^-} \lesssim \lambda_j \lesssim j^{-\cd^+}$ can also be replaced with $e{^{-\cd^- j}} \lesssim \lambda_j \lesssim e^{-\cd^+ j}$, provided that $\JJ_n^+ \lesssim \log n$. Similarly, the convexity condition in \hyperref[A2star]{\Atwostar} can be replaced with $\max_{1 \leq j \leq \JJ_n^+}1/\psi_j \lesssim n^{\cd^-}$, where we recall that $\psi_j = \min\bigl\{\lambda_{j-1} - \lambda_{j}, \lambda_j - \lambda_{j+1}\bigr\}$ (with $\psi_1 = \lambda_1 - \lambda_2$).
\end{rem}

Let us elaborate on these assumptions. \hyperref[A1star]{\Aonestar} is a weak dependence condition that requires a geometric decay, and implies in particular that $\boldsymbol{\cal \GG}$ exists. This condition is satisfied for a large number of processes in the literature such as ARMA and GARCH models. Note that instead of using $\Omega_k(2p^*)$ as dependence measure, one could also use mixing concepts like strong mixing or $\tau$-mixing (cf. ~\cite{dedecker_prieur_2005}). We remark that the method of proof can also be used under the weaker assumption of polynomial decay. Unfortunately, this leads to (significantly) more restrictive conditions for the eigenvalues $\lambdav$ and the range $\JJ_n^+$. This is not surprising, since in this case the bias $\|\boldsymbol{\cal \GG} - \boldsymbol{\cal \GG}^{b} \|_{\mathcal{L}}$ is much larger and thus more relevant, particularly if $b = b_n$ is chosen in the optimal way. In view of Lemma \ref{lem_eigen_gen_upper_bound} (see also Lemma \ref{lem_operator_comp} for a more general version), it seems to be impossible to express Assumption \ref{ass_eigenvaectors_longrun} in terms of $\boldsymbol{\cal \GG}$ without additional (heavy) assumptions for $\lambdav$ and/or $\JJ_n^+$, simply because the distance $\|\boldsymbol{\cal \GG} - \boldsymbol{\cal \GG}^{b}\|_{\mathcal{L}}$ is too large. Condition \hyperref[A2star]{\Atwostar} imposes regularity conditions on the eigenvalues $\lambdav$. We have already seen in the discussion of \hyperref[A2C]{\AtwoC} in Section \ref{sec_main} that the convexity condition is mild, and leads to the simple condition $\JJ_n^+ \lesssim n^{1/2 - \ad} (\log n)^{-1}$ (see \eqref{eq_convex_condi}). Assumption $j^{-\cd^-} \lesssim \lambda_j \lesssim j^{-\cd^+}$ implies that $\lambdav$ fluctuates between polynomial decay boundaries. Condition $1 < \cd^+ \leq \cd^- < \infty$ allows for a large variety here though, with a possible varying decay coefficient for $\lambda_j$. Moreover, as is pointed out in Remark \ref{rem_replace_poly_with_exp} above, a formulation in terms of geometric decay boundaries is also possible. Finally, \hyperref[A3star]{\Athreestar} reflects the usual non-degeneracy condition already mentioned above.\\
\\
In analogy to \eqref{defn_I_ij_b}, we introduce
\begin{align}\label{defn_I_ij_b_infty}
I_{i,j}^{(\infty,b)} = \bigl \langle \bigl(\widehat{\boldsymbol{\cal \GG}}^b - {\boldsymbol{\cal \GG}}^b\bigr)(\e_i), \e_j \bigr \rangle, \quad i,j \in \N.
\end{align}
\begin{comment}
Observe that under Assumption \ref{ass_eigenvaectors_longrun_exp_decay} we have that
\begin{align}
I_{i,j}^{(\infty,b)} = \sqrt{\widetilde{\lambda}_i \widetilde{\lambda}_j}\biggl(\sum_{k = 1}^n \frac{\overline{\eta_{k,i}^{} \eta_{k,j}^{}}}{n} + 2\sum_{h = 1}^{b}\sum_{k = h+1}^n \frac{\overline{\eta_{k,i}^{} \eta_{k-h,j}^{}}}{n-h} + R_{i,j}\biggr),
\end{align}
where $\bigl\|\max_{1 \leq i,j < \JJ_n^+}|R_{i,j}|\bigr\|_{p^*} \lesssim (\JJ_n^+)^{1/p}/n$. This follows from Lemma \ref{lem_max_mom} and analogue computations as in ~\cite{wu_asymptotic_bernoulli}.
\end{comment}

We then have the following first order expansion for the empirical eigenvalues $\widehat{\lambdav}$.
\begin{theorem}\label{theorem_exp_decay_long_run_eigen}
Assume that Assumption \ref{ass_eigenvaectors_longrun_exp_decay} holds. Then for $1 \leq \JJ < \JJ_n^+$
\begin{align*}
\biggl\|\max_{1 \leq j \leq \J} \biggr|\frac{1}{\lambda_j}\biggl(\widehat{\lambda}_j^b - \lambda_j - I_{j,j}^{(\infty,b)}\biggr)\biggr|\biggl\|_p \lesssim \frac{\J^{1/p}(n/b)^{-\ad}}{\sqrt{n/b}}.
\end{align*}
\end{theorem}

Next, we state the corresponding result for the empirical eigenfunctions $\widehat{\bf \e}$.
\begin{theorem}\label{theorem_exp_decay_long_run_fun}
Assume that Assumption \ref{ass_eigenvaectors_longrun_exp_decay} holds. Then for $1 \leq \JJ < \JJ_n^+$
\begin{align*}
\biggl\|\max_{1 \leq j \leq \J} \biggr\|\frac{1}{\sqrt{\Lambda_j}
}\biggl(\widehat{\e}_j^b - \e_j + \frac{\e_j}{2}\bigl\|\widehat{\e}_j^b - \e_j\bigr\|_{\Ln^2}^2 - \sum_{\substack{k = 1\\k \neq j}}^{\infty}\e_k \frac{I_{k,j}^{(\infty,b)}}{\lambda_j - \lambda_k}\biggr)\biggr\|_{\Ln^2}\biggl\|_p \lesssim \frac{\J^{1/p}(n/b)^{-\ad}}{\sqrt{n/b}},
\end{align*}
where $\Lambda_j = \sum_{\substack{k = 1\\k \neq j}}^{\infty}\frac{\lambda_j \lambda_k}{(\lambda_j - \lambda_k)^2}$, and we also have
\begin{align*}
\biggl\|\max_{1 \leq j \leq \J} \biggr|\frac{1}{\Lambda_j}\biggl(\bigl\|\widehat{\e}_j^b - \e_j\bigr\|_{\Ln^2}^2 - \sum_{\substack{k = 1\\k \neq j}}^{\infty}\frac{(I_{k,j}^{(\infty,b)})^2}{(\lambda_j - \lambda_k)^2}\biggr)\biggr|\biggl\|_p \lesssim \frac{\J^{1/p}(n/b)^{-\ad}}{n/b}.
\end{align*}
\end{theorem}

As before, we also have corresponding versions of Proposition \ref{prop_replace_I_with_eta} and Corollary \ref{corollary_norms}. Formulating the analogues needs a little more care and is not immediate, so we state them explicitly. To this end, denote with
\begin{align}\nonumber \label{defn_eta_infty_b}
&\etav_{i,j}^{(\infty,b)} = \etav_{i,j}^{(\infty,b,1)} + \etav_{i,j}^{(\infty,b,2)} \quad \text{where} \quad \etav_{i,j}^{(\infty,b,1)} = \sum_{k = 1}^n\frac{{\eta_{k,i}^{\star} \eta_{k,j}^{\star}}}{n},\\&\text{and} \quad \etav_{i,j}^{(\infty,b,2)} = \sum_{h = 1}^{b}\sum_{k = h+1}^n \frac{\eta_{k,i}^{b} \eta_{k-h,j}^{b} + \eta_{k-h,i}^{b} \eta_{k,j}^{b}}{n-h}.
\end{align}
Then we have the following results.
\begin{proposition}\label{prop_replace_I_with_eta_exp}
Assume that Assumption \ref{ass_eigenvaectors_longrun_exp_decay} holds. Then for $1 \leq \J < \JJ_{n}^+$, one may replace
$\{I_{k,j}^{(\infty,b)}\}_{k \in \N}$  with $\{(\widetilde{\lambda}_k \widetilde{\lambda}_j)^{1/2} \overline{\etav}_{k,j}^{(\infty,b)}\}_{k \in \N}$ in Theorems \ref{theorem_exp_decay_long_run_eigen} and \ref{theorem_exp_decay_long_run_fun}.
\end{proposition}

\begin{corollary}\label{cor_norm_bounds_exp}
Assume that Assumption \ref{ass_eigenvaectors_longrun_exp_decay} holds. Then for $1 \leq j < \JJ_{n}^+$
\begin{align*}
\bigr\|\widehat{\lambda}_j - \lambda_j\bigl\|_p \lesssim \frac{\lambda_j}{\sqrt{n}} \quad \text{and} \quad \bigr\|\|\widehat{\e}_j - e_j\|_{\Ln^2}^2\bigl\|_p \lesssim \frac{\Lambda_j}{n}.
\end{align*}
\end{corollary}

\section{Maximum deviation of empirical eigenvalues}\label{sec_applications}

As already mentioned, Theorems \ref{theorem_exp_eigen_value} and \ref{theorem_exp_eigen_vector} can be used to obtain various fluctuation results for eigenvalues or eigenfunctions. We exemplify this further in case of ${\boldsymbol{\cal \DD}} = {\boldsymbol{\cal C}}$, mentioning that a similar program can be carried out for ${\boldsymbol{\cal \DD}} = {\boldsymbol{\cal C}}_h^*{\boldsymbol{\cal C}}_h$, $h \in \Z$ fixed. To this end, we formally introduce the longrun covariance (recall that $\overline{X} = X - \E[X]$) as
\begin{align}\label{eq_var_def}
\gamma_{i,j} = \lim_{n \to \infty}  \frac{1}{n}\E\biggl[\sum_{k,l = 1}^{n}\bigl(\eta_{k,i}^2 - 1\bigr) \bigl(\eta_{l,j}^2 - 1\bigr)\biggr].
\end{align}
In Section \ref{sec_gaussian_approx} we show that this is well-defined given Assumption \ref{ass_max_eigenvaleu} below. Moreover, for $\sigma_{j}^2 = \gamma_{j,j}$ we have the usual representation $\sigma_{j}^2 = \sum_{k \in \Z} \phi_{k,j}$, where $\phi_{k,j} = \cov[\eta_{0,j}\eta_{0,j},\eta_{k,j}\eta_{k,j}]$. Consider ${\boldsymbol{\cal C}}$ with eigenvalues $\lambdav$ and denote with
\begin{align}\label{eq_defn_T_d_lambda}
T_{\JJ}^{} = \sqrt{n}\max_{1 \leq j < \JJ}\frac{|\widehat{\lambda}_j - \lambda_j|}{\sigma_{j} \lambda_j}, \quad T_{\JJ}^{Z_{}} = \max_{1 \leq j < \JJ}\bigl|Z_{j}\bigr|,
\end{align}
where $\bigl\{Z_{j}\bigr\}_{1 \leq j < \JJ}$ is a zero mean sequence of Gaussian random variables with correlation structure $\Sigma_{\JJ}^{Z_{}} = \bigl(\rho_{i,j}\bigr)_{1 \leq i,j < \JJ}$, where $\rho_{i,j} = \gamma_{i,j}/\sigma_{i} \sigma_{j}$. In the sequel, we show that
$T_{\JJ_n^+}^{}$ is close to $T_{\JJ_n^+}^{Z}$ in probability. To this end, we work under the following assumption.
\begin{ass}\label{ass_max_eigenvaleu}
For $p \geq 1$ let $q = p 2^{\pd + 4}$, $\pd = \lceil \hd/\ad \rceil$, and assume that
\begin{enumerate}%[leftmargin=1cm]
\item[\Bone]\label{B1} $\E\bigl[\|X_k\|_{\Ln^2}^2\bigr] < \infty$ and \hyperref[A2C]{\AtwoC} hold (with $\ad,\hd$ as above) such that\\ $\bigl(\JJ_n^+\bigr)^{1/p} n^{-\ad} \lesssim n^{-\delta}$, $\delta > 0$,
\item[\Btwo]\label{B2}  $\Omega_k(2q) \lesssim k^{-\bd}$, $\bd > 3/2$,
\item[\Bthree]\label{B3} $\inf_j \sigma_{j}> 0$.
\end{enumerate}
\end{ass}
Note that these assumptions are mild. In particular, the decay rate  $\bd$ in condition $\hyperref[B2]{\Btwo}$ is completely independent of the underlying dimension $\JJ_n^+$. We then have the following result.
\begin{theorem}\label{thm_eigen_max_gauss_approx}
Grant Assumption \ref{ass_max_eigenvaleu}. Then
\begin{align*}
\sup_{x \in \R}\bigl|P\bigl(T_{\JJ_n^+}^{} \leq x \bigr) - P\bigl(T_{\JJ_n^+}^{Z_{}} \leq x \bigr)\bigr| \lesssim n^{-C}, \quad C > 0.
\end{align*}
\end{theorem}

The above result provides a Gaussian approximation with an algebraic rate. Note that no conditions on the underlying covariance structure are required. If we impose a very weak decay assumption on $\gamma_{\lambda,i,j}$, we obtain the limit distribution.

\begin{corollary}\label{cor_max_limit_distrib}
Grant Assumption \ref{ass_max_eigenvaleu}, and assume in addition
\begin{align}\label{eq_condit_decay_corr}
|\gamma_{i,j}|\log(|i-j|) = \oo(1) \quad \text{for $|i-j| \to \infty$.}
\end{align}
Then for $x \in \R$
\begin{align*}
\lim_{n \to \infty}P\bigl(T_{\JJ_n^+}^{}\leq u_{\JJ_n^+}(x) \bigr) = \exp\bigl(-e^{-x}\bigr),
\end{align*}
where $u_m(x) = x/a_m  + b_m$ with $a_m = (2 \log m)^{1/2}$ and $b_m = (2 \log m)^{1/2} - (8 \log m)^{-1/2}\bigl( \log \log m + 4\pi - 4\bigr)$ for $m \in \N$.
\end{corollary}

\begin{rem}
Note that condition \eqref{eq_condit_decay_corr} is essentially the weakest possible currently known, see ~\cite{leadbetter_1974}, ~\cite{leadbetter_1983} and ~\cite{han_wu_2014_max}.
\end{rem}

Uniform control measures are an important statistical tool and have many applications. In the present context, Corollary \ref{cor_max_limit_distrib} allows for the construction of simultaneous confidence bands for $\widehat{\lambda}_j$. This in turn is very useful to assess parametric hypothesis and decay rates of the structure of $\lambdav$. A particular and important case is the determination of relevant principle components. A huge number of stopping rules have been developed in the literature (cf. ~\cite{jackson_1993}, ~\cite{jolliffe_book_2002}), which all require a uniform control of $\widehat{\lambdav}$. As pointed out by a reviewer, Corollary \ref{cor_max_limit_distrib} can be particularly useful in case of threshold rules like the scree plot, see also ~\cite{bathia2010} for related problems.

\section{Applications}\label{sec_applications_practical}

A huge bulk of testing and estimation problems in FPCA is related to the normalized scores $\{\eta_{k,j}\}_{k \in \Z, j \in \N}$ in some way or other, where the associated operator is either $\boldsymbol{\cal C}_h$ or $\boldsymbol{\cal \GG}$. Among others, we mention (two) sample mean tests and related problems (~\cite{horvath_kokoszka_book_2012}, ~\cite{horvath_kokoszka_reeder_JRSS_2013}, ~\cite{mas_mean_2007}), tests about potential serial correlation, stationarity and related issues (~\cite{bathia2010},~\cite{ferraty_2012} ,~\cite{fremdt_2013},~\cite{horvath_huskova_rice_independence_2013},~\cite{horvath_kokoszka_rice_econometrics_2014},~\cite{Kraus_panaretos_2012_biometrika},
~\cite{panaretos2013},~\cite{panaretos_kraus_maddocks_2010_JASA}), various change point problems, ~\cite{aston_2012},~\cite{berkes_2009_JRRS},~\cite{horvath_kokoska_2010}, and many more. Given a sample of size $n$, the canonical estimator of the scores is their empirical version
\begin{align*}
\widehat{\eta}_{k,j} = \langle X_k, \widehat{\e}_j \rangle (\widehat{\lambda}_j)^{-1/2}, \quad 1 \leq k \leq n, \,1 \leq j \leq \JJ_n^+.
\end{align*}
Intuitively, it is clear that the power of tests or estimation accuracy is augmented if $\JJ_n^+$ increases with the sample size, since more and more information is taken into account. From a theoretical statistical point of view, this can be made rigorous by minimax theory for estimates and Ingster's (minimax)-theory for tests (cf. ~\cite{ingster_book_2003},~\cite{mas_hilgert_2013}). In ~\cite{fremdt_horvath_kokoszka_steinebach_increasing_2014}, a striking example is presented where a very large amount of principal components is required to adequately describe the data, see also ~\cite{cai_yuan_2012}. Let us also mention that the necessity of uniform control of $\widehat{\lambdav}$ and $\widehat{\bf \e}$ also arises in the completely different field of machine learning in the context of techniques based on \textit{Reproducing Kernel Hilbert spaces}, see for instance ~\cite{blanchard_2007}. All this highlights the importance of a uniform, accurate control of $\widehat{\lambdav}$ and $\widehat{\bf \e}$ as $\JJ_n^+$ increases, and the usefulness of results like Theorems \ref{theorem_exp_eigen_value} and \ref{theorem_exp_eigen_vector}.

Let us briefly discuss how this relates to our main Assumption \ref{ass_abstract}. Due to its general formulation, \hyperref[A1D]{\AoneD} is very flexible. In particular, all the problems mentioned above can be reformulated in a (general) framework (depending on the problem and corresponding operator) such that \hyperref[A1D]{\AoneD} is valid. Regarding \hyperref[A2D]{\AtwoD}, the convexity assumption \eqref{condi_convex} leading to \eqref{eq_convex_condi} provides a general and simple condition that is recommended for all the applications. In particular, the resulting range $\JJ_n^+$ of potentially allowed principal components is quite large. \hyperref[A3D]{\AthreeD} typically reflects a non-degeneracy condition, which usually is necessary any way in the problem at hand. We do not take this discussion any further, but rather investigate two other applications a little more detailed. The first one is the functional linear model, which contains in particular first order autoregression in Hilbert spaces (coined ARH(1) or FAR(1)). As a second, very different application, we survey how and why long-memory situations can arise in a functional context and how this relates to our results.

\begin{comment}
\begin{description}
\item \hyperref[A2D]{\AtwoD} \& \hyperref[A3D]{\AthreeD}: We have already seen that
In light of the discussion about the convexity condition in \eqref{condi_convex} leading to \eqref{eq_convex_condi},
\item \hyperref[A1D]{\AoneD}: Validating condition \hyperref[A1D]{\AoneD} heavily depends on the properties of ${\bf X}$. If $\JJ_n^+ \to \infty$ it is usually much more convenient to directly impose conditions on the scores $\{\eta_{k,j}\}_{k \in \Z, j \in \N}$. In fact, this is unavoidable if the problem in question directly depends on the scores and one is aiming for optimal (minimax) rates or limit theorems, see for instance ~\cite{cai_hall_2006}, ~\cite{comte_johannes_2012}, ~\cite{cardot_mas_sarda_2007}, ~\cite{hall2007} in the context of functional linear regression. In this spirit, Proposition provides simple and general conditions .
\end{description}
\end{comment}

\subsection{Functional linear regression}\label{sec_functional_linear_regression}

A fundamental regression model in a high-dimensional context is the \textit{functional linear model}. Given ${\bf X} = \{X_k\}_{k \in \Z}$, ${\bf Y} = \{Y_k\}_{k \in \Z} \in \Ln^2(\mathcal{T})$, the basic model is defined as
\begin{align}\label{defn_model_general_flr}
X_k = {\bf \Phi}(Y_{k}) + \epsilon_k, \quad k \in \Z,
\end{align}
where ${\bf \Phi}$ is a (bounded) linear operator, mapping from $\Ln^2(\mathcal{T})$ to $\Ln^2(\mathcal{T})$, and $\varepsilonv = \{\epsilon_k\}_{k \in \Z} \in \Ln^2(\mathcal{T})$ is a noise sequence. The goal is to recover ${\bf \Phi}$, given ${\bf X}$ and ${\bf Y}$, while the noise $\varepsilonv$ is unknown. Observe that estimating ${\bf \Phi}$ is an \textit{ill-posed} problem, see e.g. ~\cite{cavalier_tsybakov_2002} for a more detailed discussion. Model \eqref{defn_model_general_flr} and its many variations have been extensively studied in the literature, with active research persisting (see e.g. ~\cite{hoermann_JTSA_2015}), and it would be impossible to survey all the results. From a theoretic perspective, a significant part of the current literature (cf. ~\cite{cai_hall_2006},~\cite{cardot_mas_sarda_2007},~\cite{comte_johannes_2012}, ~\cite{hall2007},~\cite{hall_hosseini_2006},~\cite{meister_2011} and the extensive references therein) focuses on the case where ${\bf Y}$ and $\varepsilonv$ are mutually independent (which excludes ARH(1)), and in addition $X_k, {\bf \Phi}(Y_{k}), \epsilon_k$ are all real-valued. Hence by Riesz-representation ${\bf \Phi}(\cdot) = \langle x^{\phi}, \cdot \rangle$ for some $x^{\phi} \in \Ln^2(\mathcal{T})$, and it all boils down to the estimation of $x^{\phi}$.

Let us touch on the main idea for estimating ${\bf \Phi}$. Denote with $\boldsymbol{\cal C}^y$ the covariance operator of ${\bf Y}$ with eigenvalues $\lambdav^{y}$ and eigenfunctions ${\bf \e}^{y}$. For the remainder of this section, we assume that $\varepsilonv = \{\epsilon_k\}_{k \in \Z} \in \Ln^2(\mathcal{T})$ is an IID sequence, and for each $k \in \Z$, $\epsilon_k$ and $Y_k$ are independent. Applying Fubini-Tonelli we get that for $j \in \N$
\begin{align*}
{\bf \Upsilon}(\e_j) &= \E\bigl[\langle Y_k, \e_j^{y} \rangle X_k \bigr] = \E\bigl[\langle Y_k, \e_j^{y} \rangle {\bf \Phi}(Y_k)\bigr] + \E\bigl[\langle Y_k, \e_j^{y} \rangle \epsilon_k \bigr] \\&= {\bf \Phi}\bigl(\E[\langle Y_k, \e_j^{y} \rangle Y_k]\bigr) = \lambda_j^{y} {\bf \Phi}(\e_j^{y}).
\end{align*}
Hence we obtain the alternative representation
\begin{align}
{\bf \Phi}(\cdot) = \sum_{j = 1}^{\infty} {\bf \Phi}\big(\langle \e_j^{y}, \cdot \rangle \e_j^{y} \bigr) = \sum_{j = 1}^{\infty} \frac{\lambda_j^{y} {\bf \Phi}(\e_j^{y})}{\lambda_j^{y}}\langle \e_j^{y}, \cdot \rangle = \sum_{j = 1}^{\infty} \frac{{\bf \Upsilon}(\e_j^{y})}{\lambda_j^{y}}\langle \e_j^{y}, \cdot \rangle.
\end{align}
The advantage of this representation is that all involved quantities can be estimated. Given a truncation parameter $b \in \N$, this motivates the estimate
\begin{align}\label{defn_est_flm}
\widehat{{\bf \Phi}}^b(\cdot) = \sum_{j = 1}^{b} \frac{1}{n} \sum_{k = 1}^n \frac{\langle Y_k, \widehat{\e}_j^{y} \rangle X_k }{\widehat{\lambda}_j^{y}}\langle \widehat{\e}_j^{y}, \cdot \rangle, \quad \text{$b = b_n \to \infty$ as $n$ increases.}
\end{align}
In special cases, it is known that (a version of) $\widehat{{\bf \Phi}}^b$ is sharp minimax optimal (cf. ~\cite{meister_2011}), and adaptive in slightly more general situations (cf. ~\cite{comte_johannes_2012}). The construction of $\widehat{{\bf \Phi}}^b$ illustrates the necessity of an accurate control of $\widehat{\lambdav}^y$ and $\widehat{\bf \e}^{y}$. We remark that Proposition \ref{prop_transf_abstract_to_cov} is very useful in this context. Not only can it be used to obtain precise bounds for prediction errors or the actual estimation error $\|\widehat{{\bf \Phi}}^b -  {{\bf \Phi}}\|_{\cal L}$ itself, but also for deriving various limit theorems for functions of $\widehat{{\bf \Phi}}^b$, which requires exact expansions. Limit theorems in turn are required for goodness of fit tests or the construction of confidence sets.

Let us now consider the setup where $Y_k = X_{k-1}$, which is exactly the case of an ARH(1) process. Note that for $p\in \N$ finite any ARH(p) process can be reformulated as an ARH(1) process by changing the underlying Hilbert space, see ~\cite{bosq_2000} for details. Below in Corollary \ref{cor_AHR(1)}, we provide simple yet general conditions that imply the validity of Proposition \ref{prop_transf_abstract_to_cov} for ARH(1)-processes. In view of the discussion about the convexity condition in \eqref{condi_convex} leading to \eqref{eq_convex_condi}, providing a general and simple condition, we only touch on the validity of \hyperref[A1C]{\AoneC}. Regarding the operator ${\bf \Phi}$, we assume that it possesses the spectral decomposition
\begin{align}\label{defn_struct_condition_ARH(1)}
{\bf \Phi}(\cdot) = \sum_{j = 1}^{\infty} \lambda_j^{\phi} \langle \e_j^{\phi}, \cdot \rangle \e_j^{\phi}, \quad \sum_{j = 1}^{\infty} \lambda_j^{\phi} < 1,
\end{align}
with eigenvalues $\lambdav^{\phi}$ and eigenfunctions ${\bf \e}^{\phi}$. In the sequel, let ${\bf \Theta}$ be any operator with eigenvalues $\lambdav^{\theta}$ and eigenfunctions ${\bf \e}^{\theta}$ satisfying the spectral decomposition
\begin{align}\label{defn_struct_condition_ARH(1)_Theta}
{\bf \Theta}(\cdot) = \sum_{j = 1}^{\infty} \lambda_j^{\theta} \langle \e_j^{\theta}, \cdot \rangle \e_j^{\theta}, \quad \sum_{j = 1}^{\infty} \lambda_j^{\theta} < \infty.
\end{align}
Natural candidates for ${\bf \Theta}$ in our framework are of course the operators $\boldsymbol{\cal C}_h^*\boldsymbol{\cal C}_h$ or $\boldsymbol{\cal \GG}^b$. We have the associated usual decomposition of $X_k$, given as
\begin{align*}
X_k = \sum_{j = 1}^{\infty} \sqrt{\widetilde{\lambda}_j^{\theta}} \eta_{k,j}^{\theta} \e_j^{\theta}, \,\,k \in \Z, \quad \text{$\widetilde{\lambda}_j^{\theta} = \E\bigl[\langle X_k, \e_j^{\theta} \rangle^2\bigr]$, $\eta_{k,j}^{\theta} = \langle \overline{X}_k,\e_j^{\theta} \rangle (\widetilde{\lambda}_j^{\theta})^{-1/2}$}.
\end{align*}
Similarly, denote with ${\boldsymbol{\cal C}}^{\epsilon}$ the covariance operator of $\epsilon_k$ with eigenvalues $\lambdav^{\epsilon}$ and eigenfunctions ${\bf e}^{\epsilon}$, and consider the decomposition $\epsilon_k = \sum_{j = 1}^{\infty} \sqrt{{\lambda}_j^{\epsilon}} \epsilon_{k,j} \e_j^{\epsilon}$, $k \in \Z$. We make the following distributional assumption for $\epsilon_k$. Given $q \geq 1$, there exists a $q' \geq q$ and a constant $C_q > 0$ such that
\begin{align}\label{eq_epsilon_regularity_condition}
\text{$\forall x \in \Ln^2({\cal T})$ with $\|x\|_{\Ln^2} = 1$ it holds that $\bigl\|\langle \epsilon_k, x \rangle\bigr\|_q^{2q} \leq C_q \bigl(\bigl\| \langle \epsilon_k, x \rangle \bigr\|_2^{2}\bigr)^{q'}$}.
\end{align}
Condition \eqref{eq_epsilon_regularity_condition} is mild and allows for a certain invariance in or results, see below for more details. A general example satisfying \eqref{eq_epsilon_regularity_condition} with $q' = q$ is the following. Suppose that for each fixed $k \in \Z$, $\{\epsilon_{k,j}\}_{j \in \N}$ forms a martingale difference sequence with respect to some filtration $\F_{k,j}^{\epsilon}$. Elementary calculations together with Burkholders inequality then yield the validity of \eqref{eq_epsilon_regularity_condition}. Note that since the scores of a covariance operator always have zero correlation, demanding an underlying martingale structure is a reasonable assumption. Observe that in the Gaussian case, we even have that $\{\epsilon_{k,j}\}_{j \in \N}$ is IID, which is a common assumption in the literature. Next, recall the notion of weak dependence introduced in Section \ref{sec_lag_operator}. We then have the following result.
\begin{proposition}\label{prop_ARH(1)}
Assume that ${\bf \Phi}$, ${\bf \Theta}$ satisfy representations \eqref{defn_struct_condition_ARH(1)}, \eqref{defn_struct_condition_ARH(1)_Theta}. If $\E\bigl[\|\epsilon_k\|_{\Ln^2}]<\infty$, then ${\bf X}$ is a stationary Bernoulli-shift process which can be written as $X_k = \sum_{i = 0}^{\infty} {\bf \Phi}^{i}(\epsilon_{k-i})$. If in addition $\{\epsilon_k\}_{k \in \Z}$ satisfies \eqref{eq_epsilon_regularity_condition} for some $2 \leq q \leq q'$, then
\begin{align}\label{eq_ARH(1)_geo_contraction_dep}
\max_{j \in \N}\bigl\|\eta_{k,j}^{\theta} - (\eta_{k,j}^{\theta})'\bigr\|_q \lesssim \rho^k, \quad 0 < \rho < 1, \, k \in \N.
\end{align}
\end{proposition}

Note that the geometric contraction property in \eqref{eq_ARH(1)_geo_contraction_dep} is independent of the underlying orthonormal basis ${\bf \e}^{\theta}$, which is a desirable property. A check of the proof reveals that this essentially follows from condition \eqref{eq_epsilon_regularity_condition}. We also remark that Proposition \ref{prop_ARH(1)} can be extended to more general ARH(p)-processes using the same method as in ~\cite{bosq_2000}.

Denote with $\boldsymbol{\cal C}^x$ the covariance operator of ${\bf X}$, and let ${\bf \Theta} = \boldsymbol{\cal C}^x$. We then obtain the following result.
\begin{corollary}\label{cor_AHR(1)}
Grant the assumptions of Proposition \ref{prop_ARH(1)} and let ${\bf \Theta} = \boldsymbol{\cal C}^x$. Then there exists a universal constant $C^{\boldsymbol{\cal C}}$ and universal sequence $s_n^{\boldsymbol{\cal C}} \lesssim n^{-1/4}$ such that \hyperref[A1C]{\AoneC} holds.
\end{corollary}
A related result can be established for ${\bf \Theta} = \boldsymbol{\cal \GG}^b$, we omit the details.
\begin{comment}
Summarizing, we see that Corollary \ref{cor_AHR(1)} provides simple, yet general conditions for the validity of the asymptotic expansions of $\widehat{\lambda}_j$ and $\widehat{\bf \e}_j$ in the context of ARH(1).
\end{comment}

\subsection{Weak and long memory in econometric and financial timer series}\label{sec_long_mem}

In the presence of serial dependence, the covariance operator ${\boldsymbol{\cal C}}$ as a single object is not so relevant in the context of a CLT, and the long-run operator ${\boldsymbol{\cal G}}$ is the key object. However, this can be entirely different if only serial dependence is present, but essentially no serial correlation, which is often the case in
financial or econometric time series. More recently, there has been considerable activity (see for instance ~\cite{benko_heardle_kneip_2009}, ~\cite{gabrys_2103} and particularly ~\cite{mueller_2011}) to model financial or econometric time series with the help of FPCA. In this context, it is well-known (cf. ~\cite{bollerslev_garch}), that (differenced) stock returns often display a martingale like behavior, which forms the basis for many financial discrete time models (e.g. GARCH) and continuous time models (e.g. semimartingales). On the other hand, it is equally known that the absolute or squared returns display a completely different behavior, and sometimes even exhibit long memory (cf. ~\cite{Ding199383}). As a general example, let us consider the case where $\{\epsilon_k\}_{k \in \Z}$ is an IID sequence in $\Ln^2(\mathcal{T})$, $\{X_k\}_{k \in \Z}$, $\{Y_k\}_{k \in \Z} \in \Ln^2(\mathcal{T})$ are stationary and satisfy the structural equation
\begin{align}\label{eq_struct_equation}
X_k = \epsilon_k Y_{k-1}, \quad k \in \Z, \quad \text{$Y_k \in \mathcal{E}_k$ with $\mathcal{E}_k = \sigma\bigl(\epsilon_j, \, j \leq k\bigr)$.}
\end{align}
Note that the GARCH-model is a special case of \eqref{eq_struct_equation}, see also Example 2.4 in ~\cite{hoermann_2010}. Observe that $X_k$ is a martingale difference sequence with respect to $\mathcal{E}_k$. On the other hand, $X_k^2$ (or $|X_k|$) can behave completely differently due to $\{Y_k\}_{k \in \Z}$, as is desired from a modelling perspective. This becomes relevant for the estimator $\widehat{\boldsymbol{\cal C}}$. While we still have by the martingale CLT (up to mild regularity conditions)
\begin{align*}
n^{-1/2}\sum_{k = 1}^n X_k \xrightarrow{w} \Gaussian\bigl(0, {\boldsymbol{\cal C}}\bigr),
\end{align*}
the standard estimator $\widehat{\boldsymbol{\cal C}}$ as in \eqref{defn_cov_est} in contrast is based on $X_k^2$. Depending on the behavior of $\{Y_k\}_{k \in \Z}$, we may thus witness the full palette of dependence when employing $\widehat{\boldsymbol{\cal C}}$, ranging from independence to weak dependence or even a long memory behavior of $X_k^2$. Due to the high degree of flexibility in \hyperref[A1C]{\AoneC}, our results thus provide the necessary tools for a more detailed analysis of the model in \eqref{eq_struct_equation}.

\section{Proofs of asymptotic expansions}\label{sec_proof_exp}

We introduce the following additional notation. Given functions $f,g \in \Ln^2\bigl(\TT\bigr)$ and a kernel $\mathbf{K}(r,s)$, we write
\begin{align}
\int_{\TT} f g = \int_{\TT} f(r) g(r) dr \quad \text{and} \quad \int_{\TT^2} \mathbf{K} f g = \int_{\TT^2} \mathbf{K}(r,s) f(r) g(s) dr \, ds.
\end{align}
If we have $f = g$, then we write $f^2 = f(r)^2$ and otherwise $f f = f(r) f(s)$ in the above notation. We interchangeably use $\langle \cdot, \cdot \rangle$ and $\int_{{\cal T}} \cdot $, the latter being more convenient when dealing with kernels. We also frequently apply Fubini-Tonelli without mentioning it any further. Next, we introduce the empirical kernel $\widehat{\mathbf{\DD}}$ and its analogue deterministic version ${\mathbf{\DD}}$ as
\begin{align}\nonumber
\widehat{\mathbf{\DD}} &= \widehat{\mathbf{\DD}}\bigl(r,s\bigr) = \sum_{i,j = 1}^{\infty} \sqrt{\widetilde{\lambda}_i \widetilde{\lambda}_j} \bigl(\etav_{i,j}^{\boldsymbol{\cal \DD}} + \etav_{i,j}^{\boldsymbol{\cal R}}\bigr) \e_i(r)\e_j(s)\quad \text{(note: $\widehat{\boldsymbol{\cal \DD}}(f) = \int_{\TT} \widehat{\mathbf{\DD}} f $),}\\
\mathbf{\DD} &= \mathbf{D}\bigl(r,s\bigr) = \sum_{j = 1}^{\infty} \widetilde{\lambda}_j \E\bigl[\etav_{j,j}^{\boldsymbol{\cal \DD}}\bigr] \e_j(r)\e_j(s), \quad \text{(note: $\boldsymbol{\cal D}(f) = \int_{\TT} \mathbf{D} f $).}
\end{align}

We first establish the transfer result of Proposition \ref{prop_transf_abstract_to_cov}.
\begin{proof}[Proof of Proposition \ref{prop_transf_abstract_to_cov}]
Due to $\E\bigl[\|X_k\|_{\Ln^2}^2\bigr] < \infty$, standard arguments (cf. ~\cite{hoermann_JRSS_2015}) reveal that $\boldsymbol{\cal C}$ exists and satisfies \eqref{defn_D_1} and \eqref{defn_D_2} with eigenvalues $\lambdav$ and eigenfunctions ${\bf \e}$. Moreover, we have that ${\boldsymbol{\cal C}}$ is of trace class. Since $\mm = n$, by virtue of \hyperref[A2C]{\AtwoC} and since $\E\bigl[\eta_{k,j}^2\bigr] = 1$ for $j \in \N$, we only need to verify \hyperref[A1D]{\AoneD}. Due to \hyperref[A1C]{\AoneC}, it suffices to establish a bound for $\|\etav_{i,j}^{\boldsymbol{\cal R}}\|_q$. However, using \eqref{relations_eta_C}, Cauchy-Schwarz and \hyperref[A1C]{\AoneC}, the claim follows.

\end{proof}

We now turn to the proofs of Theorems \ref{theorem_exp_eigen_value} and \ref{theorem_exp_eigen_vector}, which are developed in a series of lemmas. As starting point, we recall the following elementary preliminary result (cf. ~\cite{bosq_2000}).

\begin{comment}
\begin{align}\nonumber
\int_{\TT^2}\biggr(X_k(r) X_k(s) &- \E\bigl[X_k(r) X_k(s)\bigr]\biggl) \e_i(r) \e_j(s)dr\,ds \\&=   \nonumber \sqrt{\lambda_i \lambda_j}  \int_{\TT^2}\e_i^2(r) \e_j^2(s)dr\,ds \biggl(\eta_{k,i} \eta_{k,j} - \E\bigl[\eta_{k,i} \eta_{k,j}\bigr] \biggr) \\&=  \sqrt{\lambda_i \lambda_j} \biggl(\eta_{k,i} \eta_{k,j} - \E\bigl[\eta_{k,i} \eta_{k,j}\bigr] \biggr)
\end{align}
\end{comment}

\begin{lemma}\label{lem_funda_exp}
For $j \neq k$ we have the decomposition
\begin{align}\nonumber\label{eq_lambda_int_decomp}
\widehat{\lambda}_j \int_{\TT}\e_k (\widehat{\e}_j - \e_j) &= \lambda_k \int_{\TT}\e_k (\widehat{\e}_j - \e_j) \\&+ \int_{\TT^2} (\widehat{\mathbf{\DD}} - {\bf \DD}) {\e}_k {\e}_j + \int_{\TT^2} (\widehat{\mathbf{\DD}} - {\bf \DD}) {\e}_k (\widehat{\e}_j - \e_j).
\end{align}
\end{lemma}

Rearranging terms, we obtain from the above that (provided $\lambda_k \neq \lambda_j$)
\begin{comment}
\begin{align}
\int_{\TT}\e_k (\widehat{\e}_j - \e_j) &= (\lambda_j - \lambda_k)^{-1}\biggl(\int_{\TT^2} (\widehat{\mathbf{C}} - {\bf C}) {\e}_k \widehat{\e}_j + (\widehat{\lambda}_j - \lambda_j)\int_{\TT}\e_k (\widehat{\e}_j - \e_j) \biggr)\\&= \nonumber (\lambda_j - \lambda_k)^{-1}\biggl(\int_{\TT^2} (\widehat{\mathbf{C}} - {\bf C}) {\e}_k {\e}_j + \int_{\TT^2} (\widehat{\mathbf{C}} - {\bf C}) {\e}_k (\widehat{\e}_j - \e_j) + (\widehat{\lambda}_j - \lambda_j)\int_{\TT}\e_k (\widehat{\e}_j - \e_j) \biggr)\\&= (\lambda_j - \lambda_k)^{-1}\biggl(I_{k,j} + II_{k,j} + III_{k,j}\biggr),
\end{align}
\end{comment}

\begin{align}\label{eq_decomp_funda}\nonumber
\int_{\TT}\e_k (\widehat{\e}_j - \e_j) &= \frac{1}{\lambda_j - \lambda_k}\biggl(\int_{\TT^2} (\widehat{\mathbf{\DD}} - {\bf \DD}) {\e}_k {\e}_j \\&\nonumber+ \int_{\TT^2} (\widehat{\mathbf{\DD}} - {\bf \DD}) {\e}_k (\widehat{\e}_j - \e_j) - (\widehat{\lambda}_j - \lambda_j)\int_{\TT}\e_k (\widehat{\e}_j - \e_j) \biggr)\\&\stackrel{def}{=} \frac{1}{\lambda_j - \lambda_k}\biggl(I_{k,j} + II_{k,j} + III_{k,j}\biggr),
\end{align}
and
\begin{align}\label{eq_lambda_int_decomp_rearr}
\int_{\TT}\e_k (\widehat{\e}_j - \e_j) &= \frac{-\lambda_k + \lambda_j}{\widehat{\lambda}_j  - \lambda_j + \lambda_j - \lambda_k} \frac{1}{\lambda_j - \lambda_k} \biggl(I_{k,j} + II_{k,j}\biggr).
\end{align}
Due to the frequent use of relations \eqref{eq_decomp_funda} and \eqref{eq_lambda_int_decomp_rearr}, it is convenient to use the abbreviation
%$n II_{k,j} = \sum_{i = 1}^{\infty} \sqrt{\lambda_k \lambda_i} E_{i,j} \sum_{u=1}^n \etav_u(k,i)$
\begin{align*}
E_{k,j} = \int_{\TT}\e_k (\widehat{\e}_j - \e_j) = \langle \e_k, \widehat{\e}_j - \e_j \rangle
\end{align*}
in the sequel. We also recall the following lemma (cf. ~\cite{bosq_2000}).
\begin{comment}
We start with the decomposition
\begin{align}\label{eq_decomp_Cn}
n \int_{\TT^2} (\widehat{\mathbf{C}} - {\bf C}) \e_i {\e}_j = \sqrt{\lambda_i \lambda_j} \sum_{k = 1}^n \bigl(\eta_{k,i} \eta_{k,j} - \E[\eta_{k,i}\eta_{k,j}]\bigr),
\end{align}
which is a straightforward consequence of the orthogonality of $\e_j,\e_k$.
\end{comment}

\begin{lemma}\label{lem_eigen_eigen_product_to_norm}
For any $j \in \N$ we have
\begin{align*}
\int_{\TT} (\widehat{\e}_j - \e_j) \widehat{\e}_j = \frac{1}{2}\bigl\|\widehat{\e}_j - \e_j\bigr\|_{\Ln^2}^2 \quad \text{and}\quad \int_{\TT} (\widehat{\e}_j - \e_j) \e_j = -\frac{1}{2}\bigl\|\widehat{\e}_j - \e_j\bigr\|_{\Ln^2}^2.
\end{align*}
\end{lemma}

We proceed by deriving subsequent bounds for $I_{k,j}, II_{k,j}$ and $III_{k,j}$.
\begin{lemma}\label{lem_bound_I}
Assume that Assumption \ref{ass_abstract} holds. Then for $1 \leq q \leq p2^{\pd +4}$ we have
\begin{align*}
\bigl\|I_{k,j}\bigr\|_q \lesssim \mm^{-1/2}\sqrt{\lambda_k \lambda_j} \quad \text{uniformly for $k,j \in \N$.}
\end{align*}
\end{lemma}

\begin{proof}[Proof of Lemma \ref{lem_bound_I}]
Using the orthogonality of $\e_j,\e_k$ we have
\begin{align*}
I_{k,j} = \int_{\TT^2} (\widehat{\mathbf{\DD}} - {\bf \DD}) \e_k \e_j = \mm^{-1/2} \sqrt{\widetilde{\lambda}_k \widetilde{\lambda}_j} \mm^{1/2}\bigl(\overline{\etav}_{k,j}^{\boldsymbol{\cal D}} + \etav_{k,j}^{\boldsymbol{\cal R}} \bigr),
\end{align*}
hence the claim follows from \hyperref[A1D]{\AoneD}, Lemma \ref{lem_E_eta_ij_is_zero} and \hyperref[A3D]{\AthreeD}.
 \end{proof}

\begin{lemma}\label{lem_bound_C_op}
Assume that Assumption \ref{ass_abstract} holds. Then for $1 \leq q \leq p2^{\pd +3}$ we have
\begin{align*}
\bigl\|\|\widehat{\boldsymbol{\cal \DD}} - {\boldsymbol{\cal \DD}}\|_{\cal L}\bigr\|_q  \lesssim  \mm^{-1/2}.
\end{align*}
\end{lemma}

\begin{proof}[Proof of Lemma \ref{lem_bound_C_op}]
Since the Hilbert-Schmidt norm dominates the Operator norm, Parsevals idendtiy and Lemma \ref{lem_bound_I} yield the claim, using that \hyperref[A3D]{\AthreeD} supplies $\sum_{j = 1}^{\infty}\lambda_j < \infty$.
 \end{proof}

\begin{lemma}\label{lem_bound_II}
Assume that Assumption \ref{ass_abstract} holds. Then for $1 \leq q \leq p2^{\pd +4}$ and $k \in \N$ we have
\begin{align*}
\biggl\|\max_{1 \leq j \leq \JJ_{\mm}^+}\frac{|II_{k,j}|}{\|\widehat{e}_j - e_j\|_{\Ln^2}}\biggr\|_{q} \lesssim \sqrt{\lambda_k} \mm^{-1/2}.
\end{align*}
\end{lemma}

\begin{proof}[Proof of Lemma \ref{lem_bound_II}]
It holds that
\begin{align}\label{eq_lem_bound_II_1}
II_{k,j} = \int_{\TT^2} (\widehat{\mathbf{D}} - {\bf D}) \e_k(\widehat{\e}_j - \e_j) = \sum_{i = 1}^{\infty} \sqrt{\widetilde{\lambda}_k \widetilde{\lambda}_i} \bigl(\overline{\etav}_{k,i}^{\boldsymbol{\cal D}} + \etav_{k,i}^{\boldsymbol{\cal R}} \bigr)E_{i,j}.
\end{align}
Since $\sum_{i = 1}^{\infty} E_{i,j}^2 = \|\widehat{\e}_j - \e_j\|_{\Ln^2}^2$ by Parsevals identity, the Cauchy-Schwarz inequality gives
\begin{align}
\biggl|\sum_{i = 1}^{\infty} \sqrt{\widetilde{\lambda}_i} E_{i,j} \bigl(\overline{\etav}_{k,i}^{\boldsymbol{\cal \DD}} + \etav_{k,i}^{\boldsymbol{\cal R}} \bigr)\biggr| \leq \biggl(\sum_{i = 1}^{\infty} \widetilde{\lambda}_i \bigl(\overline{\etav}_{k,i}^{\boldsymbol{\cal \DD}} + \etav_{k,i}^{\boldsymbol{\cal R}} \bigr)^2 \biggr)^{1/2} \bigl\|\widehat{\e}_j - \e_j\bigr\|_{\Ln^2}.
\end{align}
Hence the triangle inequality, \hyperref[A1D]{\AoneD} and Lemma \ref{lem_E_eta_ij_is_zero} together with \hyperref[A3D]{\AthreeD} yield
\begin{align*}
\biggl\|\max_{1 \leq j \leq \JJ_{\mm}^+}\frac{|II_{k,j}|}{\|\widehat{e}_j - e_j\|_{\Ln^2}}\biggr\|_{q} &\leq \mm^{-1/2}\sqrt{\widetilde{\lambda}_k} \biggl(\sum_{i = 1}^{\infty} \widetilde{\lambda}_i \mm \bigl\|(\overline{\etav}_{k,i}^{\boldsymbol{\cal D}} + \etav_{k,i}^{\boldsymbol{\cal R}})^2\bigr\|_{q/2}\biggr)^{1/2} \\&\lesssim \mm^{-1/2}\sqrt{\widetilde{\lambda}_k} \lesssim \mm^{-1/2}\sqrt{\lambda_k}.
\end{align*}

\end{proof}

\begin{lemma}\label{lem_bound_spacings_deviation}
Assume that Assumption \ref{ass_abstract} holds, and let $\A_j = \bigl\{|\widehat{\lambda}_j - \lambda_j| \leq \psi_j/2\bigr\}$. Then
\begin{align*}
\max_{1 \leq j < \JJ_{\mm}^+}P\bigl(\A_j^c\bigr) \lesssim \mm^{-\ad p 2^{\pd +4}}.
\end{align*}
\end{lemma}

\begin{proof}[Proof of Lemma \ref{lem_bound_spacings_deviation}]
%anderer beweis: so wie in mas preprint, zweiter fall mit min max + trunkierung: \tau groÃŸ genug sodass D_{\almbda}(k,k) << \pis_j for k \geq \tau
Proceeding as in Lemma E.2 and E.1 in the supplement of ~\cite{mas_hilgert_2013} (or likewise Lemma 18, Lemma 16 in \cite{mas_complex_2014}), it follows that for some absolute constant $C > 0$
\begin{align*}
P\bigl(\A_j^c\bigr) \lesssim P\biggl(\sum_{\substack{k,l = 1\\k,l \neq j}}^{\infty} \frac{I_{k,l}^2}{|\lambda_k - \lambda_{j}||\lambda_l - \lambda_{j}|} +  \frac{I_{j,j}^2}{\psi_j^2} + \sum_{\substack{k = 1\\k \neq j}}^{\infty}\frac{I_{k,j}^2}{|\lambda_k - \lambda_j|\psi_j}\geq C \biggr).
\end{align*}
%comment on dependence structure: Lemma 16 does not require any special conditions
Let $p^* = p 2^{\pd +4}$. Then by the triangle inequality and Lemma \ref{lem_bound_I}
\begin{align}\label{eq_lem_bound_spacings_deviation_1}
\max_{1 \leq j < \JJ_{\mm}^+}\biggl\|\sum_{\substack{k,l = 1\\k,l \neq j}}^{\infty} \frac{I_{k,l}^2}{|\lambda_k - \lambda_{j}||\lambda_l - \lambda_{j}|}\biggr\|_{p*/2} \lesssim \max_{1 \leq j \leq \JJ_{\mm}^+}\biggl(\frac{1}{\sqrt{\mm}}\sum_{\substack{k = 1\\ k \neq j}}^{\infty} \frac{\lambda_k}{|\lambda_k - \lambda_{j}|}\biggr)^2.
\end{align}
Similarly, we get that
\begin{align}\label{eq_lem_bound_spacings_deviation_2}
\max_{1 \leq j < \JJ_{\mm}^+}\biggl\|\frac{I_{j,j}^2}{\psi_j^2}\biggr\|_{p*/2} \lesssim \max_{1 \leq j < \JJ_{\mm}^+} \frac{\lambda_j^2}{\mm \psi_j^2} \lesssim \max_{1 \leq j \leq \JJ_{\mm}^+} \biggl(\frac{1}{\sqrt{\mm}}\sum_{\substack{k = 1\\ k \neq j}}^{\infty} \frac{\lambda_k}{|\lambda_k - \lambda_{j}|}\biggr)^2,
\end{align}
and also that
\begin{align}\nonumber \label{eq_lem_bound_spacings_deviation_3}
\max_{1 \leq j < \JJ_{\mm}^+}\biggl\|\sum_{\substack{k = 1\\k \neq j}}^{\infty}\frac{I_{k,j}^2}{|\lambda_k - \lambda_j|\psi_j}\biggr\|_{p*/2} &\lesssim \max_{1 \leq j < \JJ_{\mm}^+} \frac{\lambda_j}{\sqrt{\mm}\psi_j} \frac{1}{\sqrt{\mm}}\sum_{\substack{k = 1\\ k \neq j}}^{\infty} \frac{\lambda_k \lambda_j }{|\lambda_k - \lambda_{j}|} \\& \lesssim \max_{1 \leq j \leq \JJ_{\mm}^+} \biggl(\frac{1}{\sqrt{\mm}}\sum_{\substack{k = 1\\ k \neq j}}^{\infty} \frac{\lambda_k}{|\lambda_k - \lambda_{j}|}\biggr)^2.
\end{align}
% distinguish two cases for \psi_j
Observe that due to \hyperref[A2D]{\AtwoD}, \eqref{eq_lem_bound_spacings_deviation_1}, \eqref{eq_lem_bound_spacings_deviation_2} and \eqref{eq_lem_bound_spacings_deviation_3} are all further bounded by $\lesssim \mm^{-2\ad}$. Hence we conclude via Markov's inequality and the triangle inequality that
\begin{align*}
\max_{1 \leq j < \JJ_{\mm}^+}P\bigl(\A_j^c\bigr)\lesssim {\mm}^{-\ad p^{*}},
\end{align*}
which completes the proof.

\end{proof}

The next result is our key technical lemma.

\begin{lemma}\label{lem_bound_II_improved_single}
Assume that Assumption \ref{ass_abstract} holds. Then uniformly for $1 \leq q \leq p 2^{\pd/2 + 3}$, $k \in \N$ and $1 \leq j < \JJ_{\mm}^+$
\begin{align*}
\bigl\|II_{k,j} \ind\bigl(\A_j\bigr)\bigr\|_{q} \lesssim \frac{\sqrt{\lambda_k \lambda_j}}{\sqrt{\mm}}\biggl(\bigl\|\|\widehat{\e}_j - \e_j\|_{\Ln^2}^2\bigr\|_{2q} + \mm^{-\ad}\biggr).
\end{align*}
\end{lemma}

\begin{proof}[Proof of Lemma \ref{lem_bound_II_improved_single}]
Note first that by construction of $\A_j$, we have that
\begin{align}\label{eq_fraction}
\biggl|\frac{\lambda_j - \lambda_l}{\widehat{\lambda}_j - \lambda_j + \lambda_j - \lambda_l}\ind\bigl(\A_j\bigr)\biggr| \leq 2, \quad \text{for $l \neq j$.}
\end{align}

\begin{comment}
if k > j: clear. if k < j: clear
\end{comment}

Using the decomposition in \eqref{eq_lambda_int_decomp_rearr} and bound \eqref{eq_fraction}, we obtain that
\begin{align}\label{eq_lem_eigenvec_it_3}
\bigl|E_{l,j}\ind\bigl(\A_j\bigr)\bigr| \leq \frac{2}{|\lambda_j - \lambda_l|} \bigl(|I_{l,j}| + |II_{l,j}|\bigr) \ind\bigl(\A_j\bigr).
\end{align}
We now use a backward inductive argument. Let $p_{i} = p 2^{i}$, $\tau \geq 0$, and suppose we have uniformly for $k \in \N$
\begin{align}\label{eq_lem_eigenvec_it_4}
\bigl\|II_{k,j} \ind\bigl(\A_j\bigr)\bigr\|_{p_{i}} \lesssim \mm^{-1/2}\sqrt{\lambda_k} \bigl(\sqrt{\lambda_j} + \mm^{- \tau}\bigr) \quad \text{for some ${i} \leq \pd + 4$.}
\end{align}
Then we obtain from \eqref{eq_lem_eigenvec_it_3}, the triangle inequality and Lemma \ref{lem_bound_I} that for $l \neq j$
\begin{align}\label{eq_lem_eigenvec_it_5}
\bigl\|E_{l,j} \ind\bigl(\A_j\bigr)\bigr\|_{p_{i}} \lesssim \mm^{-1/2}\frac{\sqrt{\lambda_l}}{|\lambda_j - \lambda_l|}\biggl( \sqrt{\lambda_j} + \mm^{-\tau} \biggr).
\end{align}

\begin{comment}
\begin{align}
\bigl|II_{k,j}\bigr| \ind\bigl(\A_j\bigr) \leq \frac{\sqrt{\lambda_k}}{n}\sum_{i = 1}^{\infty} \sqrt{\lambda_i}  \bigl|E_{i,j} \etav_{k,i}\bigr|\ind\bigl(\A_j\bigr)\leq \frac{\sqrt{\lambda_k}}{n}\sum_{i = 1}^{\infty} \sqrt{\lambda_i} \bigl|\etav_{k,i}\bigr|\ind\bigl(\A_j\bigr)2 \frac{|I_{i,j} + II_{i,j}}{|\lambda_i - \lambda_j|}
\end{align}
\end{comment}
Using decomposition \eqref{eq_lem_bound_II_1}, Cauchy-Schwarz and Lemma \ref{lem_E_eta_ij_is_zero} together with \hyperref[A3D]{\AthreeD}, we get
\begin{align*}
\bigl\|II_{k,j} \ind\bigl(\A_j\bigr)\bigr\|_{p_{{i}-1}} \lesssim \sqrt{\lambda_k}\sum_{l = 1}^{\infty} \sqrt{\lambda_l}  \bigl\|E_{l,j}\ind\bigl(\A_j\bigr)\bigr\|_{p_{i}} \bigl\|\overline{\etav}_{k,l}^{\boldsymbol{\cal D}} + \etav_{k,l}^{\boldsymbol{\cal R}}\bigr\|_{p_{i}},
\end{align*}
hence we obtain from Lemma \ref{lem_eigen_eigen_product_to_norm}, inequality \eqref{eq_lem_eigenvec_it_5} and \hyperref[A1D]{\AoneD}, \hyperref[A2D]{\AtwoD} that
\begin{align}\nonumber \label{eq_lem_eigenvec_it_6}
\bigl\|II_{k,j} \ind\bigl(\A_j\bigr)\bigr\|_{p_{{i}-1}}  &\lesssim \frac{\sqrt{\lambda_k}}{\sqrt{\mm}}\biggl(\sqrt{\lambda_j}\bigl\|\|\widehat{\e}_j - \e_j\|_{\Ln^2}^2\bigr\|_{p_{i}} + \frac{1}{\sqrt{\mm}}\sum_{\substack{l = 1\\l \neq j}}^{\infty} \frac{\lambda_l \bigl(\sqrt{\lambda_j} + \mm^{-\tau}\bigr)}{|\lambda_l - \lambda_j|}\biggr) \\&\lesssim \frac{\sqrt{\lambda_k}}{\sqrt{\mm}}\biggl(\bigl(\bigl\|\|\widehat{\e}_j - \e_j\|_{\Ln^2}^2\bigr\|_{p_{i}} + \mm^{-\ad}\bigr)\sqrt{\lambda_j} +  \mm^{-\ad - \tau}\biggr),
\end{align}
and this bound holds uniformly for $k \in \N$. Observe that we have now shown the validity of relation \eqref{eq_lem_eigenvec_it_4} with the updated value $\tau = \tau + \ad$, but with respect to $p_{{i}-1}$ instead of $p_{i}$. Since $\lambda_j \gtrsim \mm^{-\hd}$ with $\hd \geq 1$, it follows that after at most $\pd/2 + 1= \lceil \hd/\ad \rceil/2 + 1$ iterations we have
\begin{align*}
\bigl\|II_{k,j} \ind\bigl(\A_j\bigr)\bigr\|_{q^*} \lesssim \frac{\sqrt{\lambda_k \lambda_j}}{\sqrt{\mm}}\biggl(\bigl\|\|\widehat{\e}_j - \e_j\|_{\Ln^2}^2\bigr\|_{2q^*} + \mm^{-\ad}\biggr),
\end{align*}
where $q^* = p 2^{\pd/2 + 3}$. By Lemma \ref{lem_bound_II}, relation \eqref{eq_lem_eigenvec_it_4} is true for $\tau = 0$ (hence $\mm^{\tau} = 1$) and ${i} = \pd +4$, constituting the basis induction step, hence the proof is complete. Note that we have also shown
\begin{align}\label{eq_lem_eigenvec_it_7}
\bigl\|E_{l,j} \ind\bigl(\A_j\bigr)\bigr\|_{q^*} \lesssim \mm^{-1/2}\frac{\sqrt{\lambda_l \lambda_j}}{|\lambda_j - \lambda_l|},
\end{align}
which is of further relevance in the sequel.
 \end{proof}

\begin{proposition}\label{prop_bound_eigen_vec_norm}
Assume that Assumption \ref{ass_abstract} holds. Then for $1 \leq q \leq p 2^{\pd/2+2}$ we have uniformly for $1 \leq j < \JJ_{\mm}^+$
\begin{align*}
\bigl\|\|\widehat{\e}_j - \e_j\|_{\Ln^2}^2\bigr\|_{q} \lesssim P\bigl(\A_j^c\bigr)^{1/q} + {\mm}^{-1} \sum_{\substack{k = 1\\k \neq j}}^{\infty} \frac{\lambda_j \lambda_k}{(\lambda_k - \lambda_j)^2}\lesssim {\mm}^{-2\ad}.
\end{align*}
\end{proposition}

\begin{proof}[Proof of Proposition \ref{prop_bound_eigen_vec_norm}]
The triangle inequality and Cauchy-Schwarz give
\begin{align}\label{eq_prop_bound_eigen_vec_norm_1}
\bigl\|\|\widehat{\e}_j - \e_j\|_{\Ln^2}^2\bigr\|_{q} \leq 2 P\bigl(\A_j^c\bigr)^{1/q} + \bigl\|\|\widehat{\e}_j - \e_j\|_{\Ln^2}^2\ind\bigl(\A_j\bigr)\bigr\|_{q}.
\end{align}
We now invoke the 'traditional' way of bounding $\|\widehat{\e}_j - \e_j\|_{\Ln^2}^2$, (cf. ~\cite{bosq_2000}, ~\cite{horvath_kokoszka_book_2012}), which uses the inequality
\begin{align}
\|\widehat{\e}_j - \e_j\|_{\Ln^2}^2 \leq 2 \sum_{\substack{k = 1\\k \neq j}}^{\infty} E_{k,j}^2.
\end{align}
Hence using \eqref{eq_lem_eigenvec_it_7} and the triangle inequality, we obtain from \hyperref[A2D]{\AtwoD} that
\begin{align*}
\bigl\|\|\widehat{\e}_j - \e_j\|_{\Ln^2}^2\ind\bigl(\A_j\bigr)\bigr\|_{q} \leq 2\sum_{\substack{k = 1\\k \neq j}}^{\infty} \bigl\|E_{k,j}^2\ind\bigl(\A_j\bigr)\bigr\|_q \lesssim \frac{1}{\mm}\sum_{\substack{k = 1\\k \neq j}}^{\infty} \frac{\lambda_l \lambda_j}{(\lambda_j - \lambda_l)^2} \lesssim {\mm}^{-2 \ad}.
\end{align*}
Combining this with \eqref{eq_prop_bound_eigen_vec_norm_1} gives the first inequality, Lemma \ref{lem_bound_spacings_deviation} and Assumption \ref{ass_abstract} yield the second part.

\end{proof}

Note that $\ad \leq 1/2$ and hence $\pd/2 \geq \hd \geq 1$ and $2^{\pd/2 + 2} \geq 8$. Since
\begin{align*}
\bigl\|\|\widehat{\e}_j - \e_j\|_{\Ln^2}^2\bigr\|_{2q} \leq \sqrt{2} \bigl\|\|\widehat{\e}_j - \e_j\|_{\Ln^2}^2\bigr\|_{q}^{1/2} \quad \text{for $q \geq 1$},
\end{align*}
we obtain the following corollary to Lemma \ref{lem_bound_II_improved_single}.
\begin{corollary}\label{corollary_bound_II}
Assume that Assumption \ref{ass_abstract} holds. Then for $1 \leq q \leq 8p$ we have uniformly for $k \in \N$ and $1 \leq j < \JJ_{\mm}^+$
\begin{align*}
\bigl\|II_{k,j}\bigr\|_{q} \lesssim \frac{\sqrt{\lambda_j \lambda_k}}{\sqrt{\mm}} {\mm}^{-\ad}.
\end{align*}
\end{corollary}

\begin{proof}[Proof of Corollary \ref{corollary_bound_II}]
Lemma \ref{lem_bound_II}, Lemma \ref{lem_bound_spacings_deviation}, Lemma \ref{lem_bound_II_improved_single} and Cauchy-Schwarz give
\begin{align*}
\bigl\|II_{k,j}\bigr\|_{q} &\leq \bigl\|II_{k,j}\ind\bigl(\A_j\bigr)\bigr\|_{q} + \bigl\|II_{k,j}\ind\bigl(\A_j^c\bigr)\bigr\|_{q} \lesssim \frac{\sqrt{\lambda_j \lambda_k}}{\sqrt{\mm}} {\mm}^{-\ad} + \frac{\sqrt{\lambda_k}}{\sqrt{\mm}} {\mm}^{-\ad} P\bigl(\A_j^c\bigr)^{1/2q} \\&\lesssim \frac{\sqrt{\lambda_j \lambda_k}}{\sqrt{\mm}} {\mm}^{-\ad} + \frac{\sqrt{\lambda_k}}{\sqrt{\mm}} {\mm}^{-\ad} {\mm}^{-\ad p 2^{\pd + 3}/q}.
\end{align*}
Since $\ad p 2^{\pd + 3}/q \geq \ad 2^{\pd} \geq \hd$, we have ${\mm}^{-\ad p 2^{\pd + 3}/q} \lesssim \lambda_{\JJ_{\mm}^+}$ by \hyperref[A2D]{\AtwoD} and the claim follows.
 \end{proof}

\begin{lemma}\label{lem_eigen_uniform_bound}
%Let $\A_{n} = \{\max_{1 \leq j < \JJ_{\mm}^+}\|\widehat{\e}_j - \e_j\|_{\Ln^2} < \sqrt{2}\}$.
Assume that Assumption \ref{ass_abstract} holds. Then for $1 \leq q \leq 4p$
\begin{align*}
\bigl\|\widehat{\lambda}_j - \lambda_j - I_{j,j}\bigr\|_q \lesssim \frac{\lambda_j}{\sqrt{\mm}} {\mm}^{-\ad}, \quad \text{and} \quad \bigl\|\widehat{\lambda}_j - \lambda_j\bigr\|_q \lesssim \frac{\lambda_j}{\sqrt{\mm}}, \quad \text{uniformly for $1 \leq j < \JJ_{\mm}^+$.}
\end{align*}
\end{lemma}

\begin{proof}[Proof of Lemma \ref{lem_eigen_uniform_bound}]
We have that
\begin{align*}
\widehat{\lambda}_j  &= \int_{\TT^2} \widehat{\mathbf{D}} \widehat{\e}_j \widehat{\e}_j = \int_{\TT^2} \widehat{\mathbf{D}} (\widehat{\e}_j - \e_j) \widehat{\e}_j + \int_{\TT^2} \widehat{\mathbf{D}} \e_j \widehat{\e}_j \\&= \widehat{\lambda}_j\int_{\TT} (\widehat{\e}_j - \e_j) \widehat{\e}_j + \int_{\TT^2} (\widehat{\mathbf{D}} - {\bf D}) {\e}_j \widehat{\e}_j + \int_{\TT^2}{\mathbf{D}} \e_j \widehat{\e}_j\\& = \frac{\widehat{\lambda}_j}{2}\bigl\|\widehat{\e}_j - \e_j\bigr\|_{\Ln^2}^2 + \int_{\TT^2} (\widehat{\mathbf{D}} - {\bf D}) \e_j(\widehat{\e}_j - \e_j) + \int_{\TT^2} (\widehat{\mathbf{D}} - {\bf D}) \e_j {\e}_j + \int_{\TT^2}{\mathbf{D}} \e_j \widehat{\e}_j.
\end{align*}
Since by Lemma \ref{lem_eigen_eigen_product_to_norm}
\begin{align*}
\int_{\TT^2}{\mathbf{D}} \e_j \widehat{\e}_j &= \int_{\TT^2}{\mathbf{D}} \e_j (\widehat{\e}_j - \e_j) + \int_{\TT^2}{\mathbf{D}} \e_j {\e}_j = -\frac{\lambda_j}{2}\bigl\|\widehat{\e}_j - \e_j\bigr\|_{\Ln^2}^2 + \lambda_j,
\end{align*}
we obtain by rearranging terms (if $\|\widehat{\e}_j - \e_j\|_{\Ln^2}^2 < 2$)
\begin{align}\label{eq_lem_eigen_uniform_b_2}\nonumber
\widehat{\lambda}_j - \lambda_j &= \frac{2}{2 - \|\widehat{\e}_j - \e_j\|_{\Ln^2}^2}\biggl(\int_{\TT^2} (\widehat{\mathbf{D}} - {\bf D}) \e_j {\e}_j + \int_{\TT^2} (\widehat{\mathbf{D}} - {\bf D}) \e_j(\widehat{\e}_j - \e_j)\biggr)\\&= \frac{2}{2 - \|\widehat{\e}_j - \e_j\|_{\Ln^2}^2}\bigl(I_{j,j} + II_{j,j}\bigr).
\end{align}
Let $\B_j = \bigl\{\|\widehat{\e}_j - \e_j\|_{\Ln^2}^2 \leq 1\bigr\}$. By Lemma \ref{lem_bound_I}, Proposition \ref{prop_bound_eigen_vec_norm} and the Cauchy-Schwarz inequality we obtain
\begin{align}\label{eq_lem_eigen_uniform_b_3}\nonumber
\biggl\|I_{j,j}\biggl(1 - \frac{2}{2 - \|\widehat{\e}_j - \e_j\|_{\Ln^2}^2}\biggr)\ind\bigl(\B_j\bigr) \biggr\|_q &\lesssim \bigl\|I_{j,j}\bigr\|_{2q} \bigl\|\|\widehat{\e}_j - \e_j\|_{\Ln^2}^2\bigr\|_{2q}\\&\lesssim \frac{\lambda_j}{\sqrt{\mm}}{\mm}^{-2\ad}.
\end{align}
Similarly, Corollary \ref{corollary_bound_II} yields that %us \| \|_2q \lesssim \|\|_q
\begin{align}\label{eq_lem_eigen_uniform_b_4}
\biggl\|II_{j,j}\biggl(1 - \frac{2}{2 - \|\widehat{\e}_j - \e_j\|_{\Ln^2}^2}\biggr)\ind\bigl(\B_j\bigr) \biggr\|_q &\lesssim \frac{\lambda_j}{\sqrt{\mm}}{\mm}^{-\ad}.
\end{align}
Let ${\cal D} = \bigl\{\bigl\|\widehat{\boldsymbol{\cal D}} - {\boldsymbol{\cal D}} \bigr\|_{{\cal L}} \leq 1\bigr\}$. Lemma \ref{lem_bound_C_op} and Markovs inequality then yield that
\begin{align}\label{eq_lem_eigen_uniform_b_5}
P\bigl({\cal D}^c \bigr) \lesssim {\mm}^{-2\ad p 2^{\pd/2+3}}.
\end{align}
On the other hand, Proposition \ref{prop_bound_eigen_vec_norm} implies that $P\bigl(\B_j^c\bigr) \lesssim {\mm}^{-2\ad p 2^{\pd/2+2}}$. Since $\hd \geq 1, 1/2 > \ad$ we have $2^{\pd/2} \geq 1/2 + 1/4\ad + \hd/2\ad$ and hence ${\mm}^{-2\ad 2^{\pd/2}} \lesssim {\mm}^{-1/2 - \ad} \lambda_{\JJ_{\mm}^+}$ by \hyperref[A2D]{\AtwoD}. Combining \eqref{eq_lem_eigen_uniform_b_2}, \eqref{eq_lem_eigen_uniform_b_3}, \eqref{eq_lem_eigen_uniform_b_4} and \eqref{eq_lem_eigen_uniform_b_5} we obtain from the Cauchy-Schwarz inequality, Lemma \ref{lem_eigen_gen_upper_bound} (see ~\cite{bosq_2000} for a general version) and Lemma \ref{lem_bound_C_op}, that
\begin{align*}
\bigl\|\widehat{\lambda}_j - \lambda_j - I_{j,j}\bigr\|_q &\lesssim P\bigl(\B_j^c)^{1/q} + \bigl\|\|\widehat{\boldsymbol{\cal D}} - {\boldsymbol{\cal D}}\|_{{\cal L}} \bigr\|_{2q} P\bigl({\cal D}^c)^{1/2q} + \frac{\lambda_j}{\sqrt{\mm}} {\mm}^{-\ad} \\&\lesssim \frac{\lambda_j}{\sqrt{\mm}} {\mm}^{-\ad},
\end{align*}
which gives the first claim. The second claim follows from Lemma \ref{lem_bound_I}.

\end{proof}

\begin{lemma}\label{lem_bound_III}
Assume that Assumption \ref{ass_abstract} holds. Then for $1 \leq q \leq 2p$ we have uniformly for $k \in \N$ and $1 \leq j < \JJ_{\mm}^+$
\begin{align*}
\bigl\|III_{k,j}\ind\bigl(\A_j\bigr)\bigr\|_q \lesssim \frac{\lambda_j}{\mm}\frac{\sqrt{\lambda_k \lambda_j}}{|\lambda_k - \lambda_j|}\lesssim \frac{\sqrt{\lambda_k \lambda_j}}{\sqrt{\mm}}{\mm}^{-\ad}.
\end{align*}
\end{lemma}

\begin{proof}[Proof of Lemma \ref{lem_bound_III}]
Recall that $III_{k,j} = \bigl(\widehat{\lambda}_j - \lambda_j\bigr)E_{k,j}$. By the Cauchy-Schwarz inequality and Lemma \ref{lem_eigen_uniform_bound}, we have that
\begin{align*}
\bigl\|III_{k,j}\ind\bigl(\A_j\bigr)\bigr\|_q \lesssim \frac{\lambda_j}{\sqrt{\mm}} \bigl\|E_{k,j}\ind\bigl(\A_j\bigr)\bigr\|_{2q}.
\end{align*}
Hence the claims follow from inequality \eqref{eq_lem_eigenvec_it_7} and \hyperref[A2D]{\AtwoD}.

 \end{proof}

For the sake of reference, we state Pisiers inequality.
\begin{lemma}\label{lem_max_mom}
Let $p \geq 1$ and $Y_{j}$, $1 \leq j \leq \J$ be a sequence of random variables. Then
\begin{align*}
\bigl\|\max_{1 \leq j \leq \J}|Y_j|\bigr\|_p \leq \biggl(\sum_{j = 1}^{\J} \bigl\|Y_j\bigr\|_p^p\biggr)^{1/p} \leq \J^{1/p} \max_{1 \leq j \leq \J}\bigl\|Y_j\bigr\|_{p}.
\end{align*}
\end{lemma}

We are now ready to proof Theorems \ref{theorem_exp_eigen_value} and \ref{theorem_exp_eigen_vector}.
\begin{proof}[Proof of Theorem \ref{theorem_exp_eigen_value}]
This readily follows from Lemma \ref{lem_eigen_uniform_bound} and Lemma \ref{lem_max_mom}.

 \end{proof}

\begin{proof}[Proof of Theorem \ref{theorem_exp_eigen_vector}]
We treat the first claim. By Lemma \ref{lem_eigen_eigen_product_to_norm} we have the decomposition
\begin{align}\label{eq_theorem_exp_eigen_vec_1}
\widehat{e}_j - \e_j = -\frac{\e_j}{2}\bigl\|\widehat{\e}_j - \e_j\bigr\|_{\Ln^2}^2 + \sum_{\substack{k = 1\\k \neq j}}^{\infty}\e_k \frac{I_{k,j} + II_{k,j} + III_{k,j}}{\lambda_j - \lambda_k} \stackrel{def}{=} -A_j + B_j.
\end{align}
Note that by the triangle inequality
\begin{align*}
\bigl\|B_j\bigr\|_{\Ln^2} \leq \bigl\|\widehat{e}_j - \e_j\bigr\|_{\Ln^2} + \bigl\|A_j\bigr\|_{\Ln^2} \leq 4.
\end{align*}
Let $C_j = \sum_{\substack{k = 1\\k \neq j}}^{\infty}\e_k \frac{I_{k,j}}{\lambda_j - \lambda_k}$. Then another application of the triangle inequality gives
\begin{align*}
\bigl\|\widehat{e}_j - \e_j + A_j - C_j\bigr\|_{\Ln^2} \leq \bigl\|B_j\bigr\|_{\Ln^2} + \bigl\|C_j\bigr\|_{\Ln^2} \leq 4 + \bigl\|C_j\bigr\|_{\Ln^2}.
\end{align*}
Hence by the Cauchy-Schwarz inequality and Lemma \ref{lem_bound_I}
\begin{align*}
\bigl\|\|\widehat{e}_j - \e_j + A_j - C_j\|_{\Ln^2}\ind\bigl(\A_j^c\bigr)\bigr\|_p \lesssim 4P\bigl(\A_j^c\bigr)^{1/p} + P\bigl(\A_j^c\bigr)^{1/2p} \biggl(\frac{1}{n}\sum_{\substack{k = 1\\k \neq j}}^{\infty} \frac{\lambda_j \lambda_k}{(\lambda_j - \lambda_k)^2}\biggr)^{1/2},
\end{align*}
which by Lemma \ref{lem_bound_spacings_deviation} and \hyperref[A2D]{\AtwoD} (arguing as in the proof of Lemma \ref{lem_eigen_uniform_bound}) is bounded by
\begin{align*}
\bigl\|\|\widehat{e}_j - \e_j + A_j - C_j\|_{\Ln^2}\ind\bigl(\A_j^c\bigr)\bigr\|_p \lesssim {\mm}^{-1/2 - \ad}\bigl( \lambda_{\J_n^+} + \sqrt{\Lambda_j}\bigr).
\end{align*}
Lemma \ref{lem_max_mom} and the inequality $\Lambda_j \geq \frac{\lambda_{j}}{\lambda_{j-1}} \gtrsim \lambda_{j} \wedge 1$ then show that it suffices to consider event $\A_j$. Corollary \ref{corollary_bound_II} and Lemma \ref{lem_bound_III} give
\begin{align*}
\biggl\|\sum_{\substack{k = 1\\k \neq j}}^{\infty} \frac{(II_{k,j} + III_{k,j})^2}{(\lambda_j - \lambda_k)^2}\ind\bigl(\A_j\bigr)\biggr\|_p \lesssim {\mm}^{-1-\ad}\sum_{\substack{k = 1\\k \neq j}}^{\infty} \frac{\lambda_j \lambda_k}{(\lambda_j - \lambda_k)^2},
\end{align*}
hence the first claim follows from Lemma \ref{lem_max_mom}. Next, we treat the second claim. As before Lemma \ref{lem_eigen_eigen_product_to_norm} yields
\begin{align*}
\bigl\|\widehat{\e}_j - \e_j\bigr\|_{\Ln^2}^2 = \frac{1}{4}\bigl\|\widehat{\e}_j - \e_j\bigr\|_{\Ln^2}^4 + \sum_{\substack{k = 1\\k \neq j}}^{\infty} \frac{(I_{k,j} + II_{k,j} + III_{k,j})^2}{(\lambda_j - \lambda_k)^2}.
\end{align*}
Proceeding as in the first claim, one shows that it suffices to consider the event $\A_j$. Let
${\cal D}_j = \bigl\{\|\widehat{\e}_j - \e_j\|_{\Ln^2}^2 \leq {\mm}^{-\ad} \bigr\}$. Then proceeding as in Lemma \ref{lem_eigen_uniform_bound} we obtain
\begin{align}\label{eq_theorem_exp_eigen_vec_2}
P\bigl({\cal D}_j^c \bigr) \lesssim {\mm}^{-\ad p 2^{\pd/2 + 2}} \lesssim {\mm}^{-p - 2\ad p} \lambda_{\J_n^+}^p.
\end{align}
We thus obtain from Lemma \ref{lem_bound_I}, Corollary \ref{corollary_bound_II}, Lemma \ref{lem_bound_III} and \eqref{eq_theorem_exp_eigen_vec_2}
\begin{align}\label{eq_theorem_exp_eigen_vec_3} \nonumber
&\biggl\|\biggl(\bigl\|\widehat{\e}_j - \e_j\bigr\|_{\Ln^2}^2 -  \sum_{\substack{k = 1\\k \neq j}}^{\infty} \frac{I_{k,j}^2}{(\lambda_j - \lambda_k)^2}\biggr)\ind\bigl(\A_j\bigr)\biggr\|_p \lesssim {\mm}^{-\ad}\biggl\|\bigl\|\widehat{\e}_j - \e_j\bigr\|_{\Ln^2}^2 \ind\bigl(\A_j\bigr)\biggr\|_p  \\&\nonumber + P\bigl({\cal D}_j^c\bigr)^{1/p} + \biggl\|\sum_{\substack{k = 1\\k \neq j}}^{\infty} \frac{(I_{k,j} + II_{k,j} + III_{k,j})^2 - I_{k,j}^2}{(\lambda_j - \lambda_k)^2} \ind\bigl(\A_j\bigr)\biggr\|_p\\& \lesssim {\mm}^{-\ad}\biggl\|\bigr\|\widehat{\e}_j - \e_j\bigr\|_{\Ln^2}^2\ind\bigl(\A_j\bigr) \biggr\|_p + {\mm}^{-1- 2 \ad}\lambda_{\J_n^+} + {\mm}^{-1-\ad}\sum_{\substack{k = 1\\k \neq j}}^{\infty} \frac{\lambda_j \lambda_k}{(\lambda_j - \lambda_k)^2}.
\end{align}
Iterating this inequality once and rearranging terms, Lemma \ref{lem_bound_I} yields that
\begin{align*}
\biggl\|\biggl(\bigl\|\widehat{\e}_j - \e_j\bigr\|_{\Ln^2}^2 -  \sum_{\substack{k = 1\\k \neq j}}^{\infty} \frac{I_{k,j}^2}{(\lambda_j - \lambda_k)^2}\biggr)\ind\bigl(\A_j\bigr)\biggr\|_p \lesssim \frac{\lambda_{\J_n^+}}{{\mm}^{1+2 \ad}} + \frac{1}{{\mm}^{1+\ad}}\sum_{\substack{k = 1\\k \neq j}}^{\infty} \frac{\lambda_j \lambda_k}{(\lambda_j - \lambda_k)^2}.
\end{align*}
Since $\Lambda_j \geq \frac{\lambda_{j}}{\lambda_{j-1}} \gtrsim \lambda_{j} \wedge 1$, an application of Lemma \ref{lem_max_mom} yields the desired result.

\end{proof}

\begin{proof}[Proof of Proposition \ref{prop_replace_I_with_eta}]
Observe that since $\E\bigl[\etav_{k,j}^{\boldsymbol{\cal D}}\bigr] = 0$ for $k \neq j$, we get that
\begin{align*}
I_{k,j} = \bigl \langle \bigl(\widehat{\boldsymbol{\cal D}} - {\boldsymbol{\cal D}}\bigr)(\e_k), \e_j \bigr \rangle = \sqrt{\widetilde{\lambda}_k \widetilde{\lambda}_j}\bigl(\overline{\etav}_{k,j}^{\boldsymbol{\cal D}} + {\etav}_{k,j}^{\boldsymbol{\cal R}} \bigr).
\end{align*}
Since $\widetilde{\lambda}_j = \lambda_j/\E\bigl[\etav_{j,j}^{\boldsymbol{\cal D}}\bigr]$, the claim follows from \hyperref[A1D]{\AoneD} and routine calculations.

\end{proof}

\begin{proof}[Proof of Corollary \ref{corollary_norms}]
The claim follows from Proposition \ref{prop_replace_I_with_eta} and \hyperref[A1D]{\AoneD}.

\end{proof}

\subsection{Proofs of Lemma \ref{lem_verify_ass_poly} and Theorem \ref{thm_gaussian_part}}\label{sec_gauss_optim_proofs}

We first provide the following result about the convexity relations of $\lambda_x$.
\begin{lemma}\label{lem_verify_ass_poly}
If \eqref{condi_convex} holds, then \eqref{eq_convex_condi} is valid.
\end{lemma}

\begin{proof}[Proof of Lemma \ref{lem_verify_ass_poly}]
For the proof, the following relations are useful, which can be found in ~\cite{cardot_mas_sarda_2007},~\cite{crambes2013}.
\begin{align}\nonumber\label{eq_convex_bound}
&\text{If $j > k$ and \eqref{condi_convex} holds, then $k \lambda_k \geq j \lambda_j$ and $\lambda_k - \lambda_j \gtrsim \bigl(1 - k/j\bigr)\lambda_k$}.\\ &\text{Moreover, it holds that $\sum_{k > j} \lambda_k \leq (j+1) \lambda_j$.}
\end{align}
%In addition, if $\sum_{j = 1}^{\infty} |a_j|  < \infty$ for some sequence $\bigl\{a_j\bigr\}_{j \in \N}$, then $\sup_{j \in \N} j a_j < \infty$.
Now by \eqref{eq_convex_bound} we have
\begin{align*}
\sum_{\substack{k =1\\k \neq j}}^{\infty}\frac{\lambda_k \lambda_j}{(\lambda_j - \lambda_k)^2} \lesssim  j^2 \sum_{j > k} \frac{\lambda_j \lambda_k }{(k - j)^2 \lambda_k^2} + \sum_{j < k}^{2 j} \frac{k^2 \lambda_j \lambda_k }{(k - j)^2 \lambda_j^2} + \sum_{2j < k} \frac{\lambda_j \lambda_k }{\lambda_j^2}\lesssim j^2.
\end{align*}
In the same manner, one shows that
\begin{align*}
\sum_{\substack{k =1\\k \neq j}}^{\infty}\frac{\lambda_k}{|\lambda_j - \lambda_k|} \lesssim j \log j.
\end{align*}
 \end{proof}

\begin{proof}[Proof of Theorem \ref{thm_gaussian_part}]
First note that due to the Gaussianity of ${\bf X}$, scores $\eta_{k,i}$ and $\eta_{k,j}$ are mutually independent for $i \neq j$. Given independent standard Gaussian random variables $X,Y$, the function $XY-1$ is a two-dimensional second degree Hermite polynomial. If $X = Y$, then $X^2-1$ is a univariate Hermite polynomial of second degree. We may now invoke Theorem 4 in ~\cite{arcones1994}. The proof is based on the method of moments for partial sums of Hermite polynomials. In particular, using that $\sup_{j \in \N}\sum_{k = 0}^{\infty}\Cov(\eta_{0,j},\eta_{k,j})^2 < \infty$ (which follows from $\alpha > 3/4$) it is shown via the Diagram formula that for any fixed $p \in \N$
\begin{align}\label{eq_arcones_prop}
\sqrt{n}\max_{1 \leq i,j \leq \infty}\bigl\|\overline{\etav}_{i,j}^{\boldsymbol{\cal C}}\bigr\|_p < \infty \quad \text{and} \quad \sqrt{n}\overline{\etav}_{i,j}^{\boldsymbol{\cal C}} \xrightarrow{w} \Gaussian\bigl(0,\sigma_{i,j}^2\bigr).
\end{align}
Moreover, since $\alpha > 3/4$ one readily shows that $\max_{j \in \N}\bigl\|n^{-3/4}\sum_{k = 1}^n \eta_{k,j}\bigr\|_{2q} = \oo(1)$ for any fixed $q \in \N$. Hence \hyperref[A1C]{\AoneC} holds and using Proposition \ref{prop_replace_I_with_eta} the CLT for $\widehat{\lambda}_j$ follows. In order to prove a CLT for $\widehat{\e}_j$ we proceed as follows. Denote with
\begin{align*}
C_j = \sum_{\substack{k = 1\\k \neq j}}^{\infty}\e_k \frac{I_{k,j}}{\lambda_j - \lambda_k}, \quad  C_{j,d} = \sum_{\substack{k = 1\\k \neq j}}^{d}\e_k \frac{I_{k,j}}{\lambda_j - \lambda_k} \quad \text{for $d > j$.}
\end{align*}
Due to Theorem \ref{theorem_exp_eigen_vector} and Lemma \ref{lem_bound_I}, we have that
%j is fixed
\begin{align*}
\sqrt{n}\biggl\|\biggr\|\frac{1}{\sqrt{\Lambda_j}
}\biggl(\widehat{\e}_j - \e_j - C_j\biggr)\biggr\|_{\Ln^2}\biggl\|_1 = \oo\bigl(1\bigr).
\end{align*}
It thus suffices to consider $C_j$. Since $\sum_{k > d} \lambda_k \to 0$ as $d$ increases, Lemma \ref{lem_bound_I} implies that for any $\delta > 0$ there exists $d_{\delta} \in \N$ such that
\begin{align}
\sqrt{n}\E\bigl[\bigl\|C_j - C_{j,d_{\delta}}\bigr\|_{\Ln^2}\bigr] \leq \delta.
\end{align}
It now suffices (cf. ~\cite{ledoux_Tala_book_2011}) to establish that for any fixed $d \in \N$ (which includes the case $d = d_{\delta}$)
\begin{align}\label{eq_thm_gaussian_part_3}
\sqrt{n}C_{j,d} \xrightarrow{w} \Gaussian\bigl(0, \Sigma_{d}\bigr),
\end{align}
where $\Sigma_{d} \in \R^d \times \R^d$ denotes the corresponding covariance matrix. But, since we have for $i_l \neq j_l$, $l \in \{1,2\}$ that
\begin{align*}
\E\bigl[\overline{\etav}_{i_1,j_1}^{\boldsymbol{\cal C}} \overline{\etav}_{i_2,j_2}^{\boldsymbol{\cal C}} \bigr] = 0 \quad \text{if either $i_1 \neq i_2$ or $j_1 \neq j_2$},
\end{align*}
we may apply Theorem 4 in ~\cite{arcones1994} due to $\sup_{j \in \N}\sum_{k = 0}^{\infty}\Cov(\eta_{0,j},\eta_{k,j})^2 < \infty$, which gives \eqref{eq_thm_gaussian_part_3}. This completes the proof.

\end{proof}

\section{Proofs of Section \ref{sec_lag_operator}}\label{sec_proof_lag_operator}

For the proof of Proposition \ref{prop_lag_operator}, we require some preliminary results.
\begin{lemma}\label{lem_wu_in_hilbert}
For $p \geq 2$, let $\{X_k\}_{k \in \Z} \in \Ln^2$ satisfy
\begin{align*}
\sum_{k = 1}^{\infty} \bigl\|\|X_k - X_k'\|_{\Ln^2}\bigr\|_p < \infty.
\end{align*}
Then
\begin{align*}
\bigl\| \|X_1 + \ldots + X_n \|_{\Ln^2}\bigr\|_p \lesssim \sqrt{n}.
\end{align*}
\end{lemma}

Lemma \ref{lem_wu_in_hilbert} comes as a byproduct of the results in ~\cite{Jirak_2013_letter}, see also Lemma \ref{lem_bound_partial_sum_gen} and ~\cite{sipwu} for the original argument for real-valued sequences, which we also use in the sequel. As a next result, we state a special type of H\"{o}ffding decomposition.
\begin{comment}
\begin{lemma}
consider mean $\sum_{1 \leq k,l \leq n} \E\bigl[X_kY_l\bigr]$... , where $Y_l$ has zero mean but $X_k$ has not!
\end{lemma}
\end{comment}
\begin{lemma}\label{lem_hoeffding_decomp_aux}
Let $\{X_k\}_{k \in \Z}, \{Y_k\}_{k \in \Z} \in \R$ be stationary such that for $p \geq 2$
\begin{align}
\sum_{k = 1}^{\infty} \bigl\|X_k - X_k'\bigr\|_{2p} < \infty, \quad \sum_{k = 1}^{\infty} \bigl\|Y_k - Y_k'\bigr\|_{2p} < \infty.
\end{align}
Denote with
\begin{align*}
A_{k} = \bigl(X_k - \E[X_k]\bigr)\E[Y_1] + \bigl(Y_k - \E[Y_k]\bigr)\E\bigl[X_1\bigr].
\end{align*}
Then
\begin{enumerate} \setlength\itemsep{.6em}
  \item[\itmi]\label{itm_(i)_hoeffding} $\bigl\|\sum_{1 \leq k,l \leq n} X_k Y_l - n \sum_{k = 1}^n A_k - n^2 \E[X_1]\E[Y_1] \bigr\|_{p} \lesssim n$,
  \item[\itmii]\label{itm_(ii)_hoeffding} $\bigl\|\sum_{k = 1}^n A_k \bigr\|_{2p} \lesssim \sqrt{n}$.
\end{enumerate}
\end{lemma}

\begin{proof}[Proof of Lemma \ref{lem_hoeffding_decomp_aux}]
Using the H\"{o}ffding decomposition
\begin{align*}
\sum_{1 \leq k,l \leq n} X_k Y_l &= \sum_{1 \leq k,l \leq n} \bigl(X_k - \E[X_k]\bigr)\bigl(Y_l - \E[Y_l]\bigr) + n^2 \E\bigl[X_1\bigr] \E\bigl[Y_1\bigr]+ n \sum_{k = 1}^n A_k,
\end{align*}
claim \hyperref[itm_(i)_hoeffding]{\itmi} follows from the triangle inequality, Cauchy-Schwarz and Lemma \ref{lem_bound_partial_sum_gen}. Claim \hyperref[itm_(ii)_hoeffding]{\itmii} follows directly from Lemma \ref{lem_bound_partial_sum_gen}.

\end{proof}

\begin{proof}[Proof of Proposition \ref{prop_lag_operator}]
Let us first mention that the assumptions of Proposition \ref{prop_lag_operator} clearly imply those of Lemma \ref{lem_wu_in_hilbert} and Lemma \ref{lem_hoeffding_decomp_aux}. As another preliminary remark, observe that $\E[\|X_k\|_{\Ln^2}^{4p}] < \infty$ implies that $\boldsymbol{\cal C}_h$ exists and $\overline{X}_k = \sum_{j = 1}^{\infty} \widetilde{\lambda}_j^{1/2} \eta_{k,j} \e_j$ with $\sum_{j = 1}^{\infty} \widetilde{\lambda}_j < \infty$. Next, denote with
\begin{align}\label{eq_prop_lag_operator_0}
\etav_{i,j}^{\boldsymbol{\cal R}} =  -\bigl(\widetilde{\lambda}_i \widetilde{\lambda}_j\bigr)^{-1/2} \langle \widehat{\boldsymbol{\cal \DD}}(\e_i), \e_j \rangle + \etav_{i,j}^{\boldsymbol{\cal \DD}}, \quad i,j \in \N.
\end{align}
Employing Lemma \ref{lem_wu_in_hilbert}, lengthy routine calculations reveal that (here condition $\bd > 3/2$ is helpful)
\begin{align}\label{eq_prop_lag_operator_1}\nonumber
&\E\biggl[\biggl\|\sum_{1 \leq k,l \leq n-h} \langle X_{l+h} - \bar{X}_n, X_{k+h} - \bar{X}_n \rangle \langle X_k- \bar{X}_n, \cdot \rangle \bigl(X_l - \bar{X}_n\bigr) \\&- \sum_{1 \leq k,l \leq n-h} \langle X_{l+h} - \mu, X_{k+h} - \mu \rangle \langle X_k- \mu, \cdot \rangle \bigl(X_l - \mu\bigr) \biggr\|_{\Ln^2}^{p} \biggr] \lesssim n^p,
\end{align}
we spare the details. %Let us mention though that condition $\bd > 3/2$ can be weakened at the expense of more involved technicalities.
Observe next that we have the representation
\begin{align}\nonumber
&\sum_{1 \leq k,l \leq n-h} \langle X_{l+h} - \mu, X_{k+h} - \mu \rangle \langle X_k- \mu, \cdot \rangle \bigl(X_l - \mu\bigr) \\&=  \sum_{1 \leq k,l \leq n-h} \sum_{i,j = 1}^{\infty} \sqrt{\widetilde{\lambda}_i \widetilde{\lambda}_j} \sum_{r = 1}^{\infty} \widetilde{\lambda}_r \eta_{l+h,r}\eta_{l,i} \eta_{k+h,r} \eta_{k,j}  \langle  \e_i, \cdot \rangle \e_j.
\end{align}
From the triangle inequality and Cauchy-Schwarz, we obtain
\begin{align}\label{eq_prop_lag_operator_3}
\max_{i,r \in \N}\bigl\|\eta_{l+h,r}\eta_{l,i} - (\eta_{l+h,r}\eta_{l,i})'\bigr\|_{2p} \lesssim \Omega_{4p}(l+h) + \Omega_{4p}(l), \quad l,h \in \N.
\end{align}
Hence by \eqref{eq_prop_lag_operator_1} and Lemma \ref{lem_hoeffding_decomp_aux} \hyperref[itm_(i)_hoeffding]{\itmi} (applicable by \eqref{eq_prop_lag_operator_3}), $\sum_{r = 1}^{\infty} \widetilde{\lambda}_r < \infty$, we obtain
\begin{align*}
n^{1/2}\max_{i,j \in \N}\bigl\|\etav_{i,j}^{\boldsymbol{\cal R}}\bigr\|_p \lesssim n^{-1/2}.
\end{align*}
Finally, using Lemma \ref{lem_hoeffding_decomp_aux} \hyperref[itm_(ii)_hoeffding]{\itmii} (applicable by \eqref{eq_prop_lag_operator_3}) we get
\begin{align*}
n^{1/2}\max_{i,j \in \N}\bigl\|\overline{\etav}_{i,j}^{\boldsymbol{\cal \DD}}\bigr\|_p < \infty.
\end{align*}
Finally, we remark that the same calculations used to derive \eqref{eq_lag_operator_symmetrized} also reveal that $\E\bigl[\etav_{i,j}^{\boldsymbol{\cal \DD}}\bigr] = 0$ for $i \neq j$. Hence \eqref{defn_D_2} holds, which completes the proof.

\end{proof}
\begin{comment}
\begin{align}
\sum_{i,j = 1}^{\infty} \sqrt{\widetilde{\lambda}_i \widetilde{\lambda}_j} \sum_{l = 1}^n \sum_{k = 1}^{\infty} \widetilde{\lambda}_k \biggl(\bigl(\eta_{l+h,k}\eta_{l,i} - \E[\eta_{l+h,k}\eta_{l,i}]\bigr)\E\bigl[\eta_{l+h,k}\eta_{l,j}\bigr] + \bigl(\eta_{l+h,k}\eta_{l,j} - \E[\eta_{l+h,k}\eta_{l,j}]\bigr)\E\bigl[\eta_{l+h,k}\eta_{l,i}\bigr]  \biggr)
\end{align}
\end{comment}

\begin{proof}[Proof of Theorem \ref{thm_lag_operator_svd_f}]
Note first that an application of Lemma \ref{lem_wu_in_hilbert} together with routine calculations gives
\begin{align}\label{eq_thm_svd_f_1}
\bigl\|\|\widehat{\boldsymbol{\cal C}}_h - {\boldsymbol{\cal C}}_h \|_{\mathcal{L}}\bigr\|_{p'} \lesssim n^{-1/2}, \quad 1 \leq p' \leq p 2^{\pd + 2}.
\end{align}
%p 2^{\pd + 4}/4
Let us make the decomposition
\begin{align*}
\widehat{f}_j - f_j = \biggl(\widehat{\lambda}_j^{1/2}\widehat{f}_j - {\lambda}_j^{1/2}{f}_j + (\lambda_j^{1/2}- \widehat{\lambda}_j^{1/2}){f}_j\biggr)\biggl(\frac{1}{\lambda_j^{1/2}} + \frac{\lambda_j^{1/2} - \widehat{\lambda}_j^{1/2}}{(\widehat{\lambda}_j \lambda_j)^{1/2}}\biggr),
\end{align*}
and also
\begin{align}\nonumber \label{eq_thm_svd_f_2}
\widehat{\lambda}_j^{1/2} \widehat{f}_j - {\lambda}_j^{1/2}{f}_j&= \widehat{\boldsymbol{\cal C}}_h\bigl(\widehat{\e}_j\bigr) - {\boldsymbol{\cal C}}_h\bigl({\e}_j\bigr) \\&=  {\boldsymbol{\cal C}}_h\bigl(\widehat{\e}_j - \e_j\bigr) + \bigl(\widehat{\boldsymbol{\cal C}}_h - {\boldsymbol{\cal C}}_h\bigr)\bigl(\e_j\bigr) + \bigl(\widehat{\boldsymbol{\cal C}}_h - {\boldsymbol{\cal C}}_h\bigr)\bigl(\widehat{\e}_j - \e_j\bigr).
\end{align}
Using \eqref{eq_thm_svd_f_1}, elementary computations yield
\begin{align}\nonumber \label{eq_thm_svd_f_3}
&\bigl\|\|\widehat{\lambda}_j^{1/2} \widehat{f}_j - \lambda_j^{1/2}f_j\|_{\Ln^2}\bigr\|_{p'} \\& \nonumber \leq \bigl\|{\boldsymbol{\cal C}}_h\bigr\|_{\cal L} \bigl\|\|\widehat{\e}_j - \e_j\|_{\Ln^2}\bigr\|_{p'} + \bigl\|\|\widehat{\boldsymbol{\cal C}}_h - {\boldsymbol{\cal C}}_h\|_{\cal L }\bigr\|_{2p'}\bigl(1 + \bigl\|\|\widehat{\e}_j - \e_j\|_{\Ln^2}\bigr\|_{2q}\bigr)\\& \lesssim \bigl\|\|\widehat{\e}_j - \e_j\|_{\Ln^2}\bigr\|_{2p'} + n^{-1/2}, \quad 1 \leq p' \leq p 2^{\pd + 2}.
\end{align}
%$\bigl\|\|\widehat{\e}_j - \e_j\|_{\Ln^2}\bigr\|_{p'}$ may be infinite here. p 2^{\pd + 4}/4
Next, for $j \in \N$ consider the set $\mathcal{C}_j$ defined as
\begin{align}\label{eq_thm_svd_f_4}
\mathcal{C}_j = \bigl\{\widehat{\lambda}_j > \lambda_j/2\bigr\},  \quad P\bigl(\mathcal{C}_j^c\bigr) \lesssim n^{-2p} \quad j \in \N,
\end{align}
where the bound for $P(\mathcal{C}_j^c)$ follows from Markovs inequality and Lemma \ref{lem_eigen_uniform_bound}. Since $\|\widehat{f}_j\|_{\Ln^2} = \|{f}_j\|_{\Ln^2} = 1$, we thus obtain
\begin{align}\label{eq_thm_svd_f_5}
\bigl\|\|\widehat{f}_j - f_j\|_{\Ln^2}\ind_{\mathcal{C}_j^c}\bigr\|_{p'} \leq 2 \bigl\|\ind_{\mathcal{C}_j^c}\bigr\|_{p'} \lesssim n^{-2p/p'}, \quad p' \geq 1.
\end{align}
Similarly, since ${\boldsymbol{\cal C}}_h$ is a bounded operator, the triangle inequality, Cauchy-Schwarz, Markovs inequality, \eqref{eq_thm_svd_f_1} and \eqref{eq_thm_svd_f_4} yield for $1 \leq p' \leq p$
\begin{align*}%\nonumber \label{eq_thm_svd_f_6}
&\bigl\|\|(\widehat{\lambda}_j - \lambda_j)f_j/(2\lambda_j^{1/2}) +  {\boldsymbol{\cal C}}_h(\widehat{\e}_j - \e_j) + (\widehat{\boldsymbol{\cal C}}_h - {\boldsymbol{\cal C}}_h)(\e_j)\|_{\Ln^2}\ind_{\mathcal{C}_j^c}\bigr\|_{p'} \\& \nonumber \lesssim \bigl\|\widehat{\lambda}_j - \lambda_j\bigr\|_{2p'}\bigl\|\ind_{\mathcal{C}_j^c}\bigr\|_{2p'}/\lambda_j^{1/2}  + \bigl\|\ind_{\mathcal{C}_j^c}\bigr\|_{p'} + \bigl\|\|\widehat{\boldsymbol{\cal C}}_h - {\boldsymbol{\cal C}}_h\|_{\mathcal{L}}\bigr\|_{2p'} \bigl\|\ind_{\mathcal{C}_j^c}\bigr\|_{2p'} \\&\lesssim n^{-1/2 - p/p'}\lambda_j^{1/2} + n^{-2p/p'} + n^{-1/2 - p/p'} \lesssim n^{-3/2}.
\end{align*}
%2p' \leq 2p, mu
Multiplying with $\lambda_j^{-1/2}$, we see that it suffices to establish the claim on the set $\mathcal{C}_j$. To this end, observe that
\begin{align}\label{eq_thm_svd_f_7}
\biggl|\widehat{\lambda}_j^{1/2} - \lambda_j^{1/2} - \frac{\widehat{\lambda}_j - \lambda_j}{2\lambda_j^{1/2}}\biggr|\leq \frac{(\widehat{\lambda}_j - \lambda_j)^2}{2 \lambda_j^{3/2}}, \quad j \in \N.
\end{align}
\begin{comment}
\biggl|1/(\sqrt{\lambda_j} + \sqrt{\widehat{\lambda}_j}) - 1/(2\sqrt{\lambda_j})\biggr| = \bigl|\sqrt{\lambda_j} - \sqrt{\widehat{\lambda}_j}\bigr|/( (\sqrt{\lambda_j} + \sqrt{\widehat{\lambda}_j}) 2 \sqrt{\lambda_j} ) \leq \bigl|\sqrt{\lambda_j} - \sqrt{\widehat{\lambda}_j}\bigr|/2 (\lambda_j) \leq |\lambda_j - \widehat{\lambda}_j|/(2 \lambda_j^{3/2})
\end{comment}
Then \eqref{eq_thm_svd_f_3}, \eqref{eq_thm_svd_f_7}, Cauchy-Schwarz and Lemma \ref{lem_eigen_uniform_bound} yield
\begin{align}\nonumber \label{eq_thm_svd_f_8}
&\biggl\|\bigl\| \widehat{\lambda}_j^{1/2}\widehat{f}_j - {\lambda}_j^{1/2}{f}_j + (\lambda_j^{1/2}- \widehat{\lambda}_j^{1/2}){f}_j\bigr\|_{\Ln^2} \frac{\lambda_j^{1/2} - \widehat{\lambda}_j^{1/2}}{(\widehat{\lambda}_j \lambda_j)^{1/2}}\ind_{\mathcal{C}_j}\biggr\|_{p'}\\& \lesssim \bigl(\bigl\|\|\widehat{\e}_j - \e_j\|_{\Ln^2}\bigr\|_{4p'} + n^{-1/2}\bigr)\bigl(\lambda_j n\bigr)^{-1/2}, \quad 1 \leq p' \leq p.
\end{align}
%\bigl\|\|\widehat{\e}_j - \e_j\|_{\Ln^2}\bigr\|_{4p'} may be infinite
\begin{comment}
\begin{align*}
\biggl\| \bigl\| \widehat{\lambda}_j^{1/2}\widehat{f}_j - {\lambda}_j^{1/2}{f}_j + (\lambda_j^{1/2}- \widehat{\lambda}_j^{1/2}){f}_j\bigr\|_{\Ln^2} \frac{\lambda_j^{1/2} - \widehat{\lambda}_j^{1/2}}{(\widehat{\lambda}_j \lambda_j)^{1/2}}\ind_{\mathcal{C}_j}\biggr\|_{p'} \leq \bigl(\bigl\|\|\widehat{\e}_j - \e_j\|_{\Ln^2}\bigr\|_{6p'} + n^{-1/2}\bigr)\biggl\|\frac{|\widehat{\lambda}_j - \lambda_j|}{2\lambda_j^{3/2}} + \frac{|\widehat{\lambda}_j - \lambda_j|^2}{2\lambda_j^{5/2}}\biggr\|_{3p'/2} \lesssim \bigl(\bigl\|\|\widehat{\e}_j - \e_j\|_{\Ln^2}\bigr\|_{6p'} + n^{-1/2}\bigr) 1/\bigl(\sqrt{\lambda_j n}\bigr).
\end{align*}
\end{comment}
Using \eqref{eq_thm_svd_f_1} and \eqref{eq_thm_svd_f_2} together with Cauchy-Schwarz, \eqref{eq_thm_svd_f_7} together with Lemma \ref{lem_eigen_uniform_bound} and combining this with \eqref{eq_thm_svd_f_8}, the triangle inequality gives
\begin{align}\nonumber
&\biggl\|\biggl(\biggl\|\widehat{f}_j - f_j - \frac{(\widehat{\lambda}_j - \lambda_j)f_j}{2\lambda_j} -  \frac{{\boldsymbol{\cal C}}_h\bigl(\widehat{\e}_j - \e_j\bigr) + \bigl(\widehat{\boldsymbol{\cal C}}_h - {\boldsymbol{\cal C}}_h\bigr)\bigl(\e_j\bigr)}{\lambda_j^{1/2}}\biggr\|_{\Ln^2}\biggr)\ind_{\mathcal{C}_j} \biggr\|_{p'}\\&\lesssim  \frac{1}{\sqrt{\lambda_j n}}\bigl\|\|\widehat{\e}_j - \e_j\|_{\Ln^2}\bigr\|_{4p'} + \frac{1}{\sqrt{\lambda_j}n} +\frac{1}{n}.
\end{align}

\end{proof}

\section{Proofs of Section \ref{sec_long_run_cov}}\label{sec_proof_longrun}

\begin{proof}[Proof of Theorem \ref{thm_longrun_b_remains_valid}]
Since $\sum_{h \in \Z} \|{\boldsymbol{\cal C}}_h \|_{\mathcal{L}} < \infty$, ${\boldsymbol{\cal G}}^{b}$ exists, and by ${\boldsymbol{\cal C}}_h^* = {\boldsymbol{\cal C}}_{-h}$, ${\boldsymbol{\cal G}}^{b}$ is symmetric. Hence by the spectral theorem, \eqref{Gb_representation_1} holds. Together with \eqref{Gb_representation_2}, this gives \eqref{defn_D_1} and \eqref{defn_D_2}. It remains to derive a bound for $\etav_{i,j}^{\boldsymbol{\cal R}} = \etav_{i,j}^{\boldsymbol{\cal R}}(n)$. To this end, put $\bar{\eta}_j^b = \bar{\eta}_j^b(n) = \langle \bar{X}_n-\mu, \e_j^b \rangle (\widetilde{\lambda}_j^b)^{-1/2}$. Since $b = \oo(n)$, routine calculations then reveal the upper bound
\begin{align}\label{eq_thm_longrun_b_remains_valid_1}
\bigl\|\etav_{i,j}^{\boldsymbol{\cal R}}\bigr\|_q \lesssim \frac{1}{n} \sum_{h = 1}^b \biggl(\bigl\|\sum_{k = h+1}^n \eta_{k,i}^b \bar{\eta}_{j}^b\bigr\|_q + \bigl\|\sum_{k = h+1}^n \bar{\eta}_{i}^b \eta_{k-h,j}^b  \bigr\|_q + n\bigl\|\bar{\eta}_{i}^b \bar{\eta}_{j}^b\bigr\|_q\biggr).
\end{align}
Using Cauchy-Schwarz and \hyperref[A1b]{\Aoneb}, the claim then follows.

\end{proof}

Unfortunately, the proofs of Theorems \ref{theorem_exp_decay_long_run_eigen} and \ref{theorem_exp_decay_long_run_fun} turn out to be lengthy and technical. To explain how and why, let us briefly elaborate on the main difficulties and how they can be overcome. The main objective of course is to transfer everything to Theorem \ref{thm_longrun_b_remains_valid}. This means that we need to show that Assumption \ref{ass_eigenvaectors_longrun_exp_decay} implies Assumption \ref{ass_eigenvaectors_longrun}. But here the main problem arises. For instance, even though we can control the difference $\|{\boldsymbol{\cal \GG}}^{b} - {\boldsymbol{\cal \GG}}^{}\|_{{\cal L}}$ quite well, this is not sufficient to guarantee the validity of \hyperref[A2b]{\Atwob}. The problem here is that we need to control the whole sequence $\{\lambda_j^b\}_{j \in \N}$ with the help of $\{\lambda_j\}_{j \in \N}$, but this is impossible if $j$ is very large. Related difficulties arise for \hyperref[A1b]{\Aoneb} and \hyperref[A3b]{\Athreeb}. In order to circumvent these problems, we first work with the \textit{truncated sequence}
\begin{align*}
X_k^{\tau} = \sum_{j = 1}^{\tau} \sqrt{\widetilde{\lambda}_j} \eta_{k,j} \e_j, \quad \tau \in \N, k \in \Z.
\end{align*}
The key reason why this works is the simple fact that truncation \textit{does not change} the first $\tau$ eigenvalues and eigenfunctions. Let us elaborate on this more detailed. Define the truncated long-run covariance operator as
\begin{align}\label{eq_kernel_truncated}
\boldsymbol{\cal \GG}^{\tau}(\cdot) \stackrel{def}{=} \sum_{h \in \Z}\E\bigl[\langle \overline{X}^{\tau}_k, \cdot \rangle \overline{X}^{\tau}_{k-h}\bigr] {=} \sum_{i,j = 1}^{\tau} \sqrt{\widetilde{\lambda}_i \widetilde{\lambda}_j} \varphi_{i,j} \langle \e_i, \cdot \rangle \e_j = \sum_{j = 1}^{\tau} \lambda_j \langle \e_j, \cdot \rangle \e_j,
\end{align}
where the last equality follows from \eqref{lem_varphi_relations} with $b = \infty$ ($\varphi_{i,j}^{\infty} = \varphi_{i,j}$). Observe that this last equality implies that $\{\lambda_j\}_{1 \leq j \leq {\tau}}$ and $\{\e_j\}_{1 \leq j \leq {\tau}}$ are also eigenvalues and eigenfunctions of ${\boldsymbol{\cal \GG}}^{\tau}$. This is a key observation that we heavily use in the sequel, and therefore state as a lemma for the sake of reference.
\begin{lemma}\label{lem_truncation_lemma}
The first $\tau$ eigenvalues and eigenfunctions of the truncated covariance operator ${\boldsymbol{\cal \GG}}^{\tau}$ as in \eqref{eq_kernel_truncated} are $\{\lambda_j\}_{1 \leq j \leq {\tau}}$ and $\{\e_j\}_{1 \leq j \leq {\tau}}$.
\end{lemma}
The main strategy for the proofs of Theorems \ref{theorem_exp_decay_long_run_eigen} and \ref{theorem_exp_decay_long_run_fun} are now the following two steps.
\begin{enumerate}%[leftmargin=2cm]
\item[\Sone]\label{S1} Verify Assumption \ref{ass_eigenvaectors_longrun} for $\{X_k^{\tau}\}_{1 \leq k \leq n}$.
\item[\Stwo]\label{S2} Control the error of replacing $\{X_k\}_{1 \leq k \leq n}$ with $\{X_k^{\tau}\}_{1 \leq k \leq n}$.
\end{enumerate}
\hyperref[S1]{\Sone} will require most of our attention. In order to deal with it, we introduce the truncated version $\boldsymbol{\cal \GG}^{\diamond}$ of $\boldsymbol{\cal \GG}^{b}$, namely
\begin{align}\label{eq_rep_Gstar_kernel}
\boldsymbol{\cal \GG}^{\diamond}(\cdot) = \sum_{i,j = 1}^{\tau} \sqrt{\widetilde{\lambda}_i \widetilde{\lambda}_j} \varphi_{i,j}^{(\infty,b)} \langle \e_i, \cdot \rangle \e_j, \quad \text{where} \quad \varphi_{i,j}^{(\infty,b)} = \sum_{|h| \leq b}\E\bigl[\eta_{h,i}^{\star} \eta_{0,j}^{\star} \bigr].
\end{align}
Observe that this is a linear, symmetric Hilbert-Schmidt operator. We also denote with $\widehat{\boldsymbol{\cal \GG}}^{\diamond}$ the truncated estimator, which we define as in \eqref{defn_G_longrun_estimator} with $X_k$ replaced by $X_k^{\tau}$. Denote with $\lambda_j^{\diamond}$, $\widehat{\lambda}_j^{\diamond}$, $\e_j^{\diamond}$, $\widehat{\e}_j^{\diamond}$ and $\varphi_{j,j}^{\diamond} = \varphi_{j,j}^{(\infty,b)}$ the analogue quantities, and also put
\begin{align*}
I_{i,j}^{\diamond} = \bigl \langle \bigl(\widehat{\boldsymbol{\cal \GG}}^{\diamond} - {\boldsymbol{\cal \GG}}^{\diamond}\bigr)(\e_i^{\diamond}), \e_j^{\diamond} \bigr \rangle, \quad 1 \leq i,j \leq \tau.
\end{align*}
Obviously $\boldsymbol{\cal \GG}^{\diamond}$ has finite rank bounded by $\tau$, and hence $\lambda_j^{\diamond} = 0$ for $j > \tau$ and $\boldsymbol{ \operatorname{span}}\bigl(\{\e_j\}_{j > \tau}\bigr) \subseteq \boldsymbol{\operatorname{Ker}}\bigl(\boldsymbol{\cal \GG}^{\diamond}\bigr)$. This implies $\boldsymbol{\operatorname{Im}}\bigl(\boldsymbol{\cal \GG}^{\diamond}\bigr) \oplus \boldsymbol{\cal S}^{\diamond} = \boldsymbol{ \operatorname{span}}\bigl(\{\e_j\}_{1 \leq j \leq \tau}\bigr)$ for a linear subspace $\boldsymbol{\cal S}^{\diamond} \subseteq \boldsymbol{\operatorname{Ker}}\bigl(\boldsymbol{\cal \GG}^{\diamond}\bigr)$ such that $\operatorname{dim}\bigl(\boldsymbol{\operatorname{Im}}(\boldsymbol{\cal \GG}^{\diamond})\bigr) + \operatorname{dim}\bigl(\boldsymbol{\cal S}^{\diamond}\bigr) = \tau$, and thus we get %I can select the \e_j^{\diamond} in the Kernel appropriately
\begin{align*}
X_k^{\tau} = \sum_{j = 1}^{\tau} (\widetilde{\lambda}_j^{\diamond})^{1/2} \eta_{k,j}^{\diamond} \e_j^{\diamond}, \quad \text{where $(\widetilde{\lambda}_j^{\diamond})^{1/2}\eta_{k,j}^{\diamond} = \langle X_k^{\tau}, \e_j^{\diamond} \rangle$, $\widetilde{\lambda}_j^{\diamond} = \E\bigl[\langle X_k^{\tau}, \e_j^{\diamond} \rangle^2\bigr]$.}
\end{align*}

Throughout the remaining proofs, we make the following convention. $0 < \rho < 1$ is an absolute constant that may vary from line to line. We write
\begin{align}\label{eq_defn_tau}
\tau = n^{\td}, \quad  0 < \td < \infty,
\end{align}
and often use the expression 'for sufficiently large (but finite) $\td, C_0 > 0$', where $C_0$ only depends on $\cd^-,\cd^+,\td$ (recall $b \geq C_0 \log n$). There is no danger of 'circle arguments', we always pick $\td$ first, then $C_0$. Next, we consider a more general version of Lemma \ref{lem_eigen_gen_upper_bound} (cf. ~\cite{bhatia_mcintosh_1983},~\cite{bosq_2000}).
\begin{lemma}\label{lem_operator_comp}
Let $\boldsymbol{\cal \GG}$, $\boldsymbol{\cal H}$ be linear Hilbert-Schmidt operators with eigenvalues $\bigl\{\lambda_j^{\GG}\bigr\}_{j \in \N}$, $\bigl\{\lambda_j^{H}\bigr\}_{j \in \N}$ and eigenfunctions $\bigl\{\e_j^{\GG}\bigr\}_{j \in \N}$, $\bigl\{\e_j^{H}\bigr\}_{j \in \N}$. If $\boldsymbol{\cal H}$ is positive definit, symmetric and $\lambda_1^{H} > \ldots > \lambda_{j+1}^{H}$, then
\begin{align*}
\bigl|{\lambda}_j^{\GG} - \lambda_j^H \bigr| \leq \bigl\|{\boldsymbol{\cal \GG}} - {\boldsymbol{\cal H}} \bigr\|_{{\cal L}}, \quad \bigl\|{e}_j^G - e_j^H \bigr\|_{\Ln^2} \leq \frac{2\sqrt{2}}{\psi_j^{H}} \bigl\|{\boldsymbol{\cal \GG}} - {\boldsymbol{\cal H}} \bigr\|_{{\cal L}},
\end{align*}
where $\psi_j^{H} = \min\bigl\{\lambda_{j-1}^{H}  - \lambda_{j}^{H} , \lambda_j^{H}  - \lambda_{j+1}^{H} \bigr\}$ (with $\psi_1^{H}  = \lambda_1^{H}  - \lambda_2^{H} $).% and $\bigl\|\cdot\bigr\|_{{\cal L}}$ denotes the operator norm.
\end{lemma}

In the sequel, all operators $\boldsymbol{\cal \GG}$, $\boldsymbol{\cal H}$ in question will satisfy the conditions of Lemma \ref{lem_operator_comp}. The next lemma is our main tool box and summarizes most technical preliminary results we require in the sequel. To this end, recall the notion of $\etav_{i,j}^{(\infty,b)}$, $\etav_{i,j}^{(\infty,b,1)}$ and $\etav_{i,j}^{(\infty,b,2)}$ in \eqref{defn_eta_infty_b}.

\begin{lemma}\label{lem_long_run_tech}
Assume that Assumption \ref{ass_eigenvaectors_longrun_exp_decay} and condition \eqref{eq_defn_tau} hold. Then for $1 \leq q \leq p^*$ and sufficiently large $\td, C_0 > 0$, we have
\begin{enumerate}%[leftmargin=1cm] \setlength\itemsep{.6em}
  \item[\itmi]\label{itm_(i)_c} $\max_{i,j \in \N}\bigl\{\sqrt{n/b}\|\overline{\etav}_{i,j}^{(\infty,b)}\|_{q}, \sqrt{n}\|\overline{\etav}_{i,j}^{(\infty,b,1)}\|_{q},\sqrt{n/b}\|\overline{\etav}_{i,j}^{(\infty,b,2)}\|_{q} \bigr\}< \infty$,
  \item[\itmii]\label{itm_(ii)_c} $\max_{i,j \in \N} \bigl|\E\bigl[\eta_{k,i} \eta_{0,j} \bigr]\bigr| \lesssim \rho^{k}$, $0 < \rho < 1$,
  \item[\itmiii]\label{itm_(iii)_c} $\bigl\|{\boldsymbol{\cal \GG}}^{b} - {\boldsymbol{\cal \GG}}^{\diamond}\bigr\|_{{\cal L}}, \bigl\|\|\widehat{\boldsymbol{\cal \GG}}^{b} - \widehat{\boldsymbol{\cal \GG}}^{\diamond}\|_{{\cal L}} \bigr\|_{q} \lesssim n^{-(\cd^+-1)\td}$,
  \item[\itmiv]\label{itm_(iv)_c} $\bigl\|{\boldsymbol{\cal \GG}}^{\tau} - {\boldsymbol{\cal \GG}}^{\diamond}\bigr\|_{{\cal L}}  \lesssim \rho^{b}$, $0 < \rho < 1$,
  \item[\itmv]\label{itm_(v)_c}
    $\bigl\|\|\widehat{\boldsymbol{\cal \GG}}^{b} - {\boldsymbol{\cal \GG}}^{b}\|_{{\cal L}}\bigr\|_q, \, \bigl\|\|\widehat{\boldsymbol{\cal \GG}}^{\diamond} - {\boldsymbol{\cal \GG}}^{\diamond}\|_{{\cal L}}\bigr\|_q \lesssim \sqrt{b/n}$,
  \item[\itmvi]\label{itm_(vi)_c} $\max_{1 \leq j \leq \tau}\bigl\|\e_j - \e_j^{\diamond}\bigr\|_{\Ln^2} \lesssim \rho^{b}$, $0 < \rho < 1$,
  \item[\itmvii]\label{itm_(vii)_c} $\max_{1 \leq j \leq \JJ_n^+}\bigl\|{\e}_j^b - {\e}_j^{\diamond}\bigr\|_{\Ln^2}, \, \max_{1 \leq j \leq \JJ_n^+}\bigl\|\|\widehat{\e}_j^b - \widehat{\e}_j^{\diamond}\|_{\Ln^2}\bigr\|_q \lesssim n^{-2}\lambda_{\JJ_n^+}$,
  \item[\itmviii]\label{itm_(viii)_c}
   $\max_{1 \leq j \leq \tau}\bigl\{\lambda_j/\lambda_j^{\diamond}, \lambda_j^{\diamond}/\lambda_j\bigr\} \leq 2$, $\max_{1 \leq j \leq \tau}\bigl\{\widetilde{\lambda}_j/\widetilde{\lambda}_j^{\diamond}, \widetilde{\lambda}_j^{\diamond}/\widetilde{\lambda}_j\bigr\} \leq 2$.
  %\item[\itmviii]\label{itm_(viii)_c}
\end{enumerate}
\end{lemma}

\begin{proof}[Proof of Lemma \ref{lem_long_run_tech}]
Throughout the proofs, we frequently use representations \eqref{Gb_representation_1}, \eqref{Gb_representation_2} of ${\boldsymbol{\cal \GG}}^{b}$, $\widehat{\boldsymbol{\cal \GG}}^{b}$, and an analogue representation for ${\boldsymbol{\cal \GG}}^{\diamond}$, $\widehat{\boldsymbol{\cal \GG}}^{\diamond}$ (see \eqref{eq_rep_Gstar_kernel}). Claim \hyperref[itm_(i)_c]{\itmi} follows from \hyperref[A1star]{\Aonestar} and the results in ~\cite{wu_asymptotic_bernoulli},~\cite{han_wu_2014_max}, which are actually much more general. Claim \hyperref[itm_(ii)_c]{\itmii} can be established with the same arguments as in the proof of Lemma \ref{lem_var_cov_property}. Observe next that from \hyperref[A2star]{\Atwostar} and \hyperref[A3star]{\Athreestar} we get that
\begin{align}\label{eq_lem_long_run_tech_1}
\sum_{l > \tau} \widetilde{\lambda}_l \lesssim \sum_{l > \tau} \lambda_l \lesssim n^{-(\cd^+-1)\td}.
\end{align}
The first part of claim \hyperref[itm_(iii)_c]{\itmiii} then follows from elementary computations, \hyperref[itm_(ii)_c]{\itmii} and \eqref{eq_lem_long_run_tech_1}. For the second part, observe that by routine calculations we obtain %and Cauchy-Schwarz
\begin{align*}
\bigl\|\|\widehat{\boldsymbol{\cal \GG}}^{b} - \widehat{\boldsymbol{\cal \GG}}^{\diamond}\|_{{\cal L}} \bigr\|_{q} &\lesssim \sum_{l > \tau}^{} \widetilde{\lambda}_l \bigl(\max_{i,j \in \N} \|\overline{\etav}_{i,j}^{(\infty,b)}\|_{q} + \oo(\sqrt{b/n})\bigr) \\&+  \sum_{l > \tau} \widetilde{\lambda}_l \sum_{|h| \leq b} \max_{i,j \in \N} \bigl|\E[\eta_{h,i} \eta_{0,j}]\bigr|.
\end{align*}
Hence the claim follows from \hyperref[itm_(i)_c]{\itmi}, \hyperref[itm_(ii)_c]{\itmii} and \eqref{eq_lem_long_run_tech_1}. Claim \hyperref[itm_(iv)_c]{\itmiv} can be established as follows. Due to Lemma \ref{lem_truncation_lemma} we have $\lambda_j = \lambda_j^{\tau}$ and $\e_j = \e_j^{\tau}$ for $1 \leq j \leq \tau$. Hence from the representations in \eqref{eq_kernel_truncated} and \eqref{eq_rep_Gstar_kernel}, using Cauchy-Schwarz, \hyperref[itm_(ii)_c]{\itmii}, \hyperref[A2star]{\Atwostar} and \hyperref[A3star]{\Athreestar} we get
\begin{align*}
\bigl\|{\boldsymbol{\cal \GG}}^{\tau} - {\boldsymbol{\cal \GG}}^{\diamond}\bigr\|_{{\cal L}} \lesssim \sum_{j = 1}^{\tau} \widetilde{\lambda}_j \max_{1 \leq i,j \leq \tau}\sum_{|h| \geq b} \bigl|\E[\eta_{h,j} \eta_{0,j}]\bigr| \lesssim \rho^{b} \sum_{j = 1}^{\tau} \widetilde{\lambda}_j  \lesssim \rho^{b}.
\end{align*}
Claim \hyperref[itm_(v)_c]{\itmv} can be established as follows. For the first part, using \hyperref[itm_(i)_c]{\itmi} we get that
\begin{align*}
\bigl\|\|\widehat{\boldsymbol{\cal \GG}}^{b} - {\boldsymbol{\cal \GG}}^{b}\|_{{\cal L}}\bigr\|_q \lesssim \sum_{l = 1}^{b} \widehat{\lambda}_l \bigl(\max_{i,j \in \N}\|\overline{\etav}_{i,j}^{(\infty,b)}\|_{q} + \oo(\sqrt{b/n})\bigr) \lesssim \sqrt{b/n}.
\end{align*}
For the second part, observe that by the triangle inequality
\begin{align*}
\bigl\|\|\widehat{\boldsymbol{\cal \GG}}^{\diamond} - {\boldsymbol{\cal \GG}}^{\diamond}\|_{{\cal L}}\bigr\|_q \leq \bigl\|\|\widehat{\boldsymbol{\cal \GG}}^{\diamond} - \widehat{\boldsymbol{\cal \GG}}^{b}\|_{{\cal L}}\bigr\|_q + \bigl\|{\boldsymbol{\cal \GG}}^{\diamond} - {\boldsymbol{\cal \GG}}^{b}\bigr\|_{{\cal L}} + \bigl\|\|\widehat{\boldsymbol{\cal \GG}}^{\diamond} - \widehat{\boldsymbol{\cal \GG}}^{b}\|_{{\cal L}}\bigr\|_q.
\end{align*}
Hence the claim follows from \hyperref[itm_(iii)_c]{\itmiii} and part one. Claim \hyperref[itm_(vi)_c]{\itmvi} can be established as follows. Applying Lemma \ref{lem_truncation_lemma}, Lemma \ref{lem_operator_comp} and \hyperref[itm_(iv)_c]{\itmiv} we get
\begin{align*}
\bigl\|\e_j - \e_j^{\diamond}\bigr\|_{\Ln^2} \lesssim \bigl\|{\boldsymbol{\cal \GG}}^{\tau} - {\boldsymbol{\cal \GG}}^{\diamond}\bigr\|_{{\cal L}}/\psi_j \lesssim \rho^b/\psi_j.
\end{align*}
Due to the convexity assumption in \hyperref[A2star]{\Atwostar}, relation \eqref{eq_convex_bound} in Lemma \ref{lem_verify_ass_poly} and \hyperref[A2star]{\Atwostar} yield
\begin{align}\label{eq_lem_tech_2}
\frac{\rho^{b}}{\min_{1 \leq j \leq {\tau}}\psi_j} \lesssim \frac{\tau \rho^{b}}{\lambda_{\tau}} \lesssim \tau^{1 + \cd^-} \rho^{b} \lesssim n^{\td(1 + \cd^-)} \rho^{b}.
\end{align}
Hence for large enough $C_0 > 0$, the claim follows. In order to establish \hyperref[itm_(vii)_c]{\itmvii}, observe that by Lemma \ref{lem_operator_comp}, \hyperref[itm_(iii)_c]{\itmiii} and proceeding similarly as in \eqref{eq_lem_tech_2} we get that for large enough $\td > 0$
\begin{align*}
\bigl\|{\e}_j^b - {\e}_j^{\diamond}\bigr\|_{\Ln^2} &\lesssim \bigl\|{\boldsymbol{\cal \GG}}^{b} - {\boldsymbol{\cal \GG}}^{\diamond}\bigr\|_{{\cal L}}/\psi_j \lesssim \sum_{j > \tau} \lambda_j/\psi_j \lesssim n^{-(\cd^+-1)\td} \JJ_n^+ \lambda_{\JJ_n^+} /\lambda_{\JJ_n^+}^2 \\&\lesssim n^{-(\cd^+-1)\td + 1/2 + \cd^-} \lambda_{\JJ_n^+} \lesssim n^{-2} \lambda_{\JJ_n^+},
\end{align*}
uniformly for $1 \leq j \leq \JJ_n^+$. For the second part, we can proceed in the same way. Claim \hyperref[itm_(viii)_c]{\itmviii} can be established as follows. Note that by Lemma \ref{lem_truncation_lemma}, Lemma \ref{lem_operator_comp} and \hyperref[itm_(iv)_c]{\itmiv}
\begin{align*}
\max_{1 \leq j \leq \tau}\bigl|\lambda_j - \lambda_j^{\diamond}\bigr| \leq \bigl\|{\boldsymbol{\cal \GG}}^{\tau} - {\boldsymbol{\cal \GG}}^{\diamond}\bigr\|_{{\cal L}} \lesssim \rho^b.
\end{align*}
On the other hand, we get from \hyperref[A2star]{\Atwostar} that $\lambda_{\td} \gtrsim n^{-\cd^- \td}$. Hence we conclude that for large enough $C_0 > 0$
\begin{align*}
\min_{1 \leq j \leq \tau} \lambda_j = \lambda_{\tau} \geq 2\max_{1 \leq j \leq \tau}\bigl|\lambda_j - \lambda_j^{\diamond}\bigr|.
\end{align*}
Since $\lambda_j/\lambda_j^{\diamond} = \lambda_j - \lambda_j^{\diamond})/\lambda_j^{\diamond} + 1$ (and similarly for $\lambda_j^{\diamond}/\lambda_j$), the claim follows. For the second part, using $a^2 - b^2 = (a-b)(a+b)$ and Cauchy-Schwarz, we get that
\begin{align*}
\bigl|\widetilde{\lambda}_j^{\diamond} - \widetilde{\lambda}_j\bigr| &= \bigl|\E\bigl[\langle \overline{X}_k, \e_j^{\diamond} \rangle^2 \bigr] - \E\bigl[\langle \overline{X}_k, \e_j^{} \rangle^2 \bigr]\bigr| \\&\leq \bigl\|\langle \overline{X}_k, \e_j^{\diamond} - \e_j \rangle\bigr\|_2 \bigl\|\langle \overline{X}_k, \e_j^{\diamond} + \e_j \rangle\bigr\|_2 \\&\leq \bigl\|\e_j^{\diamond} - \e_j\bigr\|_{\Ln^2} \bigl\|\|\overline{X}_k\|_{\Ln^2}\bigr\|_2 \bigl((\lambda_j^{\diamond})^{1/2} + (\lambda_j)^{1/2}\bigr).
\end{align*}
An application of Lemma \ref{lem_truncation_lemma}, Lemma \ref{lem_operator_comp} and \hyperref[itm_(iv)_c]{\itmiv} then yields
\begin{align*}
\bigl|\widetilde{\lambda}_j^{\diamond} - \widetilde{\lambda}_j\bigr| \lesssim  \bigl\|{\boldsymbol{\cal \GG}}^{\tau} - {\boldsymbol{\cal \GG}}^{\diamond}\bigr\|_{{\cal L}}/\psi_j  \lesssim \rho^{b}/\psi_j.
\end{align*}
Using \eqref{eq_lem_tech_2} we conclude that for $C_0 > 0$ sufficiently large
\begin{align*}
\min_{1 \leq j \leq \tau} \widetilde{\lambda}_j^{} = \widetilde{\lambda}_{\tau}^{} \geq 2\bigl|\widetilde{\lambda}_j^{\diamond} - \widetilde{\lambda}_j\bigr|,
\end{align*}
and thus one readily deduces the claim.

\end{proof}

We are now ready to actually proceed with \hyperref[S1]{\Sone}.
\begin{lemma}\label{lem_Gstar_is_compact_symm_and trace_class}
Grant Assumption \ref{ass_eigenvaectors_longrun_exp_decay}. Then for sufficiently large $\td,C_0 > 0$ we have
\begin{align*}
\boldsymbol{\cal \GG}^{\diamond}(\cdot) = \sum_{j = 1}^{\tau} \lambda_j^{\diamond} \big \langle \e_j^{\diamond}, \cdot \big \rangle \e_j^{\diamond}, \quad \sum_{j = 1}^{\tau} \lambda_j^{\diamond} \leq 2C^{\boldsymbol{\cal \GG}}.
\end{align*}
\end{lemma}

\begin{proof}[Proof of Lemma \ref{lem_Gstar_is_compact_symm_and trace_class}]
By construction in \eqref{eq_rep_Gstar_kernel}, we can use Mercer's Theorem to obtain the desired decomposition. The bound for $\sum_{j = 1}^{\tau} \lambda_j^{\diamond}$ follows from Lemma \ref{lem_long_run_tech} \hyperref[itm_(viii)_c]{\itmviii} and \hyperref[A3star]{\Athreestar}.

\end{proof}
The following next three lemmas establish the validity of \hyperref[A1b]{\Aoneb}, \hyperref[A2b]{\Atwob}, and \hyperref[A3b]{\Athreeb}.

\begin{lemma}\label{lem_transf_A1}
Grant Assumption \ref{ass_eigenvaectors_longrun_exp_decay}. Then for $1 \leq q \leq p^*$ and sufficiently large $\td,C_0 > 0$ we have
\begin{enumerate} \setlength\itemsep{.6em}
  \item[\itmi]\label{itm_(i)_eta} $\max_{1 \leq i,j \leq \tau}\bigl\|\overline{\etav}_{i,j}^{\diamond} - \overline{\etav}_{i,j}^{(\infty,b)}\bigr\|_{q} \lesssim n^{-1}$,
  \item[\itmii]\label{itm_(ii)_eta} $\max_{1 \leq j \leq \tau}\bigl\|\sum_{k =1}^n(\eta_{k,j}^{\diamond} - \eta_{k,j}^{})\bigr\|_{2q} \lesssim n^{-1}$.
\end{enumerate}
Hence \hyperref[A1b]{\Aoneb} holds for $\{\overline{\etav}_{i,j}^{\diamond}\}_{1 \leq i,j \leq \tau}$ and $\{\eta_{k,j}^{\diamond}\}_{1 \leq k \leq n, \, 1 \leq j \leq \tau}$ due to Lemma \ref{lem_long_run_tech} \hyperref[itm_(i)_c]{\itmi}.
\end{lemma}

\begin{proof}[Proof of Lemma \ref{lem_transf_A1}]
We first show \hyperref[itm_(i)_eta]{\itmi}. Note that it suffices to uniformly control the distance between $\eta_{k,i}^{\diamond} \eta_{k-h,j}^{\diamond}$ and $\eta_{k,i}^{} \eta_{k-h,j}^{}$. To this end, observe that we have the decomposition
\begin{align}\nonumber \label{eq_lem_transf_A1_1}
\langle X_k^{\tau}, \e_j^{\diamond} \rangle  \langle X_{k-h}^{\tau}, \e_i^{\diamond} \rangle& =  \langle X_k^{\tau}, \e_j^{\diamond} - \e_j \rangle  \langle X_{k-h}^{\tau}, \e_i \rangle + \langle X_k^{\tau}, \e_j \rangle  \langle X_{k-h}^{\tau}, \e_i^{\diamond} - \e_i \rangle
\\&+ \langle X_k^{\tau}, \e_j^{\diamond} - \e_j \rangle  \langle X_{k-h}^{\tau}, \e_i^{\diamond} - \e_i\rangle + \langle X_k^{\tau}, \e_j \rangle  \langle X_{k-h}^{\tau}, \e_i \rangle.
\end{align}
We will deal with the error terms separately. Recall that $\overline{X} = X - \E\bigl[X\bigr]$. Applying Fubini-Tonelli, we get that
\begin{align*}
\overline{\langle X_k^{\tau}, \e_j^{\diamond} - \e_j \rangle  \langle X_{k-h}^{\tau}, \e_i^{\diamond} - \e_i\rangle} =  \int_{\mathcal{T}^2} \overline{X_k^{\tau} X_{k-h}^{\tau}} \bigl(\e_j^{\diamond} - \e_j\bigr)\bigl(\e_i^{\diamond} - \e_i\bigr).
\end{align*}
Using Cauchy-Schwarz two times, we thus obtain from the above
\begin{align*}
\biggl|\sum_{h = 1}^{b}\sum_{k = h+1}^n \frac{\overline{\langle X_k^{\tau}, \e_j^{\diamond} - \e_j \rangle  \langle X_{k-h}^{\tau}, \e_i^{\diamond} - \e_i\rangle}}{n-h}\biggr| & \leq \biggl\|\sum_{h = 1}^{b}\sum_{k = h+1}^n \frac{\overline{X_k^{\tau} X_{k-h}^{\tau}}}{n-h}\biggr\|_{\Ln^2 \times \Ln^2} \\&\times \bigl\|\e_i^{\diamond} - \e_i\bigr\|_{\Ln^2}\bigl\|\e_j^{\diamond} - \e_j\bigr\|_{\Ln^2},
\end{align*}
where $\|f\|_{\Ln^2 \times \Ln^2}^2 = \int_{\mathcal{T}^2} f(u,v) du dv $. Elementary calculations yield that
% \biggl\|\sum_{l = 1}^L \overline{V^{\diamond}_l V^{\diamond}_l}\biggr\|_{\Ln^2 \times \Ln^2}^2 = \sum_{j,i = 1}^{\infty} \widetilde{\lambda}_j \widetilde{\lambda}_i ( \sum_{l = 1}^L (...) )^2
\begin{align*}
\biggl\|\biggl\|\sum_{h = 1}^{b}\sum_{k = h+1}^n \frac{\overline{X_k^{\tau} X_{k-h}^{\tau}}}{n-h}\biggr\|_{\Ln^2 \times \Ln^2}\biggr\|_{q}^2 \leq \sum_{i,j = 1}^{\tau} \widetilde{\lambda}_i \widetilde{\lambda}_j \bigl\|\overline{\etav}_{i,j}^{(\infty,b,2)}\bigr\|_{q}^2.
\end{align*}
Since $\sum_{j = 1}^{\infty} \widetilde{\lambda}_j \leq C^{\boldsymbol{\cal \GG}} \sum_{j = 1}^{\infty} {\lambda}_j < \infty$ by \hyperref[A2star]{\Atwostar}, \eqref{lem_varphi_relations} and \hyperref[A3star]{\Athreestar}, we obtain from Lemma \ref{lem_long_run_tech} \hyperref[itm_(i)_c]{\itmi} that
\begin{align}
\biggl\|\biggl\|\sum_{h = 1}^{b}\sum_{k = h+1}^n \frac{\overline{X_k^{\tau} X_{k-h}^{\tau}}}{n-h}\biggr\|_{\Ln^2 \times \Ln^2}\biggr\|_{q}^2 \lesssim \frac{b}{n}\sum_{i,j = 1}^{\tau} \widetilde{\lambda}_i \widetilde{\lambda}_j  \lesssim \frac{b}{n}.
\end{align}
Observe that by Lemma \ref{lem_long_run_tech} \hyperref[itm_(ii)_c]{\itmii} $\max_{j \in \N}\varphi_{j,j}< \infty$. Due to \eqref{lem_varphi_relations} and Lemma \ref{lem_long_run_tech} \hyperref[itm_(viii)_c]{\itmviii} we conclude $\max_{1 \leq j \leq \tau} \lambda_j/\widetilde{\lambda}_j^{\diamond} \leq 2\max_{1 \leq j \leq \tau}\varphi_{j,j} < \infty$. From \hyperref[A2star]{\Atwostar} we thus obtain that $\max_{1 \leq j \leq \tau} (\widetilde{\lambda}_j^{\diamond})^{-1/2} \lesssim \tau^{\cd^-/2}$. Due to Lemma \ref{lem_long_run_tech} \hyperref[itm_(vi)_c]{\itmvi}, and piecing all bounds together, we get that for sufficiently large $C_0 > 0$
\begin{align}\label{eq_lem_transf_A1_5}
\max_{1 \leq i,j \leq \tau}(\widetilde{\lambda}_i^{\diamond} \widetilde{\lambda}_j^{\diamond})^{-1/2}\biggl\|\sum_{h = 1}^{b}\sum_{k = h+1}^n \frac{\overline{\langle X_k^{\tau}, \e_j^{\diamond} - \e_j \rangle  \langle X_{k-h}^{\tau}, \e_i^{\diamond} - \e_i\rangle}}{n-h}\biggr\|_{q} \lesssim n^{-1}.
\end{align}
Arguing in the same manner, one also obtains
\begin{align}\label{eq_lem_transf_A1_6}
\max_{1 \leq i,j \leq \tau}(\widetilde{\lambda}_i^{\diamond} \widetilde{\lambda}_j^{\diamond})^{-1/2}\biggl\|\sum_{h = 1}^{b_n}\sum_{k = h+1}^n \frac{\overline{\langle X_k^{\tau}, \e_j^{\diamond} - \e_j \rangle  \langle X_{k-h}^{\tau}, \e_i \rangle}}{n-h}\biggr\|_{q} \lesssim n^{-1},
\end{align}
and the same bound also applies to $\overline{\langle X_k^{\tau}, \e_j\rangle  \langle X_{k-h}^{\tau}, \e_i^{\diamond} - \e_i  \rangle}$. By virtue of the decomposition in \eqref{eq_lem_transf_A1_1}, the triangle inequality and \eqref{eq_lem_transf_A1_5}, \eqref{eq_lem_transf_A1_6}, we conclude that
\begin{align*}
\max_{1 \leq i,j \leq \tau}\bigl\|\overline{\etav}_{i,j}^{\diamond} - \overline{\etav}_{i,j}^{(\infty,b)}\bigr\|_{q} \lesssim n^{-1},
\end{align*}
which establishes \hyperref[itm_(i)_eta]{\itmi}. In order to show \hyperref[itm_(ii)_eta]{\itmii}, we can proceed in the same way. The only significant difference is that one needs to use Lemma \ref{lem_bound_partial_sum_gen} instead of Lemma \ref{lem_long_run_tech} \hyperref[itm_(i)_c]{\itmi}.
 \end{proof}

\begin{lemma}\label{lem_transf_A2}
Grant Assumption \ref{ass_eigenvaectors_longrun_exp_decay}. Then for sufficiently large $\td,C_0 > 0$, condition \hyperref[A2b]{\Atwob} holds for $\{\lambda_j^{\diamond}\}_{j \in \N}$ with the same $\ad$, $\JJ_n^+$, uniformly in $n,b$.
\end{lemma}

\begin{proof}[Proof of Lemma \ref{lem_transf_A2}]
From the triangle inequality,  Lemma \ref{lem_truncation_lemma}, Lemma \ref{lem_operator_comp} and Lemma \ref{lem_long_run_tech} \hyperref[itm_(iv)_c]{\itmiv} we get that for $1 \leq i,j \leq \tau$.
\begin{align}\nonumber \label{eq_lem_transf_A2_4}
\bigl|\lambda_i^{\diamond} - \lambda_j^{\diamond}\bigr| &\geq \bigl|\lambda_i - \lambda_j \bigr| - 2\bigl\|{\boldsymbol{\cal \GG}}^{\diamond} - {\boldsymbol{\cal \GG}}^{\tau} \bigr\|_{{\cal L}}\\& = \bigl|\lambda_i - \lambda_j \bigr| - \OO\bigl(\rho^{b}\bigr).
\end{align}
Due to the convexity assumption in \hyperref[A2star]{\Atwostar}, relation \eqref{eq_convex_bound} in Lemma \ref{lem_verify_ass_poly} and \hyperref[A2star]{\Atwostar} yield that for $i > j$
\begin{align}\label{eq_lem_transf_A2_5}
\bigl|\lambda_i - \lambda_j \bigr| \geq \bigl|\lambda_{j+1} - \lambda_j \bigr| \gtrsim \lambda_j/j \gtrsim j^{-\cd^- - 1}.
\end{align}
Combining \eqref{eq_lem_transf_A2_4} and \eqref{eq_lem_transf_A2_5}, it follows that for large enough $C_0 > 0$
\begin{align}\label{eq_lem_transf_A2_6}
\bigl|\lambda_i^{\diamond} - \lambda_j^{\diamond}\bigr| &\geq \bigl|\lambda_i - \lambda_j \bigr|\bigl(1 - \OO(\rho^b)\bigr), \quad \text{uniformly for $1 \leq i,j \leq \tau$.}
\end{align}
Using \eqref{eq_lem_transf_A2_6} and Lemma \ref{lem_long_run_tech} \hyperref[itm_(viii)_c]{\itmviii}, we thus obtain uniformly for $1 \leq j \leq \JJ_n^+ < \tau$
\begin{align*}
\max_{1 \leq j \leq \JJ_n^+}\sum_{\substack{i = 1\\j \neq i}}^{\tau} \frac{\lambda_i^{\diamond}}{|\lambda_j^{\diamond} - \lambda_i^{\diamond}|} \leq 2 \bigl(1 + \OO(\rho^b)\bigr)\max_{1 \leq j \leq \JJ_n^+}\sum_{\substack{i = 1\\j \neq i}}^{\tau} \frac{\lambda_i^{}}{|\lambda_j^{} - \lambda_i^{}|}.
\end{align*}
This together with Lemma \ref{lem_verify_ass_poly} and condition $\JJ_n^+ \lesssim n^{1/2 - \ad}(\log n)^{-\frac{3}{2}}$ yields that
\begin{align}\label{eq_lem_transf_A2_7}
(n/\log n)^{-1/2 + \ad} \max_{1 \leq j \leq \JJ_n^+}\sum_{\substack{i = 1\\j \neq i}}^{\tau} \frac{\lambda_i^{\diamond}}{|\lambda_j^{\diamond} - \lambda_i^{\diamond}|}  < \infty.
\end{align}
In the same manner, one establishes
\begin{align}\label{eq_lem_transf_A2_8}
(n/\log n)^{-1 + 2\ad} \max_{1 \leq j \leq \JJ_n^+}\sum_{\substack{i = 1\\j \neq i}}^{\tau} \frac{\lambda_i^{\diamond}\lambda_j^{\diamond}}{(\lambda_j^{\diamond} - \lambda_i^{\diamond})^2}  < \infty.
\end{align}
Combining \eqref{eq_lem_transf_A2_7}, \eqref{eq_lem_transf_A2_8} with the fact that $\lambda_{\JJ_n^+} \gtrsim n^{-\cd^-/2}$ by \hyperref[A2star]{\Atwostar} finishes the proof.

\end{proof}

\begin{lemma}\label{lem_transf_A3}
Grant Assumption \ref{ass_eigenvaectors_longrun_exp_decay}. Then for sufficiently large $\td,C_0 > 0$, \hyperref[A3b]{\Athreeb} holds for $\varphi_{j,j}^{\diamond}$, uniformly in $n,b$.
\end{lemma}

\begin{proof}[Proof of Lemma \ref{lem_transf_A3}]
Arguing similarly as in the proof of Lemma \ref{lem_long_run_tech} \hyperref[itm_(viii)_c]{\itmviii}, it follows that
\begin{align*}
\biggl|\sum_{|h| \leq b }\biggl(\E\bigl[\langle \overline{X}_k, \e_j^{\diamond} \rangle \langle \overline{X}_{k-h}, \e_j^{\diamond} \rangle \bigr] - \E\bigl[\langle \overline{X}_k, \e_j^{} \rangle \langle \overline{X}_{k-h}, \e_j^{} \rangle \bigr]\biggr) \biggr| = \oo\bigl((\widetilde{\lambda}_j^{\diamond} \widetilde{\lambda}_i^{\diamond})^{1/2}\bigr).
\end{align*}
Using Lemma \ref{lem_long_run_tech} \hyperref[itm_(viii)_c]{\itmviii}, we thus conclude from \hyperref[A3star]{\Athreestar} by routine calculations
\begin{align*}
\varphi_{j,j}^{\diamond} \geq \varphi_{j,j}/2 + \oo\bigl(1\bigr) \geq 1/(4 C^{\boldsymbol{\cal G}}), \text{uniformly for $1 \leq j \leq \tau$.}
\end{align*}
Similarly, using in addition Lemma \ref{lem_long_run_tech} \hyperref[itm_(ii)_c]{\itmii} yields
\begin{align*}
\varphi_{j,j}^{\diamond} = \varphi_{j,j} + \oo\bigl(1\bigr) \leq \sum_{|h| \leq b }\bigl|\E\bigl[\eta_{k,j} \eta_{k-h,j} \bigr]\bigr| + \oo\bigl(1\bigr) < \infty,
\end{align*}
which completes the proof.

\end{proof}

We are now ready to proceed to \hyperref[S2]{\Stwo}. To this end, we need the following preliminary result.

\begin{lemma}\label{lem_control_I_ij_exp}
Grant Assumption \ref{ass_eigenvaectors_longrun_exp_decay}. Then for $1 \leq q \leq p^*$ and sufficiently large $\td, C_0 > 0$, we have
\begin{enumerate} \setlength\itemsep{.6em}
  \item[\itmi]\label{itm_(i)} $\bigl\|I_{i,j}^{(\infty,b)}\bigr\|_{q} \lesssim \sqrt{{\lambda}_i {\lambda}_j (b/n)}$, uniformly for $i,j \in \N$,
  \item[\itmii]\label{itm_(ii)} $\max_{1 \leq j \leq \tau}\bigl\|\sum_{k = 1}^{\tau}|I_{k,j}^{(\infty,b)} - I_{k,j}^{\diamond}|^2 \bigr\|_{q/2} \lesssim n^{-2(\cd^+-1)\td}$,
  \item[\itmiii]\label{itm_(iii)} $\max_{1 \leq j \leq \tau}\bigl\|\sum_{k = 1}^{\tau}|I_{k,j}^{\diamond}|^2 \bigr\|_{q/2} \lesssim (b/n) + n^{-(\cd^+-1)\td}$.
\end{enumerate}
\end{lemma}

\begin{proof}[Proof of Lemma \ref{lem_control_I_ij_exp}]
We first show \hyperref[itm_(i)]{\itmi}. Using representation \eqref{eq_rep_X_k} with $b = \infty$ (a different $b$), elementary calculations give
\begin{align*}
\bigl\|I_{i,j}^{(\infty,b)}\bigr\|_{q} \lesssim \sqrt{\widetilde{\lambda}_i \widetilde{\lambda}_j}\biggl(\|\overline{\etav}_{i,j}^{(\infty,b)}\|_{q} + \max_{i,j \in \N}\frac{b}{n^2}\bigl\|\sum_{k = 1}^n \eta_{k,j} \bigr\|_{2q}^2 \biggr).
\end{align*}
Observe that by \hyperref[A3star]{\Athreestar} we get that $\min_{j \in \N} \varphi_{j,j} \geq 1/C^{\boldsymbol{\cal G}}$. Due to \eqref{lem_varphi_relations} we conclude $\max_{j \in \N} \widetilde{\lambda}_j/\lambda_j \leq C^{\boldsymbol{\cal G}}$. From \hyperref[A1star]{\Aonestar}, using Lemma \ref{lem_bound_partial_sum_gen} and Lemma \ref{lem_long_run_tech} \hyperref[itm_(i)_c]{\itmi}, claim \hyperref[itm_(i)]{\itmi} follows. Next, observe that by the triangle inequality, Cauchy-Schwarz and Lemma \ref{lem_truncation_lemma}, for $1 \leq j,k \leq \tau$
\begin{align*}
\bigl|I_{j,k}^{(\infty,b)} - I_{j,k}^{\diamond} \bigr| &\lesssim \bigl| \big\langle \bigl(\widehat{\boldsymbol{\cal \GG}}^{b} - \widehat{\boldsymbol{\cal \GG}}^{\diamond}\bigr)(\e_j), \e_k \big\rangle\bigl| + \bigl| \big\langle \bigl({\boldsymbol{\cal \GG}}^{b} - {\boldsymbol{\cal \GG}}^{\diamond}\bigr)(\e_j), \e_k \big \rangle\bigl| \\&+ \bigl\|\widehat{\boldsymbol{\cal \GG}}^{\diamond} - {\boldsymbol{\cal \GG}}^{\diamond}\bigr\|_{{\cal L}} \max_{1 \leq j \leq \tau}\bigl\|\e_j - \e_j^{\diamond}\bigr\|_{\Ln^2}.
\end{align*}
Since $\sum_{j = 1}^{\tau} \langle x, \e_j \rangle^2 \leq \|x\|_{\Ln^2}^2$, inequality $(a+b+c)^2 \leq 3(a^2 + b^2 + c^2)$, Lemma \ref{lem_long_run_tech} \hyperref[itm_(v)_c]{\itmv}, \hyperref[itm_(vi)_c]{\itmvi} and the triangle inequality then yield
\begin{align*}
\bigl\|\sum_{k = 1}^{\tau}\bigl|I_{j,k}^{(\infty,b)} - I_{j,k}^{\diamond} \bigr|^2\bigr\|_{q/2} &\lesssim \bigl\|\|\widehat{\boldsymbol{\cal \GG}}^{b} - \widehat{\boldsymbol{\cal \GG}}^{\diamond}\|_{{\cal L}}^2\bigr\|_{q/2} + \bigl\|{\boldsymbol{\cal \GG}}^{b} - {\boldsymbol{\cal \GG}}^{\diamond}\bigl\|_{{\cal L}}^2 + \tau^2 \rho^{2b}.
\end{align*}
Hence selecting $C_0 > 0$ sufficiently large, claim \hyperref[itm_(ii)]{\itmii} follows from Lemma \ref{lem_long_run_tech} \hyperref[itm_(iii)_c]{\itmiii}. Claim \hyperref[itm_(iii)]{\itmiii} follows from $(a+b)^2 \leq 2(a^2 + b^2)$ and \hyperref[itm_(i)_c]{\itmi}, \hyperref[itm_(ii)_c]{\itmii}.

\end{proof}

We can now complete \hyperref[S2]{\Stwo}:

\begin{proof}[Proof of Theorem \ref{theorem_exp_decay_long_run_eigen}]
By virtue of Theorem \ref{thm_longrun_b_remains_valid} and Lemmas \ref{lem_Gstar_is_compact_symm_and trace_class}, \ref{lem_transf_A1}, \ref{lem_transf_A2} and \ref{lem_transf_A3}, it suffices to show that the error of replacing all quantities in Theorem \ref{theorem_exp_decay_long_run_eigen} with their $\diamond$-analogues is negligible. More precisely,
\begin{align}
&\bigl\|\max_{1 \leq j < \JJ_n^+}|\widehat{\lambda}_j^{b} - \lambda_j^{} - I_{j,j}^{(b,\infty)}|/\lambda_j \bigr\|_p \lesssim \bigl\|\max_{1 \leq j < \JJ_n^+}|\widehat{\lambda}_j^{\diamond} - \lambda_j^{\diamond} - I_{j,j}^{\diamond}|/\lambda_j^{\diamond} \bigr\|_p + \text{error}.
\end{align}
We will do so in the sequel. Observe first that by Lemma \ref{lem_truncation_lemma} and Lemma \ref{lem_operator_comp} we have
\begin{align*}
\bigl|(\widehat{\lambda}_j^b - {\lambda}_j^{}) - (\widehat{\lambda}_j^{\diamond} - {\lambda}_j^{\diamond})\bigr| \leq \bigl|\widehat{\lambda}_j^b - \widehat{\lambda}_j^{\diamond} \bigr| + \bigl|{\lambda}_j - {\lambda}_j^{\diamond} \bigr| \leq \bigl\|\widehat{\boldsymbol{\cal \GG}}^{b} - \widehat{\boldsymbol{\cal \GG}}^{\diamond}\bigr\|_{{\cal L}} + \bigl\|{\boldsymbol{\cal \GG}}^{b} - {\boldsymbol{\cal \GG}}^{\diamond}\bigr\|_{{\cal L}}.
\end{align*}
Hence by Lemma \ref{lem_long_run_tech} \hyperref[itm_(iii)_c]{\itmiii} we get that (recall $p^* = p2^{\pd + 4}$)
\begin{align*}
\sqrt{n/b}\bigl\|\max_{1 \leq j < \JJ_n^+} \bigl|(\widehat{\lambda}_j^b - {\lambda}_j^{}) - (\widehat{\lambda}_j^{\diamond} - {\lambda}_j^{\diamond})\bigr|/\lambda_j \bigr\|_{p} \lesssim \sqrt{n/b } n^{-(\cd^+-1)\td}/\lambda_{\JJ_n^+}.
\end{align*}
Due to condition \hyperref[A2star]{\Atwostar} and $\JJ_n^+ \lesssim n^{1/2}$, we get
\begin{align*}
\sqrt{n/b}\, n^{-(\cd^+-1)\td}/\lambda_{\JJ_n^+} \lesssim n^{\frac{\cd^- +1}{2} -(\cd^+ -1)\td } \lesssim n^{-1},
\end{align*}
for $\td$ sufficiently large. We thus conclude
\begin{align}\label{eq_thm_exp_decay_long_run_eigen_1}
\sqrt{n/b }\biggr\|\max_{1 \leq j < \JJ_n^+}\bigl(\bigl|(\widehat{\lambda}_j^b - {\lambda}_j^{}) - (\widehat{\lambda}_j^{\diamond} - {\lambda}_j^{\diamond})\bigr|/\lambda_j\bigr)\biggr\|_{p} \lesssim n^{-1}.
\end{align}
An application of Lemma \ref{lem_control_I_ij_exp} \hyperref[itm_(ii)]{\itmii} yields that for sufficiently large $\td > 0$ and $C_0 > 0$
\begin{align}\label{eq_thm_exp_decay_long_run_eigen_2}
\sqrt{n/b}\bigl\|\max_{1 \leq j < \JJ_n^+}|I_{j,j}^{(\infty,b)} - I_{j,j}^{\diamond}|/\lambda_j \bigr\|_p \lesssim n^{-1}.
\end{align}
Combining \eqref{eq_thm_exp_decay_long_run_eigen_1} and \eqref{eq_thm_exp_decay_long_run_eigen_2} and using Lemma \ref{lem_long_run_tech} \hyperref[itm_(viii)_c]{\itmviii}, we arrive at
\begin{align*}
&\bigl\|\max_{1 \leq j < \JJ_n^+}|\widehat{\lambda}_j^{b} - \lambda_j^{} - I_{j,j}^{(b,\infty)}|/\lambda_j \bigr\|_p \lesssim
\bigl\|\max_{1 \leq j < \JJ_n^+}|\widehat{\lambda}_j^{\diamond} - \lambda_j^{\diamond} - I_{j,j}^{\diamond}|/\lambda_j \bigr\|_p + 1/(n \sqrt{n/b}) \\&\lesssim \bigl\|\max_{1 \leq j < \JJ_n^+}|\widehat{\lambda}_j^{\diamond} - \lambda_j^{\diamond} - I_{j,j}^{\diamond}|/\lambda_j^{\diamond} \bigr\|_p + 1/(n \sqrt{n/b}).
\end{align*}

\end{proof}

\begin{proof}[Proof of Theorem \ref{theorem_exp_decay_long_run_fun}]
Proceeding as in the proof of Theorem \ref{theorem_exp_decay_long_run_eigen}, based on Lemmas \ref{lem_transf_A1}, \ref{lem_transf_A2} and \ref{lem_transf_A3}, it suffices to show that the error of replacing all expressions by their corresponding $\diamond$-analogues is bounded by $n^{-2}$, uniformly for $1 \leq j < \JJ_n^+$. To this end, note first that due to the convexity assumption in \hyperref[A2star]{\Atwostar}, Lemma \ref{lem_verify_ass_poly} yields that uniformly for $k,j \in \N$
\begin{align}\label{eq_thm_exp_eigenfunctions_exp_1}
(j \vee k) |\lambda_j - \lambda_k|\gtrsim (\lambda_j \vee \lambda_k)|j-k|.
\end{align}
We will make frequent use of this lower bound in the sequel. We first consider the expansion of $\widehat{\e}_j^b - \e_j$. To this end, we establish preliminary bounds regarding $I_{k,j}^{(\infty,b)}$, $I_{k,j}^{\diamond}$. For $2\JJ_n^+ < \tau$, using Lemma \ref{lem_max_mom} and the triangle inequality we get
\begin{align*}
\biggl\|\max_{1 \leq j < \JJ_n^+}\frac{1}{\Lambda_j}\sum_{k > \tau} \frac{\bigl(I_{k,j}^{(\infty,b)}\bigr)^2}{(\lambda_k - \lambda_j)^2}\biggr\|_p \lesssim \bigl(\JJ_n^{+}\bigr)^{1/p} \sum_{k > \tau} \max_{1 \leq j < \JJ_n^+}\frac{1}{\Lambda_j}\frac{\bigl\|\bigl(I_{k,j}^{(\infty,b)}\bigr)^2\bigr\|_p}{(\lambda_k - \lambda_j)^2}.
\end{align*}
Since $2p \leq p^*$, $2\JJ_n^+ < \tau$, Lemma \ref{lem_control_I_ij_exp} \hyperref[itm_(i)]{\itmi}, \eqref{eq_thm_exp_eigenfunctions_exp_1} and \hyperref[A2star]{\Atwostar} yield the upper bound
\begin{align*}
\bigl(\JJ_n^{+}\bigr)^{1/p} \frac{b}{n} \sum_{k > \tau} \max_{1 \leq j < \JJ_n^+}\frac{1}{\Lambda_j}\frac{k^2 \lambda_j \lambda_k}{\lambda_j^2 (k - j)^2} \lesssim \bigl(\JJ_n^{+}\bigr)^{1/p}\frac{b n^{-(\cd^+-1)\td -1}}{\lambda_{\JJ_n^+}} \max_{1 \leq j < \JJ_n^+}\frac{1}{\Lambda_j}.
\end{align*}
Since $\Lambda_j \geq \lambda_j/\lambda_{j-1} \gtrsim \lambda_j \wedge 1$ and $\lambda_j \gtrsim j^{-\cd^-}$ by \hyperref[A2star]{\Atwostar} and $\JJ_n^+ \lesssim n^{1/2}$, we conclude from the above that for sufficiently large $\td > 0$ we have
\begin{align}\label{eq_thm_exp_eigenfunctions_exp_2}
\biggl\|\max_{1 \leq j < \JJ_n^+}\frac{1}{\Lambda_j}\sum_{k > \tau} \frac{\bigl(I_{k,j}^{(\infty,b)}\bigr)^2}{(\lambda_k - \lambda_j)^2}\biggr\|_p \lesssim n^{-2}, \quad 2\JJ_n^+ \leq \tau.
\end{align}
Arguing in a similar manner, we get that for sufficiently large $\td > 0$
\begin{align}\nonumber\label{eq_thm_exp_eigenfunctions_exp_3}
&\biggl\|\max_{1 \leq j < \JJ_n^+}\frac{1}{\Lambda_j^2}\sum_{\substack{k = 1\\k \neq j}}^{\tau} \frac{\bigl|I_{k,j}^{(\infty,b)} - I_{k,j}^{\diamond} \bigr|^2}{(\lambda_j - \lambda_k)^2}\biggr\|_{p/2} \lesssim \max_{1 \leq j < \JJ_n^+}\biggl(\frac{(\JJ_n^+)^{2/p}}{\lambda_{j}^2 \Lambda_j^2}\biggl\|\sum_{\substack{k = 1 \\ k \neq j}}^{\tau} \bigl|I_{k,j}^{(\infty,b)} - I_{k,j}^{\diamond}|^2 \biggr\|_{q/2}\biggr) \\&\lesssim n^{-2(\cd^+-1)\td + 2\cd^- + 1/p} \lesssim n^{-2}.
\end{align}
By related arguments and Lemma \ref{lem_long_run_tech} \hyperref[itm_(vi)_c]{\itmvi}, Lemma \ref{lem_control_I_ij_exp} \hyperref[itm_(iii)]{\itmiii}, we get that for sufficiently large $C_0 > 0$
\begin{align}\label{eq_thm_exp_eigenfunctions_exp_4}
\biggl\|\max_{1 \leq j < \JJ_n^+}\frac{1}{\Lambda_j}\biggl\|\sum_{\substack{k = 1\\k \neq j}}^{\tau} \frac{(\e_k^{\diamond} - \e_k) I_{k,j}^{\diamond}}{\lambda_j - \lambda_k} \biggr\|_{\Ln^2}\biggr\|_p \lesssim \max_{1 \leq j < \JJ_n^+}\biggl(\frac{(\JJ_n^+)^{1/p+1}}{\lambda_{\JJ_n^+} \Lambda_j}\sum_{\substack{k = 1 \\ k \neq j}}^{\tau} \frac{\rho^b}{|k-j|} \biggr) \lesssim \frac{1}{n^2}.
\end{align}
Combining \eqref{eq_thm_exp_eigenfunctions_exp_2}, \eqref{eq_thm_exp_eigenfunctions_exp_3} and \eqref{eq_thm_exp_eigenfunctions_exp_4} and using $\|\e_j\|_{\Ln^2} = 1$ we obtain via the triangle inequality
\begin{align}\label{eq_thm_exp_eigenfunctions_exp_5}
\biggl\|\max_{1 \leq j < \JJ_n^+}\frac{1}{\Lambda_j}\biggl\|\sum_{\substack{k = 1\\k \neq j}}^{\infty} \e_k \frac{I_{k,j}^{(\infty,b)}}{\lambda_j - \lambda_k} - \sum_{\substack{k = 1\\k \neq j}}^{\tau}\e_k^{\diamond}\frac{I_{k,j}^{\diamond}}{\lambda_j - \lambda_k}\biggr\|_{\Ln^2}\biggr\|_p \lesssim \frac{1}{n^2}.
\end{align}
Now, using $(a+b)^2 \leq 2(a^2 + b^2)$, $\|\widehat{\e}_j^b\|_{\Ln^2},\|\widehat{\e}_j^{\diamond}\|_{\Ln^2},\|{\e}_j^{\diamond}\|_{\Ln^2}, \|{\e}_j\|_{\Ln^2} = 1$, Lemma \ref{lem_long_run_tech} \hyperref[itm_(vi)_c]{\itmvi}, \hyperref[itm_(vii)_c]{\itmvii} and \eqref{eq_thm_exp_eigenfunctions_exp_5}, the triangle inequality gives for sufficiently large $\td, C_0 > 0$
\begin{align}\nonumber \label{eq_thm_exp_eigenfunctions_exp_5.5}
&\biggl\|\max_{1 \leq j < \JJ_n^+}\frac{1}{\Lambda_j}\biggl\|\widehat{\e}_j^b - \e_j - \frac{\e_j}{2}\|\widehat{\e}_j^b - \e_j\|_{\Ln^2}^2 - \sum_{\substack{k = 1\\k \neq j}}^{\infty} \e_k \frac{I_{k,j}^{(\infty,b)}}{\lambda_j - \lambda_k}\biggr\|_{\Ln^2}\biggr\|_p \\&\lesssim \frac{1}{n^2} + \biggl\|\max_{1 \leq j < \JJ_n^+}\frac{1}{\Lambda_j}\biggl\|\widehat{\e}_j^{\diamond} - \e_j^{\diamond} - \frac{\e_j^{\diamond}}{2}\|\widehat{\e}_j^{\diamond} - \e_j^{\diamond}\|_{\Ln^2}^2 - \sum_{\substack{k = 1\\k \neq j}}^{\tau} \e_k^{\diamond} \frac{I_{k,j}^{\diamond}}{\lambda_j - \lambda_k}\biggr\|_{\Ln^2}\biggr\|_p.
\end{align}
Moreover, using the same arguments as in the proof of Lemma \ref{lem_transf_A2}, it follows that
\begin{align}\label{eq_thm_exp_eigenfunctions_exp_6}
\Lambda_j = \sum_{\substack{k = 1\\k \neq j}}^{\infty} \frac{\lambda_j \lambda_k }{(\lambda_j - \lambda_k)^2} \geq \sum_{\substack{k = 1\\k \neq j}}^{\tau} \frac{\lambda_j \lambda_k }{(\lambda_j - \lambda_k)^2} \gtrsim \bigl(1 - \OO(\rho^b)\bigr) \sum_{\substack{k = 1\\k \neq j}}^{\tau} \frac{\lambda_j^{\diamond} \lambda_k^{\diamond} }{(\lambda_j^{\diamond} - \lambda_k^{\diamond})^2} \stackrel{def}{=} \Lambda_j^{\diamond},
\end{align}
and this holds uniformly for $1 \leq j < \JJ_n^+$ (we exclude $\OO(\rho^b)$ in the above definition of $\Lambda_j^{\diamond}$). Similarly, using also Lemma \ref{lem_control_I_ij_exp} \hyperref[itm_(iii)]{\itmiii} in addition, it follows that
\begin{align}\label{eq_thm_exp_eigenfunctions_exp_7}
&\biggl\|\max_{1 \leq j < \JJ_n^+}\frac{1}{\Lambda_j}\biggl\|\sum_{\substack{k = 1\\k \neq j}}^{\tau} \e_k^{\diamond}I_{k,j}^{\diamond}\biggl(\frac{1}{\lambda_j - \lambda_k} - \frac{1}{\lambda_j^{\diamond} - \lambda_k^{\diamond}}\biggr)\biggr\|_{\Ln^2}\biggr\|_p \lesssim \rho^b,
\end{align}
for sufficiently large $C_0 > 0$. Using first \eqref{eq_thm_exp_eigenfunctions_exp_7} and then \eqref{eq_thm_exp_eigenfunctions_exp_6}, it follows that \eqref{eq_thm_exp_eigenfunctions_exp_5.5} is further bounded by
\begin{align}
\lesssim \frac{1}{n^2} + \biggl\|\max_{1 \leq j < \JJ_n^+}\frac{1}{\Lambda_j^{\diamond}}\biggl\|\widehat{\e}_j^{\diamond} - \e_j^{\diamond} - \frac{\e_j^{\diamond}}{2}\|\widehat{\e}_j^{\diamond} - \e_j^{\diamond}\|_{\Ln^2}^2 - \sum_{\substack{k = 1\\k \neq j}}^{\tau} \e_k^{\diamond} \frac{I_{k,j}^{\diamond}}{\lambda_j^{\diamond} - \lambda_k^{\diamond}}\biggr\|_{\Ln^2}\biggr\|_p.
\end{align}
This completes the proof for the expansion of $\widehat{\e}_j^b - \e_j$. The treatment of the expansion $\|\widehat{\e}_j^b - \e_j\|_{\Ln^2}^2$ only requires minor adaption of the previous arguments, we omit the details.

\end{proof}

\begin{proof}[Proof of Proposition \ref{prop_replace_I_with_eta_exp}]
This follows from Lemma \ref{lem_long_run_tech} \hyperref[itm_(i)_c]{\itmi} and analogue computations as in the proof of Theorem \ref{thm_longrun_b_remains_valid}.

\end{proof}

\begin{proof}[Proof of Corollary \ref{cor_norm_bounds_exp}]
This follows from Proposition \ref{prop_replace_I_with_eta_exp} and Lemma \ref{lem_long_run_tech} \hyperref[itm_(i)_c]{\itmi}.
\end{proof}

\section{Proofs of Section  \ref{sec_applications}}

We need to introduce some further notation. To this end, we slightly reformulate our notion of weak dependence in an equivalent way. In the sequel, $\bigl\{\epsilon_{k}\bigr\}_{k \in \Z} \in \mathds{S}$ denotes an IID sequence in some measure space $\mathds{S}$ and $\F_k = \sigma\bigl(\epsilon_j,\, j \leq k\bigr)$ the corresponding filtration. For $d \in \N$, we then consider the variables
\begin{align*}
U_{k,h} = H_h\bigl(\F_k\bigr), \quad k \in \Z,\, 1 \leq h \leq d,
\end{align*}
where $H_h$ are measurable functions. Note that by considering different measure spaces $\mathds{S}$, we can virtually model any spatial dependence structure we want, with the extreme cases where $U_{k,h} = U_{k,h+1}$ or $U_{k,h}$ and $U_{k,h+1}$ are independent. Compared to Section \ref{sec_applications}, this setup is notationally more convenient, and prevents us from the necessity of considering different sequences $\bigl\{\epsilon_{k,h}\bigr\}_{k \in \Z}$ for each coordinate $h$. As a measure of dependence, we then consider
\begin{align*}
\theta_{j,p} = \max_{1 \leq h \leq d}\bigl\|U_{j,h} - U_{j,h}'\bigr\|_p, \quad p \geq 1,
\end{align*}
where $U_{k,h} = H_h\bigl(\F_k'\bigr)$, $\F_k' = \sigma\bigl(\ldots \epsilon_{-1},\epsilon_0', \epsilon_1, \ldots, \epsilon_k\bigr)$, and $\bigl\{\epsilon_k'\bigr\}_{k \in \Z}$ is an independent copy of $\bigl\{\epsilon_k\bigr\}_{k \in \Z}$.

\subsection{Gaussian approximation for weak dependence}\label{sec_gaussian_approx}

In this section, a high dimensional Gaussian approximation result is established, which is a key ingredient in the proof of Theorem \ref{thm_eigen_max_gauss_approx}. This result may be of independent interest. Let $S_{n,h} = \sum_{k = 1}^n U_{k,h}$, and denote with
\begin{align}
T_d = \frac{1}{\sqrt{n}} \max_{1 \leq h \leq d}\bigl|S_{n,h}\bigr|, \quad T_d^{Z} = \max_{1 \leq h \leq d}\bigl|Z_h\bigr|,
\end{align}
where $\bigl\{Z_h\bigr\}_{1 \leq h \leq d}$ is a sequence of zero mean Gaussian random variables. We also formally introduce
\begin{align*}
\gamma_{i,j} = \lim_{n \to \infty}\frac{1}{n}\E\bigl[S_{n,i} S_{n,j} \bigr],
\end{align*}
existence is shown below in Lemma \ref{lem_var_cov_property}. We also put $\sigma_h^2 = \gamma_{h,h}$. Throughout this section, we work under the following assumption.
\begin{ass}\label{ass_gaussian_approx}
The sequence $\bigl\{U_{k,h}\bigr\}_{k \in \Z}$ is stationary for each $1 \leq h \leq d$, such that for $p > 2$ and $d \lesssim n^{\dd}$
\begin{enumerate}%[leftmargin=1cm]
\item[\Cone]\label{C1} $\E\bigl[U_{k,h}\bigr] = 0$ and $\theta_{j,p} \lesssim j^{-\cd}$ with $\cd > 3/2 $,
\item[\Ctwo]\label{C2} $\dd < p/2 - 1$,
\item[\Cthree]\label{C3} $\inf_h \sigma_h > 0$.
\end{enumerate}
\end{ass}

We then have the following Gaussian approximation result.

%and $\Sigma_d^{} = \bigl(\gamma_{i,j}\bigr)_{1 \leq i,j \leq d}$.
\begin{theorem}\label{thm_gauss_approx}
Grant Assumption \ref{ass_gaussian_approx}. Then
\begin{align*}
\sup_{x \in \R}\bigl|P\bigl(T_d \leq x \bigr) - P\bigl(T_d^{Z} \leq x \bigr) \bigr| \lesssim n^{-C}, \quad C > 0,
\end{align*}
where $\bigl\{Z_h\bigr\}_{1 \leq h \leq d}$ has the same covariance structure as $n^{-1/2}\bigl\{S_{n,h}\bigr\}_{1 \leq h \leq d}$. Alternatively, we may also choose $\bigl(\gamma_{i,j}\bigr)_{1 \leq i,j \leq d}$ as covariance structure.
\end{theorem}

\begin{comment}
The proof of Theorem \ref{thm_gauss_approx} is lengthy, and requires a number of preliminary lemmas. The general idea consists of a coupling argument, exploiting the Bernoulli structure of $U_{k,h}$. Employing a Fuk-Nagaev type inequality (Theorem 2 in ~\cite{Wu_fuk_nagaev}) as one of the major tools, we can subsequently reduce the problem to the IID case, where we can resort to the literature.\\
\\
\end{comment}
We first establish some additional notation. Let $K = n^{\kd}$, $L = n^{\ld}$ such that $n = K L$ and $0 < \kd,\ld < 1$. To simplify the discussion, we always assume that $K,L \in \N$. For each $1 \leq l \leq L$, let $\bigl\{\epsilon_{k}^{l}\bigr\}_{k \in \Z} \in \mathds{S}$ be mutually independent sequences of IID random variables. For $K(l-1) < k \leq Kl$, $1 \leq l \leq L$, denote with
\begin{align*}
U_{k,h}^{(K,\diamond)} = H_h\bigl(\F_{k,h}^{K,\diamond} \bigr), \quad \text{where $\F_{k,h}^{K,\diamond} = \sigma\bigl(\F_{K(l-1)}^l,\epsilon_{K(l-1)+1},\epsilon_{K(l-1)+2},\ldots,\epsilon_k\bigr)$},
\end{align*}
where $\F_k^l = \sigma\bigl(\epsilon_j^l\, j \leq k \bigr)$. For $1 \leq m < K$ put
\begin{align}
V_{l,h}^{\diamond}(m) = \sum_{k = K(l-1) + 1}^{K(l-1) + m - 1} U_{k,h} + \sum_{k = K(l-1) + m}^{Kl} U_{k,h}^{(K,\diamond)},
\end{align}
and $V_{l,h}^{\diamond} = V_{l,h}^{\diamond}(1)$. The random variables $V_{l,h}^{\diamond}$ play a key role in the proof of Theorem \ref{thm_gauss_approx}. Note in particular that $\bigl\{V_{l,h}^{\diamond}\bigr\}_{1 \leq l \leq L}$ is IID by construction for each $h$. Finally, put $S_{L,h}(V) = \sum_{l = 1}^L V_{l,h}$ and $S_{L,h}^{\diamond}(V) = \sum_{l = 1}^L V_{l,h}^{\diamond}$, and note that $S_{n,h} = S_{L,h}(V)$. In the sequel, we make frequent use of the following lemma.
\begin{lemma}\label{lem_bound_partial_sum_gen}
Suppose that $\sum_{j = 1}^{\infty} \theta_{j,p} < \infty$ for $p \geq 2$. Then
\begin{align*}
\max_{1 \leq h \leq d}\bigl\|U_{1,h} + \ldots + U_{n,h} \bigr\|_p \lesssim \sqrt{n}.
\end{align*}
\end{lemma}
For the proof and variants of this result, see ~\cite{sipwu}. The next lemma controls the approximation error between $S_{L,h}(V)$ and $S_{L,h}^{\diamond}(V)$.

\begin{lemma}\label{lem_control_V_approx}
Grant Assumption \ref{ass_gaussian_approx}. For any $K = n^{\kd}$ with $0 < \kd < 1$ there exists a $\delta > 0$ and a constant $C > 0$ such that
\begin{align*}
P\bigl(\bigl|S_{L,h}(V) - S_{L,h}^{\diamond}(V)\bigr| \geq  C n^{1/2 - \delta}\bigr) \lesssim n^{-\frac{p-2}{2} + p \delta}.
\end{align*}
\end{lemma}

\begin{proof}[Proof of Lemma \ref{lem_control_V_approx}]
Let $x_n = x \sqrt{n}$, $x > 0$. For $1 \leq m  < K$ we have that
\begin{align*}
P\bigl(\bigl|S_{L,h}(V) - S_{L,h}^{\diamond}(V)\bigr| \geq 2 x_n \bigr) &\leq P\biggl(\biggl|\sum_{l = 1}^L \sum_{k = K(l-1) + 1}^{K(l-1) + m - 1} U_{k,h} - U_{k,h}^{(K,\diamond)} \biggr| \geq x_n \biggr) \\&+ P\biggl(\biggl|\sum_{l = 1}^L V_{l,h}^{} - V_{l,h}^{\diamond}(m) \biggr| \geq x_n \biggr).
\end{align*}
Denote with $\alpha_{j,p} = \bigl(j^{p/2 - 1} \theta_{j,p}^p\bigr)^{1/(p+1)}$ and $A = \sum_{j = 1}^{\infty} \alpha_{j,p}$. Note that by \hyperref[C1]{\Cone} we have
\begin{align}
\alpha_{j,p} \lesssim  j^{-\mathfrak{B}(p,\cd)},  \quad \text{where} \quad \mathfrak{B}(p,\cd) = \frac{p(\cd-1/2)+1}{p+1} > 1,
\end{align}
and thus $A < \infty$. Due to Theorem 2 in ~\cite{Wu_fuk_nagaev}, there exist constants $C_{p,1},C_{p,2} > 0$ such that
\begin{align*}
P\biggl(\biggl|\sum_{l = 1}^L \sum_{k = K(l-1) + 1}^{K(l-1) + m - 1} U_{k,h} - U_{k,h}^{(K,\diamond)} \biggr| \geq x_n \biggr) &\leq \frac{C_{1,p} Lm}{x_n^p} + \sum_{j = 1}^{\infty} \exp\biggl(-\frac{C_{p,2} \alpha_{j,p}^2 x_n^2}{A^2 L m \theta_{j,2}^2}\biggr) \\&+  \exp\biggl(-\frac{C_{p,2} x_n^2}{L m\|U_{k,h}\|_2^2}\biggr).
\end{align*}
Setting $x = y\sqrt{L \, m} A^{1 + 1/p}/\sqrt{n}$, it follows that $\alpha_{j,p}^2 x_n^2 /(A^2 L\,m \theta_{j,2}^2) \geq j^{1 - 2/p}y^2$ and hence
\begin{align*}
\exp\biggl(-\frac{C_{p,2} \alpha_{j,p}^2 x_n^2}{A^2 L\,m \theta_{j,2}^2}\biggr) \leq \exp\biggl(-C_{p,2} j^{1 - 2/p}y^2\biggr).
\end{align*}
Choosing $m$ such that $\sqrt{n}/\sqrt{L m} = n^{2 \delta}$ and $y = n^{\delta}$, $\delta > 0$, it follows that
\begin{align}
P\biggl(\biggl|\sum_{l = 1}^L \sum_{k = K(l-1) + 1}^{K(l-1) + m - 1} U_{k,h} - U_{k,h}^{(K,\diamond)} \biggr| \geq n^{1/2 - \delta} A^{1 + 1/p} \biggr) \lesssim n^{-\frac{p-2}{2} + p\delta}.
\end{align}
Next, put $\Delta_{k,h}(U) = U_{k,h} - U_{k,h}^{(K,\diamond)}$. By the triangle inequality, we have
\begin{align*}
\bigl\|\Delta_{k,h}(U) - \Delta_{k,h}(U)'\bigr\|_p \leq 2 \bigl(\theta_{k,p} \wedge \bigl\|\Delta_{k,h}(U)\bigr\|_p\bigr).
\end{align*}
Let $(k)_K = k \mod K$. Then Theorem 1 in ~\cite{wu_2005} yields that
\begin{align*}
\max_{1 \leq h \leq d}\bigl\|\Delta_{k,h}(U)\bigr\|_p^2 = \max_{1 \leq h \leq d}\bigl\|U_{k,h} - U_{k,h}^{(K,\diamond)}\bigr\|_p^2 \lesssim \sum_{j = (k)_K}^{\infty} \theta_{j,p}^2 \stackrel{def}{=} \Theta_{(k)_K,p}.
\end{align*}
Since clearly $\Theta_{(k)_K,p}$ is monotone decreasing, we have $\Theta_{(k)_K,p} \leq \Theta_{(m)_K,p}$ for $m \leq k \leq K$. Combining this with the above, it follows that for $m \leq (k)_K$ (since $m = (m)_K$)
\begin{align}
\max_{1 \leq h \leq d}\bigl\|\Delta_{k,h}(U) - \Delta_{k,h}(U)'\bigr\|_p \leq 2 \biggl(\theta_{k,p} \wedge \sqrt{\Theta_{m,p}}\biggr) \stackrel{def}{=}\vartheta_{k,p}(m).
\end{align}
Put $\beta_{j,p}(m) = \bigl(j^{p/2 - 1} \vartheta_{j,p}^p(m)\bigr)^{1/(p+1)}$ and $B(m) = \sum_{j = 1}^{\infty} \beta_{j,p}(m)$. Then another application of Theorem 2 in ~\cite{Wu_fuk_nagaev} yields that
\begin{align*}
P\biggl(\biggl|\sum_{l = 1}^L V_{l,h}^{} - V_{l,h}^{\diamond}(m) \biggr| \geq x_n \biggr) &\leq C_{1,p}\frac{n}{x_n^p} + \sum_{j = 1}^{\infty} \exp\biggl(-\frac{C_{p,2} \beta_{j,p}^2(m) x_n^2}{B^2(m) n \vartheta_{j,2}^2(m)}\biggr) \\&+  \exp\biggl(-\frac{C_{p,2} x_n^2}{n \max_{k \geq m}\|\Delta_{k,h}(U)\|_2^2}\biggr).
\end{align*}
Let $y_n = n^{\delta} \sqrt{L m}/\sqrt{n} = n^{-\delta}$. Arguing similarly as before, it follows (since $m = (m)_K$)
\begin{align*}
P\biggl(\biggl|\sum_{l = 1}^L V_{l,h}^{} - V_{l,h}^{\diamond}(m) \biggr| \geq x_n \biggr) &\lesssim \frac{n}{x_n^p} + \sum_{j = 1}^{\infty} \exp\biggl(-\frac{C_{p,2} j^{1 + -2/p}y_n^2}{B(m)^2}\biggr) \\&+ \exp\biggl(-\frac{C_{p,2} y_n^2}{\Theta_{m,p}}\biggr).
\end{align*}
Since $\Theta_{m,p} \lesssim m^{-2 \cd + 1}$, we conclude
\begin{align*}
B(m) &\lesssim \sum_{j > M} \alpha_{j,p} + \sum_{j = 1}^M \bigl(j^{p/2 - 1} m^{-p \cd + p/2} \bigr)^{1/(p+1)} \lesssim M^{-\mathfrak{B}(p,\cd)} + M^{\frac{3p}{2p + 2}} m^{\frac{-2p \cd + p}{2p + 2}}.
\end{align*}
Setting $m \thicksim n^{\nu}$, $\nu > 0$, balancing the above and choosing $\delta$ sufficiently small, we obtain
\begin{align}
\frac{y_n^2}{B(m)^2} \wedge \frac{y_n^2}{\Theta_{m,p}} \gtrsim n^{\delta}.
\end{align}
This implies that
\begin{align*}
P\biggl(\biggl|\sum_{l = 1}^L V_{l,h}^{} - V_{l,h}^{\diamond}(m) \biggr| \geq n^{1/2 - \delta} A^{1+1/p} \biggr) &\lesssim n^{-\frac{p-2}{2} + p\delta}.
\end{align*}
Note that by the above choice of $m = n^{\nu}$ we require that $L \thicksim n^{1 - 4 \delta - \nu}$. Choosing $\nu$ sufficiently close to $1$, we can select $\kd < 1$ arbitrarily close to $1$, which completes the proof.

\end{proof}

In the sequel, we also require the following result.
\begin{lemma}\label{lem_control_V_direct}
Grant Assumption \ref{ass_gaussian_approx}. Then
\begin{align*}
P\biggl(\biggl|V_{l,h}^{\diamond} \biggr| \geq \sqrt{K} \log n \biggr) \lesssim K^{1 - p/2} \bigl(\log n)^{p}.
\end{align*}
\end{lemma}

\begin{proof}[Proof of Lemma \ref{lem_control_V_direct}]
Since $V_{l,h}^{\diamond} \stackrel{d}{=} V_{l,h}$, Theorem 2 in ~\cite{Wu_fuk_nagaev} and arguing similarly as in Lemma \ref{lem_control_V_approx} yields
\begin{align*}
P\biggl(\biggl|V_{l,h}^{\diamond} \biggr| \geq y \sqrt{K} \biggr) &\lesssim \frac{K^{1-p/2}}{y^p} + \sum_{j = 1}^{\infty} \exp\biggl(-\frac{C_{p,2} j^{1 + -2/p}y^2}{A^2}\biggr) \\&+ \exp\biggl(-\frac{C_{p,2} y^2}{\|U_{k,h}\|_2^2}\biggr).
\end{align*}
Setting $y = \log n$, the claim follows.

\end{proof}

Next, we establish some useful results concerning the covariances $\phi_{k,i,j} = \E\bigl[U_{0,i}U_{k,j}\bigr]$.
\begin{lemma}\label{lem_var_cov_property}
Grant Assumption \ref{ass_gaussian_approx}. Then
\begin{description}
\item[(i)] $\sup_{i,j}|\phi_{k,i,j}| \lesssim k^{-\cd + 1/2}$,
\item[(ii)] $\sup_{i,j} \sum_{k = 0}^{\infty} |\phi_{k,i,j}| < \infty$,
\item[(iii)] $\gamma_{i,j} = \phi_{0,i,j} + 2 \sum_{k = 1}^{\infty} \phi_{k,i,j} < \infty$,
\item[(iv)] $\sum_{k,l = 1}^n \E\bigl[U_{k,i}U_{l,j}\bigr] = n \gamma_{i,j} - \sum_{k \in \Z}^{\infty} n \wedge |k| \phi_{k,i,j}$.
\end{description}
\end{lemma}

\begin{proof}[Proof of Lemma \ref{lem_var_cov_property}]
Claims (iii) and (iv) are well-known in the literature, and follow from elementary computations from (ii).%, see for instance ~\cite{Jirak_2012_ch_cov}.
Since (i) implies (ii) due to $\cd > 3/2$, it suffices to establish (i). To this end, let $U_{k,h}^* = H_h\bigl(\F_k^*\bigr)$, where $\F_k^* = \sigma\bigl(\ldots,\epsilon_{-1}',\epsilon_0',\epsilon_1,\ldots, \epsilon_k\bigr)$. Since then $\E\bigl[U_{k,h}^{*}\bigl|\F_0\bigr] = \E\bigl[U_{k,h}\bigr] = 0$, Cauchy-Schwarz and Jensens inequality yield
\begin{align*}
\bigl|\E\bigl[U_{0,i}U_{k,i}\bigr] \bigr| = \bigl|\E\bigl[U_{0,i}\E[U_{k,j}\bigl|\F_{0}]\bigr] \bigr| \leq \bigl\|U_{0,i} \bigr\|_2 \bigl\|U_{k,j} - U_{k,j}^{*}\bigr\|_2.
\end{align*}
Theorem 1 in ~\cite{wu_2005} and \hyperref[C1]{\Cone} then imply that
\begin{align*}
\bigl|\E\bigl[U_{0,i}U_{k,j}\bigr] \bigr| \lesssim \biggl(\sum_{l = k}^{\infty} \theta_{l,2}^2\biggr)^{1/2} \lesssim k^{-\cd + 1/2}.
\end{align*}

\end{proof}

For $1 \leq i,j \leq d$ denote with
\begin{align*}
\gamma_{i,j}^{(n)} = \frac{1}{n}\E\bigl[S_{n,i} S_{n,j} \bigr], \quad \gamma_{i,j}^{(\diamond,n)} = \frac{1}{n}\E\bigl[S_{L,i}^{\diamond}(V) S_{L,j}^{\diamond}(V) \bigr].
\end{align*}

\begin{rem}\label{rem_cov_defined}
Note that Lemma \ref{lem_var_cov_property} (iv) yields that
\begin{align*}
\bigl|\gamma_{i,j} - \gamma_{i,j}^{(n)}\bigr| \lesssim \frac{1}{n} \sum_{k = 1}^n k^{3/2 - \cd} + \sum_{k > n}^{\infty} k^{-\cd + 1/2} \lesssim n^{3/2 - \cd}.
\end{align*}
\end{rem}

\begin{lemma}\label{lem_diff_cov}
Grant Assumption \ref{ass_gaussian_approx}. Then
\begin{align*}
\max_{1 \leq i,j \leq d}\bigl|\gamma_{i,j}^{(n)} - \gamma_{i,j}^{(\diamond,n)} \bigr| \lesssim n^{-1/2} L.
\end{align*}
\end{lemma}

\begin{rem}\label{rem_cov_defined_2}
Note that we obtain from Remark \ref{rem_cov_defined} that
\begin{align*}
\bigl|\gamma_{i,j} - \gamma_{i,j}^{(\diamond,n)}\bigr| \lesssim n^{-\frac{1}{2}} L + n^{\frac{3}{2} - \cd}.
\end{align*}
\end{rem}

\begin{proof}[Proof of Lemma \ref{lem_diff_cov}]
We have that
\begin{align*}
\biggl|\E\bigl[S_{L,i}(V) S_{L,j}(V) \bigr] -  \E\bigl[S_{L,i}^{\diamond}(V) S_{L,j}^{\diamond}(V) \bigr]\biggr| &\leq \sum_{l = 1}^L \bigl\| V_{l,j}^{\diamond} - V_{l,j}\bigr\|_2 \bigl\|S_{L,j}^{\diamond}(V)\bigr\|_2 \\&+ \sum_{l = 1}^L \bigl\| V_{l,i}^{\diamond} - V_{l,i}\bigr\|_2 \bigl\|S_{L,i}(V)\bigr\|_2.
\end{align*}
By the Marcinkiewicz–Zygmund inequality, Lemma \ref{lem_bound_partial_sum_gen} and \hyperref[C1]{\Cone} we have
\begin{align}\label{eq_lem_diff_cov_1}
\max_{1 \leq h \leq d}\bigl\|S_{L,h}^{\diamond}(V)\bigr\|_2 \lesssim \sqrt{n} \quad \text{and} \quad  \max_{1 \leq h \leq d}\bigl\|S_{L,h}(V)\bigr\|_2 \lesssim \sqrt{n}.
\end{align}
Using the triangle inequality and Theorem 1 in ~\cite{wu_2005}, it follows that
\begin{align}\nonumber\label{eq_lem_diff_cov_2}
\max_{1 \leq h \leq d} \sum_{l = 1}^L \bigl\| V_{l,h}^{\diamond} - V_{l,h}\bigr\|_2 &\lesssim \max_{1 \leq h \leq d} L \sum_{k = 1}^{\infty}\bigl\|U_{k,h} - U_{k,h}^*\bigr\|_2 \\&\lesssim L \sum_{k = 1}^{\infty} \sqrt{\sum_{j \geq k} \theta_{j,2}^2} \lesssim L \sum_{k = 1}^{\infty} j^{-\cd +1/2} \lesssim L.
\end{align}
Hence combining \eqref{eq_lem_diff_cov_1} and \eqref{eq_lem_diff_cov_2} we obtain
\begin{align*}
\max_{1 \leq i,j \leq d}\bigl|\gamma_{i,j}^{(n)} - \gamma_{i,j}^{(\diamond,n)} \bigr| \lesssim n^{-1/2} L.
\end{align*}

\end{proof}

Next, we state some Gaussian approximation results. To this end, we require the following condition. For $\varepsilon, u(\varepsilon) > 0$ we have

\begin{align}\label{eq_max_max_u}
P\biggl(\max_{1 \leq h \leq d}\max_{1 \leq l \leq L}|V_{l,h}^{\diamond}| \geq \sqrt{K u(\varepsilon)} \biggr) \leq \varepsilon.
\end{align}
Denote with
\begin{align*}
T_{L,d}^{\diamond} &= \frac{1}{\sqrt{n}}\max_{1 \leq h \leq d}\bigl|S_{L,h}^{\diamond}(V)\bigr|, \quad T_d^{Z,\diamond} = \max_{1 \leq h \leq d}\bigl|Z_h^{\diamond}\bigr|,
\end{align*}
where $\bigl\{Z_h^{\diamond}\bigr\}_{1 \leq h \leq d}$ is a zero mean Gaussian sequence with covariance structure $\Sigma_d^{(\diamond,n)} = \bigl(\gamma_{i,j}^{(\diamond,n)}\bigr)_{1 \leq i,j \leq d}$. We have the following Gaussian approximation result, which is an adaptation of Theorem 2.2 in ~\cite{chernozhukov2013}.

\begin{lemma}\label{lem_Kato}
Assume the validity of \eqref{eq_max_max_u} and that
\begin{description}
\item[(i)] $K^{-1/2}\min_{1 \leq h \leq d}\min_{1 \leq l \leq L} \bigl\|V_{l,h}^{\diamond}\bigr\|_2 > 0$,
\item[(ii)] $K^{-1/2}\max_{1 \leq h \leq d}\max_{1 \leq l \leq L} \bigl\|V_{l,h}^{\diamond}\bigr\|_4 < \infty$.
\end{description}
Then it holds that
\begin{align*}
&\sup_{x \in \R}\bigl|P\bigl(T_{L,d}^{\diamond} \leq x \bigr) - P\bigl(T_{d}^{Z} \leq x \bigr)\bigr| \\&\lesssim L^{-1/8} \bigl(\log(d L/\varepsilon)\bigr)^{7/8} + L^{-1/2} \bigl(\log(d L/\varepsilon)\bigr)^{3/2}u(\varepsilon) + \varepsilon.
\end{align*}
\end{lemma}

We also require the following two results, which are Lemmas 2.1 and 3.1 in ~\cite{chernozhukov2013}, slightly adapted for our purpose.

%Note that by considering the random variables $X_i$ and $-X_i$ (resp. $Y_i, -Y_i$), we obtain a corresponding result for absolute values, which we formulate in the Corollary below.
\begin{lemma}\label{lem_div_max_Gauss}
Let $\bigl\{X_h\bigr\}_{1 \leq h \leq d}$ and $\bigl\{Y_h\bigr\}_{1 \leq h \leq d}$ be zero mean Gaussian sequences, and denote with $\gamma_{i,j}^X, \gamma_{i,j}^Y$ the corresponding covariances for $1 \leq i,j \leq d$. If $0 < \inf_h \gamma_{h,h}^X \leq \sup_h \gamma_{h,h}^X < \infty$, then
\begin{align*}
\sup_{x \in \R}\bigl|P\bigl(\max_{1 \leq h \leq d}|X_h| \leq x \bigr) - P\bigl(\max_{1 \leq h \leq d}|Y_h| \leq x \bigr)\bigr| \lesssim \delta^{1/3} \bigl(1 \vee \log(d/\delta) \bigr)^{2/3},
\end{align*}
where $\delta = \max_{1 \leq i,j \leq d}\bigl|\gamma_{i,j}^X - \gamma_{i,j}^Y \bigr|$.
\end{lemma}

\begin{lemma}\label{lem_gauss_anti}
Let $\bigl\{X_h\bigr\}_{1 \leq h \leq d}$ be a zero mean Gaussian sequence, and denote with $\gamma_{i,j}^X$ the corresponding covariances for $1 \leq i,j \leq d$. If $0 < \inf_h \gamma_{h,h}^X \leq \sup_h \gamma_{h,h}^X < \infty$, then
\begin{align*}
\sup_{x \in \R}P\bigl(\max_{1 \leq h \leq d}|X_h - \delta | \leq x \bigr) \lesssim \delta \sqrt{1 \vee \log(d / \delta)}.
\end{align*}
\end{lemma}

We are now ready to give the proof of Theorem \ref{thm_gauss_approx}.
\begin{proof}[Proof of Theorem \ref{thm_gauss_approx}]
First note that by Lemma \ref{lem_control_V_approx} and Booles inequality we have
\begin{align*}
P\bigl(\max_{1 \leq h \leq d}|S_{L,h}(V) - S_{L,h}^{\diamond}(V)| \geq C_1 n^{1/2 - \delta}\bigr) \lesssim d n^{-\frac{p-2}{2} + \delta p}.
\end{align*}
Since $d \lesssim n^{\dd}$ we obtain from \hyperref[C2]{\Ctwo} that
\begin{align}\label{eq_thm_norm_approx_2}
P\bigl(\max_{1 \leq h \leq d}|S_{L,h}(V) - S_{L,h}^{\diamond}(V)| \geq C_1 n^{1/2 - \delta}\bigr) \lesssim n^{-C_2}, \quad C_2 > 0.
\end{align}
Employing this bound, we get that
\begin{align*}
P\bigl(T_{d} \leq x \bigr) \leq P\bigl(T_{L,d}^{\diamond} \leq x + C_1 n^{- \delta}\bigr)  + \OO\bigl(n^{-C_2} \bigr).
\end{align*}
In the same manner one obtains a lower bound, hence
\begin{align}\nonumber \label{eq_thm_norm_approx_3}
&P\bigl(T_{L,d}^{\diamond} \leq x -C_1 n^{- \delta}\bigr) - \OO\bigl(n^{-C_2} \bigr) \leq P\bigl(T_{d} \leq x \bigr) \\&\leq P\bigl(T_{L,d}^{\diamond} \leq x + C_1 n^{- \delta}\bigr) + \OO\bigl(n^{-C_2}\bigr).
\end{align}
Next, we apply Lemma \ref{lem_Kato} to $T_{L,d}^{\diamond}$. To this end, we need to verify its conditions. Note that by the independence of $V_{l,h}^{\diamond}$, we have that
\begin{align*}
\gamma_{h,h}^{(\diamond,n)} = \frac{1}{L K}\sum_{l = 1}^L \bigl\|V_{l,h}^{\diamond}\bigr\|_2^2 = \frac{1}{K}\bigl\|V_{1,h}^{\diamond}\bigr\|_2^2.
\end{align*}
Hence we deduce from Lemma \ref{lem_var_cov_property}, Lemma \ref{lem_diff_cov}, Remark \ref{rem_cov_defined_2} and \hyperref[C3]{\Cthree} that
\begin{align*}
K^{-1}\bigl\|V_{1,h}^{\diamond}\bigr\|_2^2 \geq \gamma_{h,h}^{(n)} - \oo\bigl(1\bigr) \geq \sigma_h^2 - \oo\bigl(1\bigr) > 0,
\end{align*}
uniformly in $h$, and thus (i) holds. Next we verify (ii). This, however, readily follows from Lemma \ref{lem_max_mom} and \hyperref[C1]{\Cone}. Finally, we need to establish \eqref{eq_max_max_u}. Set $u(\varepsilon) = (\log n)^2$. Using Booles inequality and Lemma \ref{lem_control_V_direct} gives
\begin{align*}
&P\biggl(\max_{1 \leq h \leq d}\max_{1 \leq l \leq L}|V_{l,h}^{\diamond}| \geq \sqrt{K u(\varepsilon)} \biggr) \\&\leq \sum_{h = 1}^d\sum_{l = 1}^L P\bigl(|V_{l,h}^{\diamond}| \geq \sqrt{K u(\varepsilon)} \bigr) \lesssim d L K^{-\frac{p-2}{2}} (\log n)^p.
\end{align*}
By \hyperref[C2]{\Ctwo} and choosing $\kd$ sufficiently close to $1$, we get that
\begin{align*}
P\biggl(\max_{1 \leq h \leq d}\max_{1 \leq l \leq L}|V_{l,h}^{\diamond}| \geq \sqrt{K u(\varepsilon)} \biggr) \lesssim n^{-C_3}, \quad C_3,
\end{align*}
and \eqref{eq_max_max_u} holds with $\varepsilon \thicksim n^{-C_3}$. Since $L \thicksim n^{\ld}$ with $\ld > 0$ due to $\kd < 1$, Lemma \ref{lem_Kato} yields that
\begin{align}\label{eq_thm_norm_approx_6}
\sup_{x \in \R}\bigl|P\bigl(T_{L,d}^{\diamond} \leq x \bigr) - P\bigl(T_{d}^{Z} \leq x \bigr)\bigr| \lesssim n^{-C_4}, \quad C_4 > 0.
\end{align}
Combining this with \eqref{eq_thm_norm_approx_3}, we deduce that
\begin{align}\nonumber\label{eq_thm_norm_approx_7}
&P\bigl(Z_{d}^{\diamond} \leq x -C_1 n^{- \delta}\bigr) - \OO\bigl(n^{-C_5} \bigr) \leq P\bigl(T_{d} \leq x \bigr) \\&\leq P\bigl(Z_{d}^{\diamond} \leq x + C_1 n^{- \delta}\bigr) + \OO\bigl(n^{-C_5}\bigr).
\end{align}
Next, since $\log d \lesssim \log n$, Lemma \ref{lem_gauss_anti} yields that
\begin{align}\label{eq_thm_norm_approx_8}
\sup_{x \in \R}\bigl|P\bigl(Z_{d}^{\diamond} \leq x -C_1 n^{-\delta}\bigr) - P\bigl(Z_{d}^{\diamond} \leq x \bigr)\bigr| \lesssim n^{-\delta} \sqrt{\log n}.
\end{align}
In addition, by Remark \ref{rem_cov_defined_2}
\begin{align*}
\max_{1 \leq i,j \leq d}\bigl|\gamma_{i,j}^{(\diamond,n)} - \gamma_{i,j}^{}\bigr| \lesssim n^{-\frac{1}{2}}L + n^{\frac{3}{2} - \cd} \lesssim n^{-C_6}, \quad C_6 > 0.
\end{align*}
Hence an application of Lemma \ref{lem_div_max_Gauss} yields
\begin{align}
\sup_{x \in \R}\bigl|P\bigl(Z_{d}^{\diamond} \leq x\bigr) - P\bigl(Z_{d} \leq x \bigr)\bigr| \lesssim n^{-C}, \quad C > 0.
\end{align}

\end{proof}

\subsection{Proofs of Section \ref{sec_applications}}\label{sec_proof_max_and_trace}

\begin{proof}[Proof of Theorem \ref{thm_eigen_max_gauss_approx}]
Denote with
\begin{align*}
T_{\JJ_n^+}^{\eta} = \frac{1}{\sqrt{n}}\max_{1 \leq j < \JJ_n^+}\frac{\bigl|\sum_{k = 1}^n (\eta_{k,j}^2 - 1)\bigr|}{\sigma_{0,j}}.
\end{align*}
We first show that we may apply Theorem \ref{thm_gauss_approx} to $T_{\JJ_n^+}^{\eta}$. To this end, we need to verify Assumption \ref{ass_gaussian_approx}. Observe that \hyperref[B2]{\Btwo} implies $\bigl\|\eta_{k,j}\bigr\|_{q} < \infty$ (cf. ~\cite{wu_2005}). Moreover, using $a^2 - b^2 = (a-b)(a+b)$, it follows from Cauchy-Schwarz
\begin{align*}
\bigl\|\eta_{k,j}^2 - (\eta_{k,j}^2)'\bigr\|_{q} \leq 2 \bigl\|\eta_{k,j} - \eta_{k,j}'\bigr\|_{2q} \bigl\|\eta_{k,j}\bigr\|_{2q} \lesssim \Omega_k(2q) \lesssim k^{-\bd}.
\end{align*}
Since $\bd > 3/2$ by \hyperref[B2]{\Btwo}, \hyperref[C1]{\Cone} follows. Next, note that \hyperref[B1]{\Bone} implies that $\JJ_n^+ \lesssim n^{p (\ad - \delta)}$. Since $q/2 - 1 > p 2^{\pd + 2} > p \ad$ (recall $0 < \ad < 1$), \hyperref[C2]{\Ctwo} holds. Finally, \hyperref[B3]{\Bthree} gives \hyperref[C3]{\Cthree}, hence Assumption \ref{ass_gaussian_approx} is verified. We proceed with the proof. For $j \in \N$, denote with $I_{j,j}^* = \lambda_j\sum_{k = 1}^n \bigl(\eta_{k,j}^2 -1 \bigr)/n$, and note that by the above and Lemma \ref{lem_bound_partial_sum_gen} we have
\begin{align}\label{eq_thm_eigen_max_gauss_approx_2}
\bigl\|I_{j,j}^*\bigr\|_p \lesssim \bigl(\lambda_j/n^{1/2}\bigr), \quad j \in \N.
\end{align}
Introduce the set
\begin{align*}
\mathcal{M} = \bigl\{\max_{1 \leq j < \JJ_n^+}\lambda_j^{-1}\bigl|\widehat{\lambda} - \lambda_j - I_{j,j}^* \bigr| \geq n^{-1/2 - \delta/2}\bigr\}.
\end{align*}
Then Markovs inequality together with Proposition \ref{prop_replace_I_with_eta} and \eqref{eq_thm_eigen_max_gauss_approx_2} yields
\begin{align}
P\bigl(\mathcal{M}^c\bigr) \lesssim n^{-p \delta/2} \lesssim n^{-C_1}, \quad C_1 > 0.
\end{align}
Due to Theorem \ref{thm_gauss_approx} and the above, we have the inequalities
\begin{align*}
P\bigl(T_{\JJ_n^+}^{} \leq x \bigr) &\leq P\bigl(T_{\JJ_n^+}^{\eta} \leq x + n^{-\delta/2} \bigr) + P\bigl(\mathcal{M}^c\bigr) \\&\leq P\bigl(T_{\JJ_n^+}^{Z} \leq x + n^{-\delta/2} \bigr) + \OO\bigl(n^{-C_2}\bigr), \quad C_2 > 0,
\end{align*}
where $T_{\JJ_n^+}^{Z}$ is as in \eqref{eq_defn_T_d_lambda}. An application of Lemma \ref{lem_gauss_anti} yields that this is further bounded by
\begin{align*}
P\bigl(T_{\JJ_n^+}^{} \leq x \bigr) \leq P\bigl(T_{\JJ_n^+}^{Z_{}} \leq x \bigr) + \OO\bigl(n^{-C_2} + n^{-\delta/2} \log n\bigr).
\end{align*}
In the same manner, we obtain a lower bound, hence
\begin{align}
\sup_{x\in \R}\bigl|P\bigl(T_{\JJ_n^+}^{} \leq x\bigr) - P\bigl(T_{\JJ_n^+}^{Z_{}} \leq x \bigr)\bigr|\lesssim n^{-C_3}, \quad C_3 > 0,
\end{align}
which completes the proof.

\end{proof}

\begin{proof}[Proof of Corollary \ref{cor_max_limit_distrib}]
Due to Theorem \ref{thm_eigen_max_gauss_approx}, it suffices to show that
\begin{align*}
P\bigl(T_{\JJ_n^+}^{Z_{\lambda}} \leq u_{\JJ_n^+}(z)\bigr) \to \exp\bigl(-e^{-z}\bigr).
\end{align*}
This, however, follows from Theorem 14 and Theorem 1 in ~\cite{han_wu_2014_max}.

\end{proof}

\section{Proofs of Section \ref{sec_applications_practical}}\label{sec_proof_applications_practical}

\begin{proof}[Proof of Proposition \ref{prop_ARH(1)}]
Due to \eqref{defn_struct_condition_ARH(1)}, Theorem 3.6  in ~\cite{bosq_2000} yields the Bernoulli-shift representation $X_k = \sum_{i = 0}^{\infty} {\bf \Phi}^{i}(\epsilon_{k-i})$. Next, using the orthogonality of $\{\epsilon_{k,j}\}_{j \in \N}$, we get
\begin{align}\label{eq_prop_ARH(1)_0}
\bigl\|\langle \epsilon_{k}, \e_l^{\theta} \rangle \bigr\|_2^2 &= \sum_{j = 1}^{\infty} {\lambda}_j^{\epsilon}  \langle \e_j^{\epsilon}, \e_l^{\theta} \rangle^2 .
\end{align}
On the other hand, since $\epsilon_k$ and $X_{k-1}$ are independent, we obtain
\begin{align}\label{eq_prop_ARH(1)_1}
\widetilde{\lambda}_l^{\theta} = \bigl\|\langle X_k, \e_l^{\theta} \rangle \bigr\|_2^2 = \bigl\|\langle {\bf \Phi}(X_{k-1}), \e_l^{\theta} \rangle \bigr\|_2^2 + \bigl\|\langle \epsilon_k, \e_l^{\theta} \rangle \bigr\|_2^2 \geq \bigl\|\langle \epsilon_k, \e_l^{\theta} \rangle \bigr\|_2^2.
\end{align}
For $k \geq 1$, using the triangle inequality, the linearity of ${\bf \Phi}$, the fact that ${\bf \Phi}(\e_j^{\phi}) = \lambda_j^{\phi} \e_j^{\phi}$ and \eqref{eq_epsilon_regularity_condition} yields that
\begin{align*}
\widetilde{\lambda}_l^{\theta}\bigl\| \eta_{k,l}^{\theta} - (\eta_{k,l}^{\theta})'\bigr\|_q^{2} &\lesssim \biggl(\sum_{i = 1}^{\infty} (\lambda_i^{\phi})^k |\langle \e_i^{\phi}, \e_l^{\theta} \rangle|\bigl\|\langle \epsilon_0 - \epsilon_0', \e_i^{\phi} \rangle \bigr\|_q \biggr)^2 \\&\lesssim \biggl(\sum_{i = 1}^{\infty} (\lambda_i^{\phi})^k |\langle \e_i^{\phi}, \e_l^{\theta} \rangle|\bigl\|\langle \epsilon_0 - \epsilon_0', \e_i^{\phi} \rangle \bigr\|_2^{q'/q}\biggr)^{2} \\&\lesssim \biggl(
\sum_{i = 1}^{\infty}(\lambda_i^{\phi})^k  \biggl(\sum_{j = 1}^{\infty} {\lambda}_j^{\epsilon} \|\epsilon_{0,j}\|_2^{2q'/q}\langle \e_j^{\epsilon}, \e_i^{\phi} \rangle^2 \langle \e_i^{\phi}, \e_l^{\theta} \rangle^2 \biggr)^{1/2}\biggr)^{2},
\end{align*}
where we also used $\bigl(\sum_{j = 1}^{\infty} \lambda_j^{\epsilon} \langle \e_j^{\epsilon}, \e_i^{\phi} \rangle^2 \bigr)^{(q' - q)/2q} < \infty$ in the last step (recall $q' \geq q$). Note that we have the inequality
\begin{align}
\langle \e_j^{\epsilon}, \e_i^{\phi} \rangle^2 \langle \e_i^{\phi}, \e_l^{\theta} \rangle^2 \leq \langle \e_j^{\epsilon}, \e_l^{\theta} \rangle^2,
\end{align}
which can be readily derived by contradiction (assume the converse and sum over $j$ on both sides). Hence by the triangle inequality and \eqref{defn_struct_condition_ARH(1)}, the above is further bounded by
\begin{align*}
&\lesssim \biggl(
\sum_{i = 1}^{\infty}(\lambda_i^{\phi})^k  \biggl(\sum_{j = 1}^{\infty} {\lambda}_j^{\epsilon} \|\epsilon_{0,j}\|_2^{2q'/q} \langle \e_j^{\epsilon}, \e_l^{\theta} \rangle^2 \biggr)^{1/2}\biggr)^{2} \\&\lesssim \biggl(
\sum_{i = 1}^{\infty}(\lambda_i^{\phi})\biggr)^{2k} \sum_{j = 1}^{\infty} {\lambda}_j^{\epsilon} \|\epsilon_{0,j}\|_2^{2q'/q} \langle \e_j^{\epsilon}, \e_l^{\theta} \rangle^2
\lesssim \rho^k \sum_{j = 1}^{\infty} {\lambda}_j^{\epsilon} \langle \e_j^{\epsilon}, \e_l^{\theta} \rangle^2,
\end{align*}
for $0 < \rho < 1$. Combining this with \eqref{eq_prop_ARH(1)_0}, \eqref{eq_prop_ARH(1)_1} we arrive at
\begin{align}\label{eq_prop_ARH(1)_2}
\bigl\|\eta_{k,l}^{\theta} - (\eta_{k,l}^{\theta})'\bigr\|_q^2 \lesssim \frac{\rho^k}{\widetilde{\lambda}_j^{\theta}} \sum_{j = 1}^{\infty} {\lambda}_j^{\epsilon} \langle \e_j^{\epsilon}, \e_l^{\theta}  \rangle^2 \lesssim \rho^k, \quad k \geq 1.
\end{align}
If $k = 0$, we get from \eqref{eq_epsilon_regularity_condition} that
\begin{align}\label{eq_prop_ARH(1)_3}
\widetilde{\lambda}_l^{\theta}\bigl\| \eta_{k,l}^{\theta} - (\eta_{k,l}^{\theta})'\bigr\|_q^2 = \bigl\|\langle \epsilon_k - \epsilon_k', \e_l^{\theta}  \rangle  \bigr\|_q^2 \lesssim \sum_{j = 1}^{\infty} {\lambda}_j^{\epsilon} \|\epsilon_{k,j}\|_q^2 \langle \e_j^{\epsilon}, \e_l^{\theta}  \rangle^2.
\end{align}
If $k < 0$ we have $\eta_{k,j}^{\theta} = (\eta_{k,j}^{\theta})'$, and hence the claim follows from \eqref{eq_prop_ARH(1)_2} and \eqref{eq_prop_ARH(1)_3}. Observe that by telescoping and Kolmogorov's zero one law, we also get that $\max_{j \in \N}\|\eta_{k,j}\|_q < \infty$.

\end{proof}

\begin{proof}[Proof of Corollary \ref{cor_AHR(1)}]
This follows from Lemma \ref{lem_bound_partial_sum_gen}.

\end{proof}

\begin{comment}
\section*{Acknowledgments}

I would like to thank the Associate Editor and in particular the anonymous Reviewer for the many constructive comments, suggestions and corrections. The generous help has been of major benefit.
\end{comment}

\begin{small}

\end{small}

\end{document}